\documentclass{article}
\usepackage{graphicx}
\usepackage{epstopdf}
\usepackage{amsfonts}
\usepackage{amsmath}
\usepackage{amssymb}
\usepackage{url}
\usepackage{fancyhdr}
\usepackage{indentfirst}
\usepackage{enumerate}

\usepackage{dsfont}
\usepackage[colorlinks=true,citecolor=blue]{hyperref}
\usepackage{amsthm}
\usepackage{color}
\usepackage{natbib}
\usepackage{comment}
\usepackage{capt-of}
\usepackage{multirow}
\usepackage{subfig}
\def\d{\mathrm{d}}

\newcommand{\D}{\mathcal {D}}

\newcommand{\X}{\mathcal {X}}

\newcommand{\E}{\mathbb{E}}
\newcommand{\B}{\mathcal{B}}
\newcommand{\F}{\mathcal{F}}
\newcommand{\e}{\mathcal{E}}

\newcommand{\R}{\mathbb{R}}
\newcommand{\N}{\mathbb{N}}
\newcommand{\p}{\mathbb{P}}

\newcommand{\id}{\mathds{1}}

\renewcommand{\(}{\left(}
\renewcommand{\)}{\right)}
\renewcommand{\[}{\left[}
\renewcommand{\]}{\right]}
\renewcommand{\ge}{\geqslant}
\renewcommand{\le}{\leqslant}
\renewcommand{\geq}{\geqslant}
\renewcommand{\leq}{\leqslant}
\renewcommand{\epsilon}{\varepsilon}

\newcommand{\essinf}{\mathrm{ess\mbox{-}inf}}

\theoremstyle{plain}
\newtheorem{theorem}{Theorem}

\newtheorem{lemma}{Lemma}
\newtheorem{proposition}{Proposition}
\theoremstyle{definition}

\newtheorem{example}{Example}

\newtheorem{assumption}{Assumption}
\theoremstyle{remark}
\newtheorem{remark}{Remark}

\newcommand{\cet}{\begin{center}}
	\newcommand{\ecet}{\end{center}}

\usepackage{setspace}

\begin{document}
	
\title{A Framework of State-dependent Utility Optimization with General Benchmarks}

\author{ Zongxia Liang\thanks{\scriptsize Department of Mathematical Sciences, Tsinghua University, China. Email: {\texttt{liangzongxia@mail.tsinghua.edu.cn}}}
	\and Yang Liu\thanks{\scriptsize Division of Mathematics, School of Science and Engineering, The Chinese University of Hong Kong, Shenzhen, Guangdong 518172, China. Email: {\texttt{yangliu16@cuhk.edu.cn} }}
	\and Litian Zhang\thanks{\scriptsize Corresponding author. Department of Mathematical Sciences, Tsinghua University, China. Email: { \texttt{zhanglit19@mails.tsinghua.edu.cn}}}
}

%
%

\date{}

\maketitle

\begin{abstract}	
	Benchmarks in the utility function have various interpretations, including performance guarantees and risk constraints in fund contracts and reference levels in cumulative prospect theory. In most literature, benchmarks are a deterministic constant or a fraction of the underlying wealth variable; thus, the utility is also a function of the wealth. In this paper, we propose a general framework of state-dependent utility optimization with stochastic benchmark variables, which includes stochastic reference levels as typical examples. We provide the optimal solution(s) and investigate the issues of well-definedness, feasibility, finiteness, and attainability. The major difficulties include: (i) various reasons for the non-existence of the Lagrange multiplier and corresponding results on the optimal solution; (ii) measurability issues of the concavification of a state-dependent utility and the selection of the optimal solutions. Finally, we show how to apply the framework to solve some constrained utility optimization problems with state-dependent performance and risk benchmarks as some nontrivial examples.	
\end{abstract}
\section{Introduction}\label{intro}
Fix a probability space $(\Omega, \F, \mathbb{P})$. The classical framework of expected utility maximization in portfolio selection  (cf. \citet{M1969}) is given by
\begin{equation}\label{problem0}
	\begin{split}
		\sup_{X}&~\mathbb{E} [u(X)]\\
		\text{subject}\text{ to}& ~\mathbb{E}[\xi X]\le x_0 \text{ and }u(X)>-\infty~\text{a.s.},
	\end{split}
\end{equation}
where $X: \Omega \to \R$ is a random variable representing the wealth and $u: \D \rightarrow \R$ is a differentiable and strictly concave utility function on the wealth level ($\D \subset \R$ is the domain of $u$). The random variable $\xi: \Omega \to (0, +\infty)$ is a so-called pricing kernel. The number $x_0 \in \R$ represents the budget constraint upper bound. It is clear that the objective $\E[u(X)]$ is invariant under the same distribution of $X$.
\vskip 5pt
This model has been challenged in many aspects: the non-concavity of utility functions and the application of stochastic benchmarks. Practically, the portfolio manager's objective function $u$ is no longer concave because of convex incentives in hedge funds; see \citet{C2000} and \citet{BS2014}. Since the seminal work of cumulative prospect theory (CPT) in \citet{TK1992}, the non-concave S-shaped utility (with a reference point $B \in \R$) has been widely adopted in the above model. Specifically, in an S-shaped utility, an individual displays different risk attitudes on the region smaller or larger than $B$; see \citet{BKP2004}, \citet{KZ2007} and \citet{HK2018}.

\vskip 5pt
The manager usually makes decisions based on the performance of some benchmark variable, e.g., a minimum riskless money market value or a minimum stochastic performance constraint; see \citet{BST2006}. In the literature, researchers also begin to investigate various benchmarks and apply different constraints on the wealth value and the benchmark. The model $\mathbb{E} \left[u\left(X-B\right)\right]$ is widely adopted and $B$ is interpreted as a benchmark of the wealth. In \citet{LL2020}, $B$ is a deterministic reference point in the S-shaped utility. In \citet{BKP2004}, $B$ is a fraction function of $X$, which can be regarded as a deterministic value after some transformations, and the wealth variable $X$ is required to be non-negative. In \citet{BHT2001}, the benchmark $B$ is a random guarantee variable and $X\ge B$ is required, and one can also eliminate the randomness of the benchmark by studying a new variable $X'=X-B\ge0$ and hence the objective function is still a univariate function of the wealth. Other than cases where $B$ can be handled as a deterministic value, it is of importance to study (real) stochastic reference points; see \citet{S2003}. Furthermore, \citet{CBD2006} and \citet{BPS2007} respectively adopt the models $\mathbb{E} \left[u\left(\frac{X}{B}\right)\right]$ and $\E \left[u\left(X f_1\left(X,B\right)\right)\right]$, with various specific benchmark variables and some specified function $f_1: \R \times \R \to \R$;
 see \citet{KR2007}, \citet{BVY2018} and \citet{BSV2019} for more concrete models. In conclusion, the benchmarks in the literature may be stochastic and are usually exogenous to the decision variable.


\vskip 5pt
In this paper, we propose a framework of state-dependent utility optimization with general benchmarks:
\begin{equation}\label{problem1}
	\begin{split}
		\sup_{X}&~\mathbb{E} \left[U\left(X,B\right)\right]\\
		\text{subject}\text{ to}& ~\mathbb{E}\left[\xi X\right]\le x_0 \text{ and } U(X,B)>-\infty~\text{a.s.,}
	\end{split}
\end{equation}
where $E$ is a measurable space and $B: (\Omega,\mathcal{F}) \to (E,\mathcal{E})$ is an $E$-valued random variable representing the benchmark. A multivariate utility function $U: \R \times E \to \R\cup\{-\infty\}$ depends on both the wealth/decision variable $X$ and the benchmark variable $B$.

The framework lies in a rather general setting and the objective is no longer distribution-invariant. {\color{black}The benchmark $B$ is required to be measurable on a space $E$ and measurable of the pricing kernel $\xi$, and it can be a deterministic function, a random variable, or a random vector on $(\Omega, \F, \mathbb{P})$.}
\footnote{{\color{black} The probability space often comes from a complete financial market where $X$ can be replicated. In this paper, we focus on the problem with a mathematical aspect and omit the details of the financial market.}}
Further, the utility function $U$ is only required to be non-decreasing and upper semicontinuous (hence may be discontinuous and non-concave) on $X$ and measurable on $B$. If $B$ is deterministic, Problem \eqref{problem1} reduces to Problem \eqref{problem0} with a univariate utility, and $B$ can be regarded as a reference point in the utility. 
Under some assumptions on $\xi$ and $B$, \cite{BMRV2015} investigate optimal payoffs under a specific state-dependent setting. However, in the literature, there is a lack of comprehensive and rigorous analysis of the general state-dependent utility optimization.
\vskip 5pt
Our contribution is to rigorously provide the optimal solution(s) and  investigate the following issues of this new framework \eqref{problem1}:
\begin{enumerate}[(i)]
	\item Optimality: a wealth variable $X$ is called optimal if $X$ solves Problem \eqref{problem1}.
	\item Feasibility: a wealth variable $X$ is called feasible if $\E \left[U\left(X,B\right)\right] > -\infty$. Problem \eqref{problem1} is called feasible if it admits a feasible solution.
	\item Finiteness: a wealth variable $X$ is called finite\footnote{For simplicity, the concept ``finite'' in (iii) only refers to $+\infty$. Therefore, if Problem \eqref{problem1} or a wealth variable $X$ is infeasible, or there is no wealth variable satisfying all constraints in Problem \eqref{problem1}, we still call it finite; see also Assumption \ref{asmp:well-defined}. These cases are trivial and can be easily recognized in our formulation.} if $\E \left[U\left(X,B\right)\right] < +\infty$. Problem \eqref{problem1} is called finite if the supremum in \eqref{problem1} does not equal $+\infty$.
	\item Attainability: Problem \eqref{problem1} is called attainable if it admits an optimal solution.
	\item Uniqueness: Problem \eqref{problem1} is called unique if, for any two optimal solutions $X$ and $\tilde{X}$, they are equal almost surely.
\end{enumerate}
{\color{black}
We summarize our main results as follows:
\begin{enumerate}
    \item In Theorems \ref{thm_concave}-\ref{general thm}, we establish a stochastic version of the concavification theorem. We introduce $\tilde{U}$ as the concavification of $U(x,b)$ in $x$ and define $\mathcal{X}_b(y)=\mathop{\arg\sup}_{x\in\mathbb{R}\cup\{\pm\infty\}}{\big[U(x,b)-yx\big]}$. We first prove that the concavified problem $\sup\limits_{X: \mathbb{E}[\xi X]\le x_0\atop \tilde{U}(X,B)>-\infty}\mathbb{E}[\tilde{U}(X,B)]$ is also well-defined and has the same value as the original problem, and then we show that $X$ is an optimal solution if and only if $X$ satisfies the budget constraint and locates in a random set $\mathcal{X}_B(\lambda\xi)$ for some $\lambda\ge0$;
    \item In Theorem \ref{main}, we give a measurable selection from the random set $\mathcal{X}_B(\lambda\xi)$ by introducing 
    $\overline{X}_b(y) \triangleq \sup \mathcal{X}_b(y)$ and $\underline{X}_b(y) \triangleq \inf \mathcal{X}_b(y)$. Without assuming that the Lagrange multiplier always exists, we find out the expression of the optimal solution and propose the case where the problem becomes unattainable. Moreover, when the problem is unattainable, we also give the optimal value of the problem and find a sequence of convergent feasible variables $\{X_k\}_{k\ge0}$ whose objective values converge to the optimal value.
\end{enumerate}
}
For the classical framework \eqref{problem0}, the standard approach solving  \eqref{problem0}  is the duality method (cf. \citet{KS1998}). With some assumptions and standard conditions on $u$, for any $x_0$ at the domain of $u$, one can always obtain a unique, finite and non-trivial optimal solution $X^* = (u')^{-1}(\lambda^* \xi)$ for Problem \eqref{problem0}, where $\lambda^* \in (0, +\infty)$ is the Lagrange multiplier solved from $\E[\xi (u')^{-1}(\lambda^* \xi)] = x_0$. For a non-concave utility function $u$, the problem can be solved by the Legendre-Fenchel transformation (cf. \citet{R1970}) and the concavification technique (cf. \citet{C2000}). This technique aims to prove that the optimal solution under a non-concave utility is also the optimal one under its concavification (the minimal dominating concave function of the non-concave utility) and solve the latter problem; it generally requires the assumption of a non-atomic $\xi$ (cf. \citet{BS2014}). The existence of $\lambda^*$ is a key issue. Traditionally, it is always assumed a priori that the function $g(\lambda) \triangleq \E[\xi (u')^{-1}(\lambda \xi)]$ is finite (i.e., $g(\lambda)<+\infty$ for all $\lambda>0$) and the Lagrange multiplier always exists (cf. \citet{KLS1987,KS1999,W2018}). To this point, \citet{JXZ2008} investigate this issue in the classic framework \eqref{problem0} and provide a counterexample that $g(\lambda)$ equals to $+\infty$ for small $\lambda$, and hence the Lagrange multiplier does not exist. In the non-concave univariate setting, \citet{R2013} proves the existence of $\lambda^*$ when the optimal solution exists. Usually, $g$ is continuous and decreasing on its domain. Therefore, the existence of $\lambda^*$ for a proper $x_0$ is guaranteed by the intermediate value theorem. However, the result on optimal solutions when the Lagrange multiplier does not exist is absent.
\vskip 5pt
In this paper, we study an extended version of Problem \eqref{problem0} with weaker assumptions, larger space of utilities, and more detailed conclusions. {\color{black} We only require some of the following classical assumptions:} 
\begin{enumerate}[(I)]
	\item The utility function satisfies the asymptotic elasticity condition (cf. \citet{KS1999}), Inada conditions and other conditions;	
	\item The Lagrange multiplier always exists; see Case 1 in Section \ref{section_finite};
	\item The probability space $(\Omega, \F, \p)$ or the pricing kernel $\xi$ is non-atomic\footnote{A measure $\mu$ is called non-atomic, if for any measurable set $A$ with a positive measure, there exists a measurable subset $B\subset A$ satisfying $\mu[A]>\mu[B]>0$. A random variable $X$ is called non-atomic if its distribution measure is non-atomic. The probability space $(\Omega, \F, \p)$ is called non-atomic if its probability measure $\mathbb{P}$ is non-atomic, which is equivalent to the existence of a continuous distribution; see Proposition A.31 in \citet{FS16}.};
	\item Problem \eqref{problem1} is finite; see Assumption \ref{ass_finite}.
\end{enumerate}
In contrast to (I)-(IV), our discussion includes the following cases that: (I) the utility has a linear tail; (II) the Lagrange multiplier does not exist for some initial value; (III) the pricing kernel is atomic; (IV) Problem \eqref{problem1} is infinite. 
\vskip 5pt
Through investigation on the new framework \eqref{problem1}, we contribute to demonstrate and provide analysis to the following technical cases:
\begin{enumerate}[(i)]
	\item The optimal solution may be non-unique, i.e., there may be a ``random set", denoted by $\X_B(\lambda^* \xi)$ in \eqref{eq:set_rv}. 
 It is because the ``conjugate point" in the Legendre-Fenchel transformation is no longer always unique for our $U$; see Section \ref{section_general}.
	\item Define $\underline{X}_B(\lambda\xi\big) = \inf \X_B(\lambda \xi)$ in \eqref{eq:supinf}. The analogue $g(\lambda)=\E\big[\xi \underline{X}_B\big(\lambda\xi\big)\big]$ defined in \eqref{eq:g} may also equal to $+\infty$ as in \citet{JXZ2008}; see Figure \ref{fig:ex g} (iii)(v)(vi). Moreover, it may even be discontinuous on its domain; see Figure \ref{fig:ex g} (ii)(iv)(vi). Hence, the intermediate value theorem cannot directly guarantee the existence of the Lagrange multiplier; see Section \ref{section_finite}.
	\item In addition to (i), to find $\lambda^*$, we need to select a measurable function (random variable) from the set $\X_B(\lambda^* \xi)$. The measurability issue also arises when applying concavification to state-dependent (or, multivariate) utility functions; see Sections \ref{section_general} and \ref{section_finite}.
	\item In the classical framework \eqref{problem0}, it is also known that $X^*$ is a decreasing function of $\xi$ ; see \citet{C2000} for a detailed economic discussion. In the general framework \eqref{problem1}, the optimal solution $X^*$ may not be a decreasing function of $\xi$ {\color{black} or comonotonic to $\xi$}; see \cite{BMRV2015} for a concrete model. The fact is because the objective function is no longer distribution-invariant on $X$. The technique is in contrast to the results of the quantile formulation approach.
\end{enumerate}
The non-existence of Lagrange multiplier (ii) and the measurability issue (iii) are the biggest difficulties in the discussion.
Technically, together with the situation where $g$ may not be finite, the optimal solution(s) and the above issues are discussed in Theorems \ref{thm_concave}-\ref{main}, summarized in Table \ref{table1} and visualized in Figure \ref{fig:ex g}. In Theorems \ref{thm_concave}-\ref{general thm}, we overcome the measurability difficulties, apply the variational method to obtain the optimal solution(s), and hence give a stochastic version of the concavification theorem. {\color{black} In Theorem \ref{main}, we give an expression of the optimal solution and cover the case where $g$ is discontinuous or infinite. The result also includes the case where Problem \eqref{problem1} is unattainable and finds the optimal value.} The insights of some proofs are illustrated by Figure \ref{fig:ex g}. 
\vskip 5pt
Moreover, the benchmark is also motivated to serve as a risk management constraint if one converts the risk constraint into an unconstrained utility optimization problem by a Lagrange multiplier. The first example is that the so-called liquidation boundary is set as the benchmark process, which the wealth is required to be always above; see \citet{HJ2007}. The second example is the constraints on default probability and Value-at-Risk to mitigate excessive risk taking; see \citet{CH2016} and \citet{NS2020}. These constraint problems can be also transformed to our Problem \eqref{problem1} by emerging the constraints into the utility function as a benchmark variable via Lagrangian duality arguments; see \citet{DZ2020}. {\color{black} We will give some nontrivial examples as applications of our results in Section \ref{section_ex2}.}
\vskip 5pt
The rest of this paper is organized as follows. Section \ref{section_model} establishes the model settings of Problem \eqref{problem1}. The optimal solution(s) are obtained in Section \ref{section_general}. The issues of feasibility, finiteness, attainability, and uniqueness are respectively investigated in Sections \ref{section_finite}-\ref{section_condition}. Section \ref{section_univ} provides a complete result to connect with the univariate framework \eqref{problem0}. Section \ref{section_ex2} presents a concrete application for our framework. Section \ref{section_conclusion} concludes the paper. The proofs are in the Appendix.

\section{Preliminaries}\label{section_model}

In this section, we specify the required settings and assumptions of a multivariate utility function $U:\mathbb{R} \times E\to\mathbb{R}\cup\{-\infty\}$.

\begin{assumption}[{\color{black} Utility and benchmark}]
	\label{ass_standing}
	$U(x, \cdot)$ is measurable on $(E,\mathcal{E})$ for any $x \in \R$. For every $b\in E$, we define $\underline{x}_b\triangleq\inf\{x\in\mathbb{R}: U(x,b)>-\infty\}$
	as the lower bound of the wealth, which is allowed to vary with the benchmark value $b$.
	Suppose that for each $b\in E$:
	\begin{enumerate}[(a)]
		\item $-\infty< \underline{x}_b < +\infty$, and $\mathbb{E}\left[\xi\underline{x}_B\right]>-\infty$;
		\item $U(\cdot, b)$ is non-decreasing and upper semicontinuous on $[\underline{x}_b,+\infty)$;
		\item $\limsup\limits_{x\to +\infty}\frac{U(x,b)}{x}=0$.
	\end{enumerate}
\end{assumption}
We give some explanations on these conditions. All of them coincide with the classical theory. Condition (a) means that
for any variable $X$ under consideration in Problem \eqref{problem1}, one should have $X \geq \underline{x}_B$ (or sometimes $X>\underline{x}_B$). Condition (a) is consistent with the classical settings, as we embed the restriction on the lower bound of the return $X$ into the utility function $U$. For the instance of a (univariate) CRRA utility $u$, the wealth $X$ is required to be nonnegative, and the domain of the utility $u$ is $[0,+\infty)$. Here, the domain is extended to $\mathbb{R}$ and the value of the left tail should be $-\infty$; in this case, we have $\underline{x}=0$. For an S-shaped utility $u$, the wealth $X$ is bounded from below by a deterministic liquidation level $L \in \R$, and $u$ should be truncated at the finite left endpoint $L$ and the value on the left tail is $-\infty$; in this case, $\underline{x} = L$. Moreover, the constraint $\mathbb{E}[\xi\underline{x}_B]>-\infty$ holds automatically in the univariate case where $\underline{x}_B$ usually equals $0$ or other deterministic numbers. The assumption is easy to satisfy, as in the classic case we often assume $\underline{x}_B=0$. Condition (b) means that the utility functions do not necessarily have concavity and differentiability, and include S-shaped functions and step functions. Upper semicontinuity is indeed equivalent to right-continuity when the nondecreasing property holds in Assumption \ref{ass_standing}. It leads to two possible cases at  $\underline{x}_b$:
\begin{enumerate}[(I)]
	\item $U(\underline{x}_b,b)>-\infty$ with $U(\cdot,b)$ being right-continuous at $\underline{x}_b$, and $U(x,b)=-\infty$ for $x<\underline{x}_b$, i.e., a truncation occurs at $\underline{x}_b$;
	\item $U(\underline{x}_b,b)=-\infty$, and $\lim\limits_{x\to\underline{x}_b+}U(x,b)=-\infty$.
\end{enumerate}
Hence, Condition (b) contains both power utilities (type I for positive exponents and type II for negative exponents) and logarithm utilities (type II). Condition (c) is required to ensure the finiteness of the optimization problem if the utility function is not differentiable. It is slightly weaker than the classical condition ($u'(+\infty)=0$) 
and can be interpreted as the diminishing marginal utility. 
Unless specified, we suppose that Assumption \ref{ass_standing} holds throughout.

To study Problem \eqref{problem1}, the following two assumptions are also helpful:
\begin{assumption}[Well-definedness]\label{asmp:well-defined}
	Problem \eqref{problem1} is well-defined\footnote{Similar to the concept ``finite", if Problem \eqref{problem1} is infeasible, or there is no wealth variable satisfying all constraints in Problem \eqref{problem1}, we still call it well-defined (but meaningless).}, i.e., for every random variable $X$ satisfying $\mathbb{E}\left[\xi X\right]\le x_0$ and $U(X,B)>-\infty~\text{a.s.}$, the expectation $\mathbb{E}\left[U(X,B)\right]$ is well-defined, i.e.,\footnote{Throughout, we denote $x^+ = \max\{x, 0\}$ and $x^- = -\min\{x, 0\}$ for any $x \in \R$. }
	\begin{equation}\label{eq:well-defined}
		\mathbb{E}\left[U(X,B)^+\right]<+\infty \text{ or }\mathbb{E}\left[U(X,B)^-\right]<+\infty.
	\end{equation}
\end{assumption}

\begin{assumption}[Finiteness of the problem]\label{ass_finite}
	Problem \eqref{problem1} is finite, i.e.,
	$$
	\sup_{X: \mathbb{E}\left[\xi X\right]\le x_0\atop U(X,B)>-\infty}~\mathbb{E} \left[U\left(X,B\right)\right] < +\infty.
	$$
\end{assumption}

Based on Assumption \ref{asmp:well-defined} and the fact that $U(\cdot,b)$ is nondecreasing, it is equivalent to study Problem \eqref{problem1} with the budget constraint $\mathbb{E}[\xi X] = x_0$. If not, we can replace $X$ by $X'=X+c \id_{\{\xi<n\}}$ for some $c, n > 0$, which will increase both $\mathbb{E}[\xi X]$ and $\mathbb{E}[U(X,B)]$. For the second coordinate $b \in E$, as we only require the measurability, we will refer to the first coordinate $x$ when discussing the other properties of $U$, such as concavity, differentiability, etc. 

In Proposition \ref{suff cond}, we give a sufficient condition of Assumptions \ref{asmp:well-defined} and \ref{ass_finite}, in order to conveniently verify the two assumptions. 

\begin{proposition}\label{suff cond}
	Suppose that $\xi\in L^1(\Omega)$ {\color{black} with Assumption \ref{ass_standing} and for some $\delta\in(0,1)$ both of the followings hold:}
	\begin{enumerate}
		\item[(1)] There exist $u_1(b)\ge0$, $u_2(b)\ge0$ and $K(b)\ge0$ {\color{black} for every $b\in E$} such that for any $x>K(b)$, $U(x,b)\le u_1(b)+u_2(b)x^\delta$ with $\xi^{-\delta} u_1(B)\in L^{\frac{1}{1-\delta}}(\Omega)$ and $\xi^{-\delta} u_2(B)$$\in L^{\frac{1}{1-\delta}}(\Omega)$ , and $\xi K(B)\in L^1(\Omega)$;
		\item[(2)] {\color{black} There exists $\theta(b)\ge0$, $b\in E$ such that   $\xi\theta(B)\in L^1(\Omega)$ and $\xi^{-\delta}\gamma(B)\in L^{\frac{1}{1-\delta}}(\Omega)$,
		where $\gamma(b)=U(\underline{x}_b+\theta(b),b)$.}
	\end{enumerate}
	Then  Problem \eqref{problem1} is well-defined and finite for any $x_0\in\R$.
\end{proposition}

We point out that Assumption \ref{asmp:well-defined} holds automatically if $U$ itself has a lower bound on its domain, and we also have another sufficient condition proposed in Section \ref{section_condition}. For Assumption \ref{ass_finite}, it is not a crucial assumption in this paper. In fact, as one of the main results, Theorem \ref{main} gives the existence and expressions of the optimal solution without using the finiteness assumption, and one can verify the finiteness after the expression of the optimal solution is solved. If fortunately, the problem is finite, then Theorem \ref{general thm} shows that the expression of the optimal solution is unique. Assumption \ref{ass_finite} is only needed in the discussion on the unique expression of the solution. Moreover, we also have a tractable result of the finiteness issue in Section \ref{section_condition}.

Noting that the benchmark variable $b$ is involved in every parameter above, we have to give integrability conditions on each of them. The conditions we propose in Proposition \ref{suff cond} seem complicated, but it is indeed easy to satisfy; see the following examples.
\begin{example}
	If $B$ is deterministic, then $u_1(B),~u_2(B),~K(B),~\theta(B)$ and $\gamma(B)$ are all deterministic, and hence we only need $\mathbb{E}[\xi]<+\infty$ and $\mathbb{E}[\xi ^{-\frac{\delta}{1-\delta}}]<+\infty$, which hold if $\xi$ is lognormal.
\end{example}
\begin{example}\label{eg2}
	Fix $b \in \R$. For an S-shaped utility
	$
	U(x,b)=(x-b)^p\id_{(b,+\infty)}(x)-k(b-x)^p\id_{[0,b]}(x)
	$,
	if we take $u_1(b)=1$,~$u_2(b)=0$, $K(b)=0$, $\delta=p$, $\theta(b)=b$, $\gamma(b)=0$, then Proposition \ref{suff cond} requires  $\mathbb{E}[\xi]<+\infty$, $\mathbb{E}[\xi^{-\frac{p}{1-p}}]<+\infty$ and $\mathbb{E}[\xi B]<+\infty$. If we use $\theta(b)=0$, $\gamma(b)=-kb^p$, then $\mathbb{E}[\xi B]<+\infty$ can be replaced by $\mathbb{E}[\xi^{-\frac{p}{1-p}}B^{\frac{p}{1-p}}]<+\infty$.
\end{example}


%

%
%
%
%
%
%
%



Finally, we define the bliss point (cf. \citet{BH1998})
\begin{equation}\label{eq:bliss}
	\overline{x}_b\triangleq\inf\left\{x\in\mathbb{R}:U(x,b)=U(+\infty, b)\right\}.
\end{equation}
In most literature, $\overline{x}_b = +\infty$ for each $b \in E$. In our model, $\overline{x}_b \in (-\infty,+\infty]$ and is allowed to be finite. In this light, a wealth variable $X$ is called a bliss solution, if $X\ge\overline{x}_B$, i.e., $U(X(\omega),B(\omega))$ attains its maximum at every state point $\omega$ without any risk; see Remark \ref{rmk_bliss} later for more details.

\section{Optimality}
\label{section_general}

In this section, we give our main results on the optimal solution(s) in Theorems \ref{thm_concave}-\ref{general thm}. 
To begin with, we develop the state-dependent Legendre-Fenchel transformation. For $b\in E$ and $0\le y < +\infty$, we define the conjugate function $V_b$ and the conjugate set $\mathcal{X}_b$ of $U(\cdot,b)$ as
\begin{eqnarray}	 \label{conjugate}
	&&V_b(y)=\sup_{x\in\mathbb{R}}{\big[U(x,b)-yx\big]}\in\mathbb{R}\cup\{+\infty\},\\
	&&	\mathcal{X}_b(y)=\mathop{\arg\sup}_{x\in\mathbb{R}\cup\{\pm\infty\}}{\big[U(x,b)-yx\big]}. \label{conjugate point}
\end{eqnarray}
As $U$ is upper semicontinuous, if $x_n\to x\in\mathbb{R}$ with $U(x_n,b)-yx_n\to V_b(y)$, then
\begin{equation*}
	V_b(y)=\lim_{n\to+\infty}\(U(x_n,b)-y x_n\)\le U(x,b)-yx\le V_b(y).
\end{equation*}
Hence $x\in\mathcal{X}_b(y)$. Thus, the notation \eqref{conjugate point} can be expressed as following:
\begin{itemize}
	\item If $x\in\mathbb{R}$, then $x\in\mathcal{X}_b(y)$ if and only if $U(x,b)-yx=V_b(y)$.
	\item If $x\in\{\pm\infty\}$, then $x\in\mathcal{X}_b(y)$ if and only if {\color{black} there exists a sequence of real numbers$\{x_n\}\uparrow$(or $\downarrow$) $x$ such that $U(x_n,b)-yx_n\to V_b(y)$.} 
\end{itemize}
And we define $\mathcal{X}_b(+\infty)=\{\underline{x}_b\}$; see also Lemma \ref{lem}(iv) for this definition. Then $\mathcal{X}_b(y)$ is non-empty for $y\in[0,+\infty]$.
For $\lambda \in [0, +\infty]$, we define a ``random set"
\begin{eqnarray}\label{eq:set_rv}	
	\mathcal{X}_B(\lambda \xi):\Omega\to 2^{\R\cup\{\pm\infty\}}, ~\omega\mapsto \mathcal{X}_{B(\omega)}(\lambda \xi(\omega)).
\end{eqnarray}

Compared to the classical results, one may conjecture that $\X_B(\lambda \xi)$ is the optimal solution to Problem \eqref{problem1}. However, it is worth pointing out that $\mathcal{X}_B(\lambda\xi)$ here is defined in terms of a set, as it may become no longer unique when
$U(\cdot,b)$ is non-concave for some $b \in E$. Hence, we have to study the random set $\X_B(\lambda \xi)$ for the optimal solutions. Indeed, Theorem \ref{general thm} shows that the optimal solution is still located in $\X_B(\lambda \xi)$, but different from the classical case, we need extra work to find out a measurable selection; see Section \ref{section_finite} later. In the following content, we will use the notation $\mathcal{X}_b^U$ and $\underline{x}_b^U$ instead of $\mathcal{X}_b$ and $\underline{x}_b$ in case of possible confusion.




Now we present our main results, some of which require Assumption \ref{ass_finite}; details on the finiteness issue will be studied in Sections \ref{section_finite}-\ref{section_condition}. We first investigate the concave utility function in Theorem \ref{thm_concave}. In the concave case, we do not assume that the utility satisfies conditions such as the Inada condition or $\underline{x}_b>-\infty$, which means that $U$ can take finite values on $\mathbb{R}$. {\color{black}We also do not require the probability measure $\mathbb{P}$ to be non-atomic.}

\begin{theorem}\label{thm_concave}
	Suppose that Assumptions \ref{asmp:well-defined}-\ref{ass_finite} hold and $U(\cdot, b)$ is nondecreasing, upper semicontinuous and concave for any $b \in E$. We have that $X$ is an optimal solution of Problem \eqref{problem1} if and only if
	\begin{equation}
		X\in \mathcal{X}_B(\lambda\xi) \text{ a.s. for some } \lambda\in[0,+\infty]
	\end{equation}
	satisfying the budget constraint $\mathbb{E}[\xi X]=x_0$ and $U(X,B)>-\infty$.
\end{theorem}

For the non-concave case, to give a concavification theorem under the multivariate setting, there are many technical issues such as measurability to be addressed. Hence we need to investigate in detail the properties of the concavification of functions satisfying certain conditions. In the following Theorem \ref{general thm}, we give results on general utility functions {\color{black}under the assumption that the probability measure is non-atomic.}  
\begin{theorem}\label{general thm}
	Suppose that $(\Omega,\mathcal{F},\mathbb{P})$ is non-atomic and Assumptions \ref{ass_standing}-\ref{asmp:well-defined} hold.
	\begin{enumerate}[(i)]
		\item The concavification problem $\sup\limits_{X: \mathbb{E}[\xi X]\le x_0\atop \tilde{U}(X,B)>-\infty}\mathbb{E}[\tilde{U}(X,B)]$ is well-defined, and we have
		\begin{equation}\label{eq:equivalence}
			\sup_{X:\mathbb{E}[\xi X]\le x_0\atop U(X,B)>-\infty} \mathbb{E}[U(X,B)]=\sup_{X: \mathbb{E}[\xi X]\le x_0\atop \tilde{U}(X,B)>-\infty}\mathbb{E}[\tilde{U}(X,B)],
		\end{equation}
		where $\tilde{U}(\cdot,b)$ is the concavification of $U(\cdot,b)$ {\color{black} on $\mathbb{R}$} for any $b \in E$;
		\item Suppose further that Assumption \ref{ass_finite} holds. Then $X$ is an optimal solution if and only if $X\in\mathcal{X}^U_B(\lambda\xi)$ a.s. for some $\lambda\in[0,+\infty]$ with $\mathbb{E}[\xi X]=x_0$ and $U(X,B)>-\infty$.
	\end{enumerate}
\end{theorem}

Although the basic strategy solving non-concave optimization is the concavification technique, many technical issues (measurability, well-definedness, etc) are required to be addressed in the non-concave and multivariate setting. In Theorem \ref{general thm}(i), we rigorously prove the availability of the concavification technique for a multivariate utility, while in the second part, we give a necessary and sufficient condition on optimal solutions, and we do not assume a priori that a Lagrange multiplier exists. Conversely, our results indicate that the existence of a Lagrange multiplier is necessary for Problem \eqref{problem1} to be solvable.

\begin{remark}\label{rmk:finiteness}
	From the proof of Theorems \ref{thm_concave}-\ref{general thm}, we will  see that the ``if'' part in Theorems \ref{thm_concave}-\ref{general thm} only needs Assumptions  \ref{ass_standing}-\ref{asmp:well-defined}, but not Assumption \ref{ass_finite}. That is, if we have found some $X\in \mathcal{X}_B(\lambda\xi)$ with some $\lambda\in[0,+\infty]$ satisfying the budget constraint $\mathbb{E}[\xi X]=x_0$ and $U(X,B)>-\infty$, then $X$ is an optimal solution to Problem \eqref{problem1} no matter whether it is finite. This will lead to a tractable Theorem \ref{prop:suff2} verifying the finiteness.
\end{remark}



To close this section, we summarize that for the benchmark $B$ and the multivariate $U$, a tractable approach is proposed to determine (the existence and expression of) the finite optimal solution $X^*$. However, in the abstract setting, we only know that $X^*\in\mathcal{X}_B(\lambda\xi)$. Different from the classical case where the optimal solution is consequently determined, we have to prove the existence of a measurable selection and find out its expression; see  Section \ref{section_finite}.

\section{Finiteness, attainability and uniqueness}\label{section_finite}
We are going to fully investigate the feasibility, finiteness, attainability and uniqueness; see the definitions in Section \ref{intro} of optimal solution(s). The results are presented in Theorem \ref{main} and summarized in Table \ref{table1} and Figure \ref{fig:ex g}.

As we have discussed in Section \ref{section_general}, to find a finite optimal solution, we desire to determine $\lambda \in [0, +\infty]$ and $X\in\mathcal{X}_B(\lambda\xi)$ satisfying $\mathbb{E}[\xi X]=x_0$. The key difficulty here is that we need to determine both a Lagrange multiplier $\lambda^*$ and a measurable selection of  $\X_B(\lambda^* \xi)$. If the desired $\lambda^*$ does not exist, then Problem \eqref{problem1} is either infinite or unattainable. In this section, we will investigate the issues of finiteness, attainability, and uniqueness. Moreover, even if the optimal solution exists, we find that it may not be unique under some novel conditions. For finiteness, we also propose sufficient conditions to attain a finite optimal solution for Problem \eqref{problem1} in Section \ref{section_condition}.

First, for any $y\in\[0,+\infty\]$ and $b \in E$, as the set $\X_b(y)$ is non-empty, we define
\begin{equation}\label{eq:supinf}
	\overline{X}_b(y) \triangleq \sup \mathcal{X}_b(y), ~~ \underline{X}_b(y) \triangleq \inf \mathcal{X}_b(y).
\end{equation}
These two quantities are maximum and minimum of the set $\X_b(y)$. They are important in measurable selection. Some properties are listed in Lemma \ref{lem}.
\begin{lemma}\label{lem}
	Suppose that $U(\cdot, b)$ is nondecreasing and upper semicontinuous for any $b \in E$, and $U(x, \cdot)$ is measurable for any $x \in \mathbb{R}$. 
	Then the functions $V_b$ (see \eqref{conjugate}), $\overline{X}_b$ and $\underline{X}_b$ satisfy:
	\begin{enumerate}[(i)]
		\item for any $y\in\[0,+\infty\]$, $ \underline{X}_b(y)\ge\underline{x}_b$, and for any $y\in\(0,+\infty\]$, $ \overline{X}_b(y)\le\overline{x}_b$ (note that $\underline{x}_b$ and $\overline{x}_b$ are defined in Assumption \ref{ass_standing});
		\item  {\color{black} for any $y\in\[0,+\infty\]$, we have} $\overline{X}_b(y)\in\mathcal{X}_b(y)$, $\underline{X}_b(y)\in\mathcal{X}_b(y)$, and $\overline{X}_b(y_2)\le \underline{X}_b(y_1)$ holds for any $0\le y_1<y_2\le+\infty$;
		\item both $\overline{X}_b(y)$ and $\underline{X}_b(y)$ are nonincreasing in $y$ and Borel-measurable in $(y,b)$, {\color{black} where $y\in\[0,+\infty\]$} ;
		\item $\lim\limits_{y\to 0+}\overline{X}_b(y)=\lim\limits_{y\to 0+}\underline{X}_b(y)=\overline{x}_b$,~$\lim\limits_{y\to +\infty}\overline{X}_b(y)=\lim\limits_{y\to +\infty}\underline{X}_b(y)=\underline{x}_b$;
		\item for any $y_0\in(0,+\infty)$, $\lim\limits_{y\to y_0-}\underline{X}_b(y)=\overline{X}_b(y_0)$,$\lim\limits_{y\to y_0+}\overline{X}_b(y)=\underline{X}_b(y_0)$.
	\end{enumerate}
\end{lemma}

\begin{figure}[htbp]
	\renewcommand{\thesubfigure}{\roman{subfigure}}
	\centering
	\subfloat[Arabic numerals][Case 1, continuous]{
		\includegraphics[width=7.5cm]{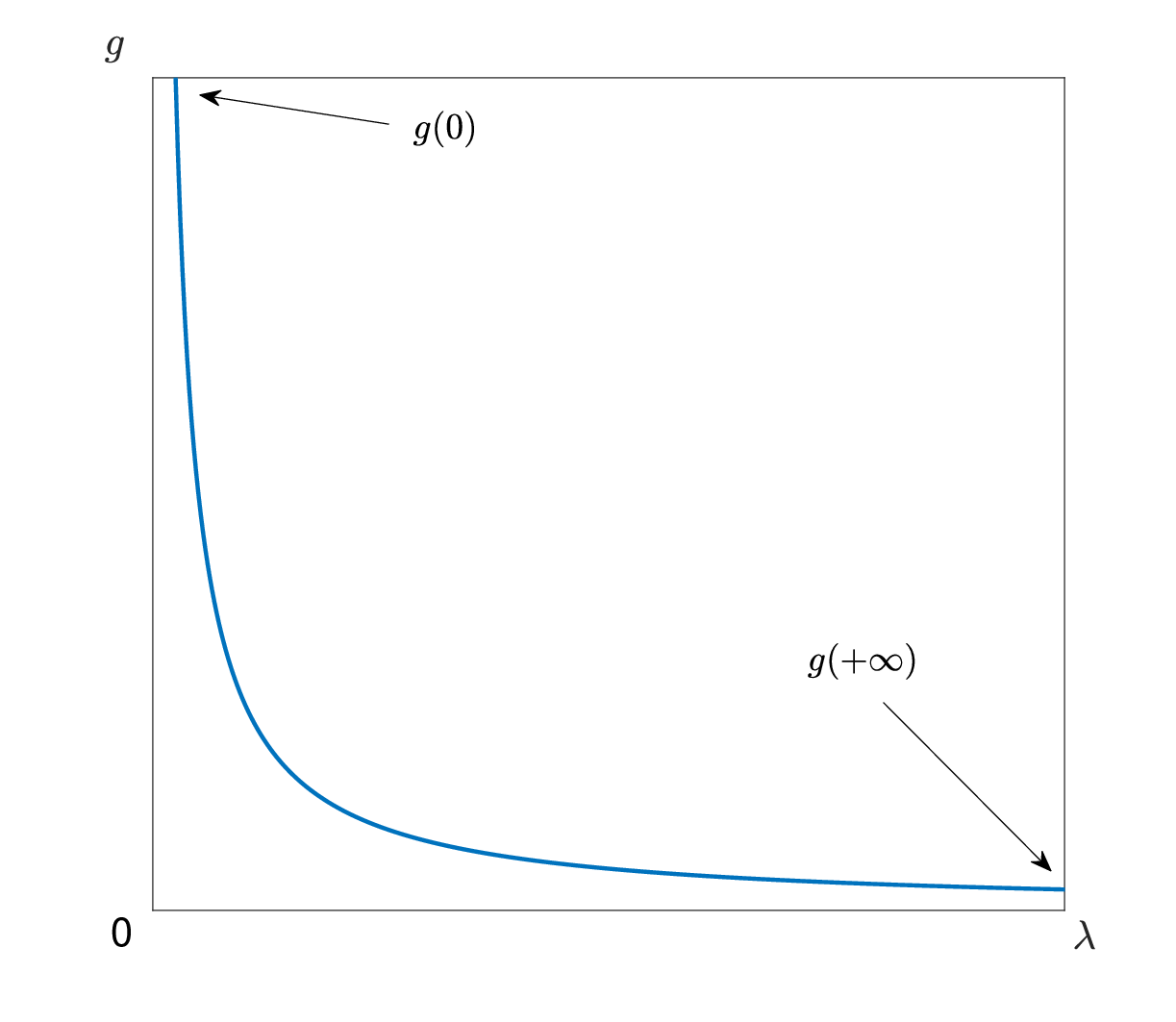}
	}
	\quad
	\subfloat[Case 1, discontinuous]{
		\includegraphics[width=7.5cm]{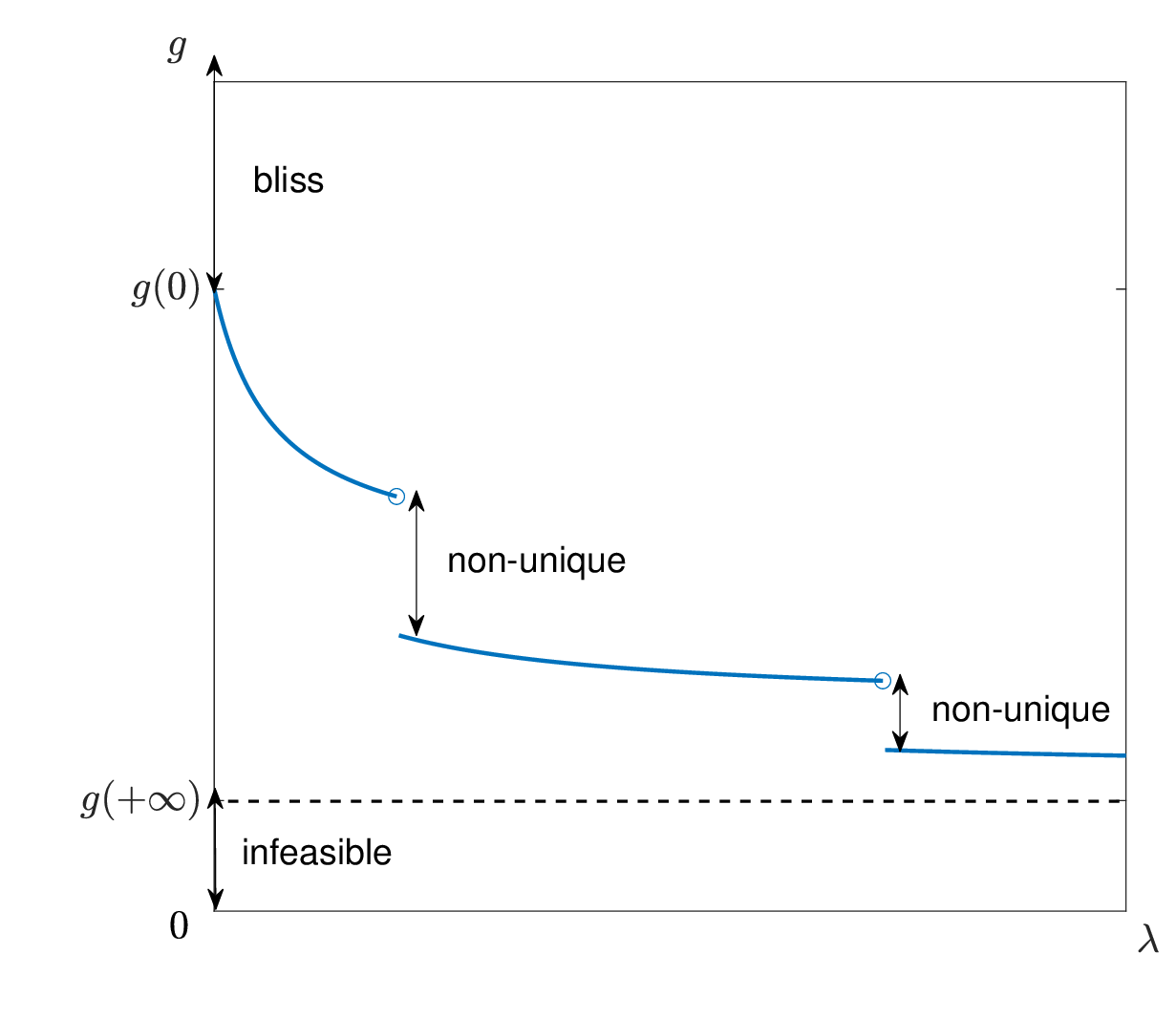}
	}
	\quad
	\subfloat[Case 2, continuous, $g(\lambda_0)<+\infty=g(\lambda_0-)$]{
		\includegraphics[width=7.5cm]{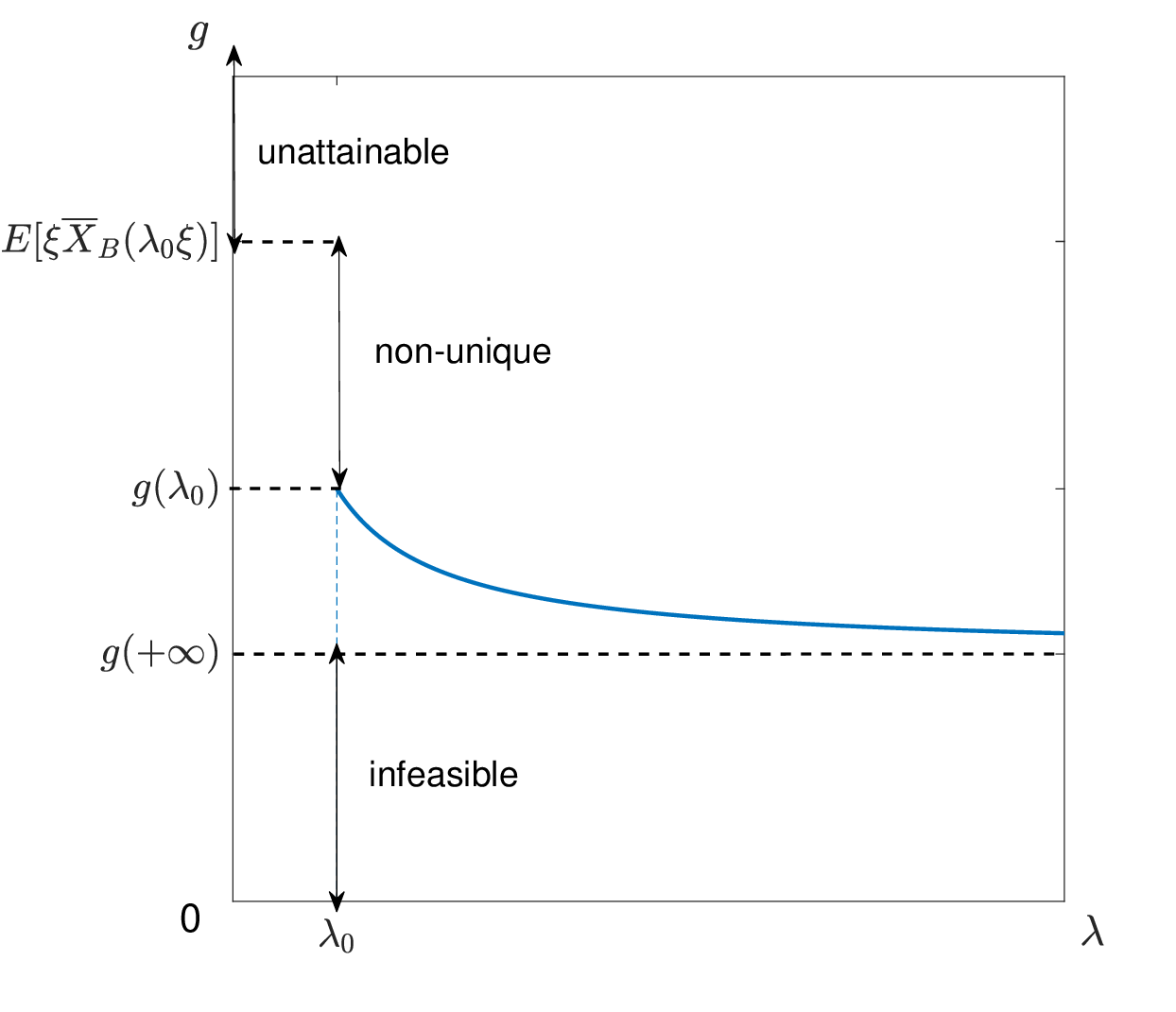}
	}
	\quad
	\subfloat[Case 2, discontinuous, $g(\lambda_0)<+\infty=g(\lambda_0-)$]{
		\includegraphics[width=7.5cm]{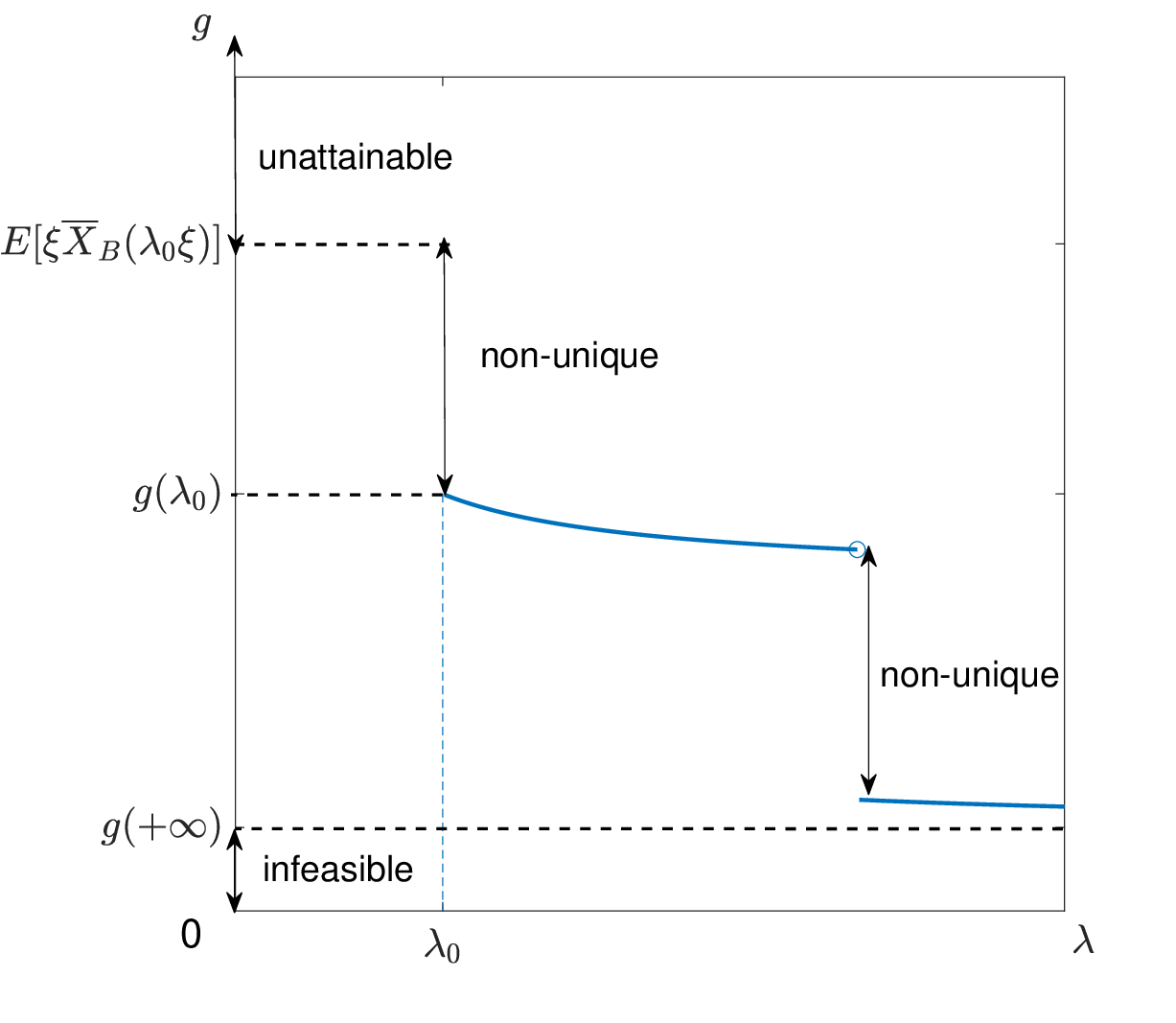}
	}
	\quad
	\subfloat[Case 2, continuous, $g(\lambda_0)=+\infty$]{
		\includegraphics[width=7.5cm]{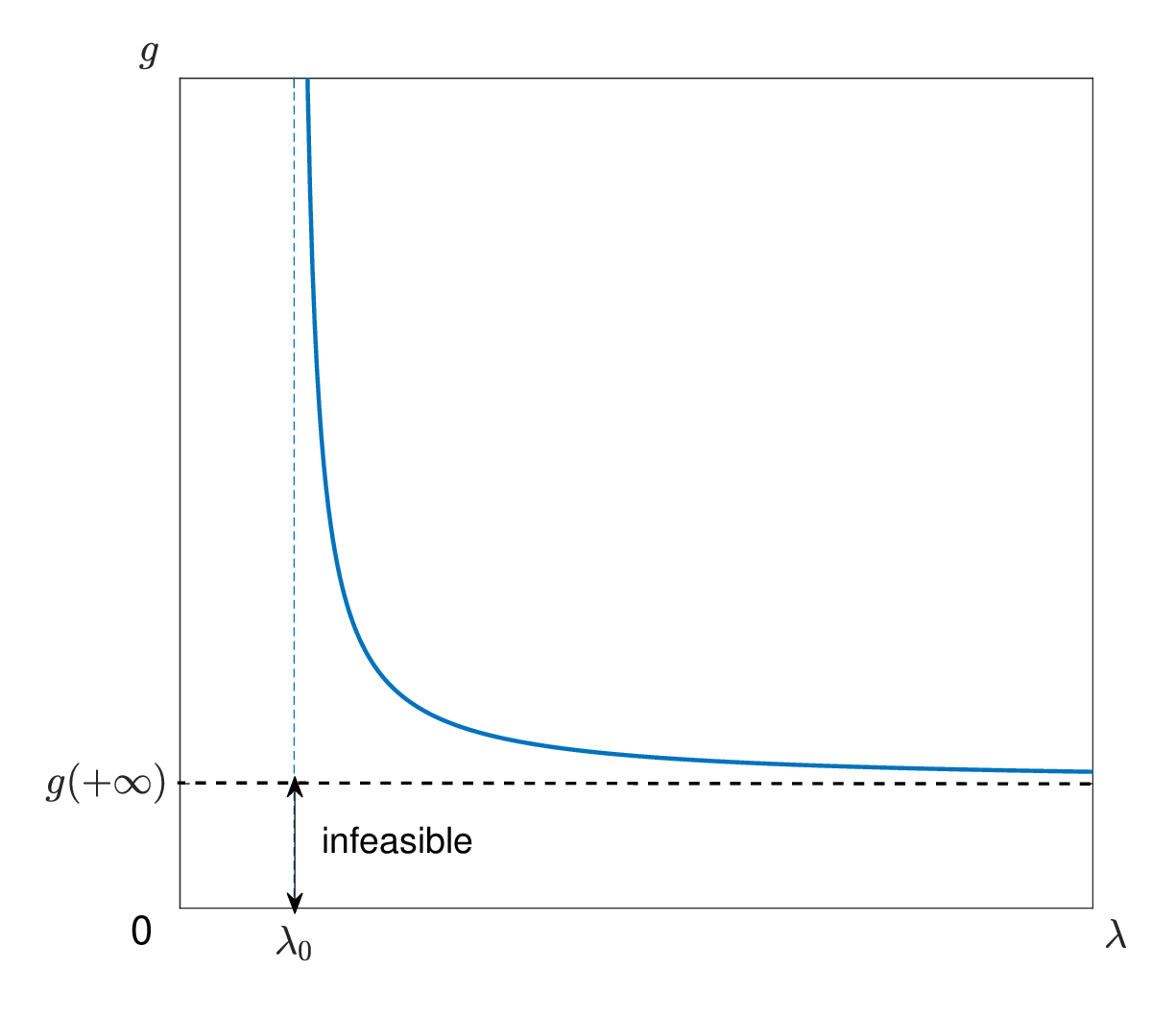}
	}
	\quad
	\subfloat[Case 2, discontinuous, $g(\lambda_0)=+\infty$]{
		\includegraphics[width=7.5cm]{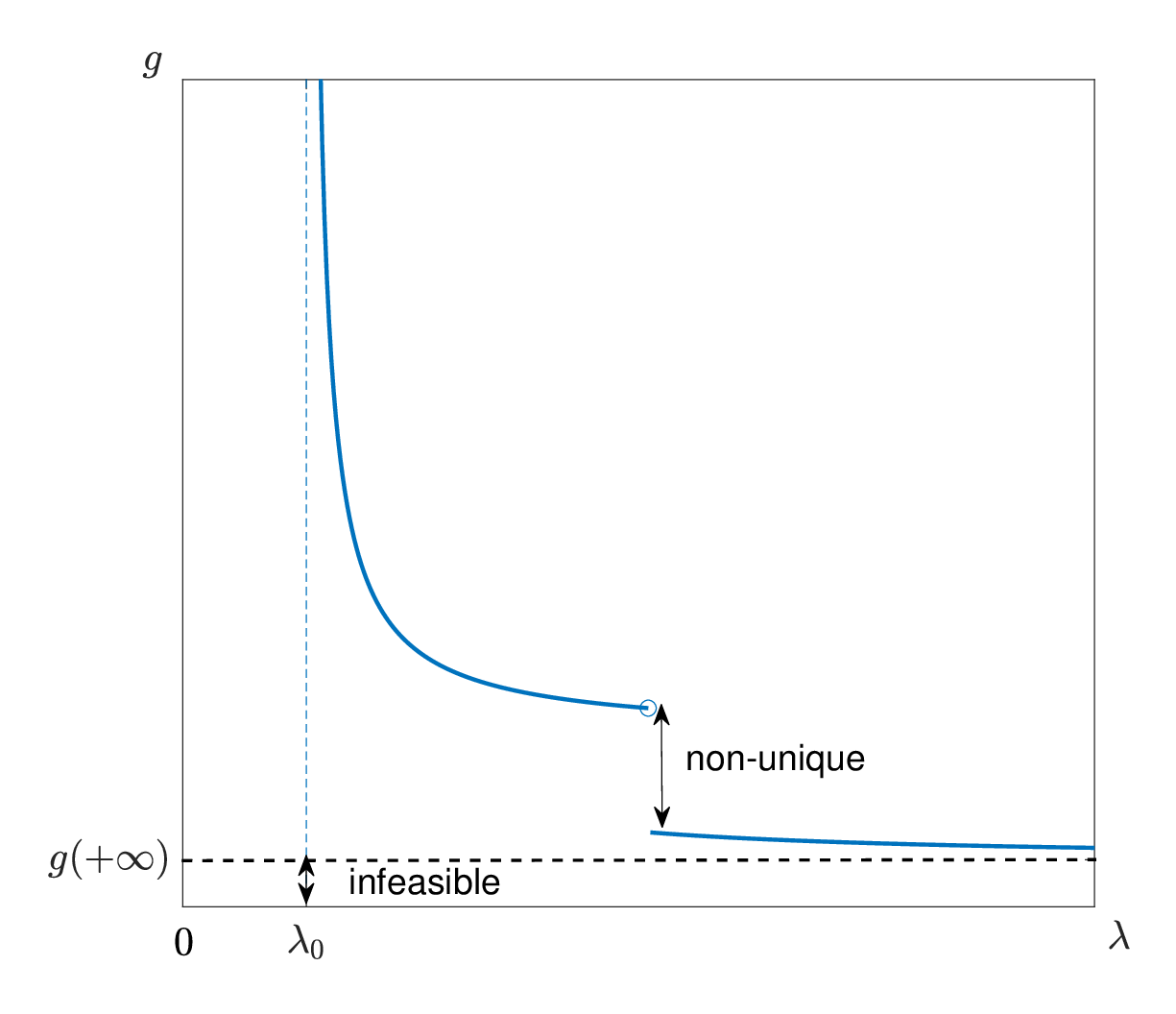}
	}
	\caption{Some examples for $g$. Case 3 is not included because $g$ is not finite. Especially, in Case 1, $g(0)$ can be finite or infinite and there is no bliss region for the latter. In Case 2, the term $\mathbb{E}[\xi\overline{X}_B(\lambda_0\xi)]$ can be finite or infinite, and there is no unattainable region for the latter. Meanwhile, in all three cases, the term $g(+\infty)$ may equal to or less than 0, and then there will be no infeasible region.}
	\label{fig:ex g}
\end{figure}

From Lemma \ref{lem} we know that $\underline{X}_b(\cdot)$ is nonincreasing and right-continuous on $[0,+\infty]$, while $\overline{X}_b(y)$ is nonincreasing and left-continuous. They both have countable discontinuous points. Moreover, $\underline{X}_b(\cdot)$ or $\overline{X}_b(\cdot)$ is discontinuous at $y\in(0,+\infty)$ if and only if its right limit does not equal to its left limit, i.e., $\underline{X}_b(y)\not=\overline{X}_b(y)$, which means that  $\mathcal{X}_b(y)$ is not a singleton.

For $\lambda\in[0,+\infty]$, we take $X^*=\underline{X}_B(\lambda\xi)$ as a candidate of the measurable selection. Define
\begin{equation}\label{eq:g}
	g(\lambda)=\E\[\xi \underline{X}_B \(\lambda\xi\)\].
\end{equation}
As $\underline{X}_B(\lambda\xi)\ge\underline{x}_B$ and $\mathbb{E}[\xi\underline{x}_B]>-\infty$, we know
\begin{equation}\label{eq:g_expect-}
\mathbb{E}\left[\xi\underline{X}_B(\lambda\xi)^-\right]\le\mathbb{E}[\xi\underline{x}_B^-]<+\infty.
\end{equation}
Hence the expectation in \eqref{eq:g} is well-defined, and $g(\lambda)>-\infty$ holds for any $\lambda\in[0,+\infty]$.

Different from the literature, there are two new features of $g$ in our discussion: finiteness and continuity. For the finiteness, there are three possible cases in terms of $g$ (Figure \ref{fig:ex g}):


\begin{enumerate}
	\item[Case 1:] $g(\lambda)<+\infty$ for any $\lambda>0$.
	
	\item[Case 2:] There exists some $\lambda_0>0$ such that
	\begin{equation}\label{eq:case}
		g(\lambda)<+\infty \text{ on } (\lambda_0,+\infty), ~ g(\lambda) = +\infty \text{ on } (0, \lambda_0).
	\end{equation}
	
	\item[Case 3:] $g(\lambda)=+\infty$ for any $\lambda>0$.
\end{enumerate}
{\color{black}
\begin{remark}
	Proposition \ref{suff cond} is also a sufficient condition for Case 1 to hold; see the proof of Proposition \ref{suff cond} for details.
\end{remark}
}
In the literature, Case 1 is always assumed to ensure the existence of the optimal Lagrange multiplier $\lambda^*$ (cf. \citet{KLS1987,KS1999}). From the perspective of Theorem \ref{thm_known}, in the classical univariate problem \eqref{problem0}, we can write $g(\lambda) = \mathbb{E}[\xi (u')^{-1}(\lambda\xi)]$, which is a continuous function if $\xi$ is non-atomic. Therefore, for any $x_0$, one can adopt the intermediate value theorem and find a Lagrange multiplier $\lambda^* > 0$ satisfying $g(\lambda)=x_0$ and solve the problem (cf. \citet{JXZ2008}).

On the infiniteness issue of $g$, Case 2 and Case 3 are investigated in detail in this paper. In Theorem \ref{main}, we give a complete investigation of all three cases. For Case 3, Theorem \ref{main}(1) shows that either Problem (\ref{problem1}) admits at most one feasible solution, or its optimal value equals $+\infty$.  
Case 2 is more complicated, as in this case, the optimal Lagrange multiplier may exist only for some $x_0$, though we are only using $\underline{X}_B(\lambda\xi)$  and $\overline{X}_B(\lambda\xi)$ as a representation of $\mathcal{X}_B(\lambda\xi)$,  Theorem \ref{main} shows that a finite optimal solution does not exist if we cannot find a finite optimal solution using  $\underline{X}_B(\lambda\xi)$  and $\overline{X}_B(\lambda\xi)$.

On the continuity issue of $g$, $g$ may be discontinuous in our model. In fact, based on Lemma \ref{lem}, we know that $\underline{X}_b(y)$ is nonincreasing and right-continuous with respect to $y$. The monotone convergence theorem implies that $g$ is nonincreasing and right-continuous, $g(\lambda_0)=\lim\limits_{\lambda\to\lambda_0+}g(\lambda)$, which can be finite or infinite.
However, $g$ is discontinuous at some $\lambda$ if the set of random variables $\X_B(\lambda \xi)$ is not a singleton, that is, $\underline{X}_B(\lambda\xi) < \overline{X}_B(\lambda\xi)$ happens with a positive probability value.
In this case, to find an optimal solution for $x_0 \in (g(\lambda), g(\lambda-))$, our basic technique is to ``gradually'' change $X^*$ from $\underline{X}_B(\lambda\xi)$ to $\overline{X}_B(\lambda\xi)$. As such, the value $\mathbb{E}[\xi X^*]$ will vary from $g(\lambda)$ to $g(\lambda-)$, and we will obtain an optimal solution. Significantly, we can construct infinitely many optimal solutions in this case. These results are concluded in Theorem \ref{main}. 

Before we propose the main result of this section, we introduce the concept of the feasible set:
\begin{equation*}
	I\triangleq \left\{x_0: \text{Problem}~ \eqref{problem1} ~\text{has a feasible solution for the initial value} ~x_0 \right\}.
\end{equation*}

As we are considering $X$ and $x_0$ in the range of real numbers (that is, not necessarily positive), if $x_0 < \E\left[\xi\underline{x}_B\right]$, the constraint $\E\left[\xi X\right]\le x_0$ leads to $X< \underline{x}_B$ on a positive-measured set, and hence $\E\left[U(X,B)\right]=-\infty$.  Moreover, as we are studying a state-dependent utility $U$, it may happens that $\E\left[U(X,B)\right]=-\infty$ even if $x_0 \ge \E\left[\xi\underline{x}_B\right]$. Therefore, the notation $I$ is necessary. However, in most cases where $U$ is truncated from a CRRA utility, an S-shaped utility, or other utilities, $U$ has a lower bound itself in its domain. That is, if we have $X>\underline{x}_B$ almost surely, then we have $\E\left[U(X,B)\right]>-\infty$, and $I$ trivially becomes $\left[\mathbb{E}[\xi \underline{x}_B],+\infty\right]$ or $\left(\mathbb{E}[\xi \underline{x}_B],+\infty\right]$. To this light, we do not focus on whether the feasible set $I$ is trivial and use the following result to characterize the structure of the  $I$.

\begin{proposition}\label{thm_feasible}
	Suppose that Assumption \ref{ass_standing} holds. Exactly one of the following holds:
	\begin{enumerate}[(i)]
		\item $I=[\mathbb{E}[\xi \underline{x}_B],+\infty)$;
		\item $I=(\tilde{x},+\infty)$ for some $\tilde{x}\ge\mathbb{E}[\xi \underline{x}_B]$, where $\tilde{x}$ {\color{black} is allowed to equal  $+\infty$}.
	\end{enumerate}
	Moreover, if $x_0=\mathbb{E}[\xi \underline{x}_B]$, there is at most one feasible solution and the candidate is $X=\underline{x}_B$. Finally, if $x_0\ge\mathbb{E}[\xi \overline{x}_B]$, Problem \eqref{problem1} admits a bliss solution.
\end{proposition}

In light of Proposition \ref{thm_feasible}, the most complicated and unclear case of Problem \eqref{problem1} is $\mathbb{E}[\xi \underline{x}_B]<x_0<\mathbb{E}[\xi \overline{x}_B]$.
In particular, when $\underline{x}_B\le 0$ almost surely, every $x_0>0$ locates in this region as $\overline{x}_B$ is typically taken to be
$+\infty$. When $\overline{x}_B<+\infty$ and $\mathbb{E}[\xi\overline{x}_B]<+\infty$, it means that $U(\cdot,B)$ is a constant on $[\overline{x}_B,+\infty)$, and it is possible to obtain a bliss solution for large $x_0$. Proposition \ref{thm_feasible} shows that $I$ is a connected interval with the right endpoint being $+\infty$, but we do not know its left endpoint, which depends highly on $U$ and $B$. In the rest of the paper, we will focus mainly on the case that $\mathbb{E}[\xi \underline{x}_B]<x_0<\mathbb{E}[\xi \overline{x}_B]$ with $x_0\in I$.

\begin{remark}\label{rmk_bliss}
	Practically, $\underline{x}_B>-\infty$ means that the manager cannot bear a loss exceeding $-\underline{x}_B$, and then the condition $x_0\ge\E\big[\xi\underline{x}_B\big]$ requires enough initial capital to face the potential risk. Moreover, if $x_0$ is large enough such that $\mathbb{E}[\xi\overline{x}_B]<x_0<+\infty$, then the manager can reach the maximal utility without any risk. Hence the results indicate that people with a high tolerance for loss are easy to find a satisfactory strategy, while people who are easy to be satisfactory or have enough initial capital will obtain a bliss solution.
\end{remark}

Based on the notation $I$, we introduce our results on optimal solutions in all of the Cases 1-3. Noting that $g(\lambda_0)=\lim\limits_{\lambda\to\lambda_0+}g(\lambda)$, which may be finite or infinite. In the following context, $\mathbb{E}\left[\xi\overline{X}_B(\lambda_0\xi)\right]$ is also allowed to be $+\infty$.

\begin{theorem}\label{main}
	Suppose that $(\Omega,\mathcal{F},\mathbb{P})$ is non-atomic and Assumptions \ref{ass_standing}-\ref{asmp:well-defined} hold.
	\begin{enumerate}[(1)]
		\item If Case 3 holds, then Problem (\ref{problem1}) is infinite for any $x_0\in I\backslash\{\mathbb{E}[\xi\underline{x}_B]\}$.
		\item If Case 1 or Case 2 holds, then we have $\lambda_0 \in [0, +\infty)$ in \eqref{eq:case}. Suppose that $x_0\in I$. We have:
		\begin{enumerate}
			\item if there exists a real number $\lambda^*\ge\lambda_0$ such that $g(\lambda^*)=x_0$, then $X^* =\underline{X}_B\big(\lambda^*\xi\big)$ is an optimal solution;
			\item if  $g(\lambda^*)<x_0\leq g(\lambda^*-)$ for some discontinuous point $\lambda^*>\lambda_0$, it has a positive probability that the set $\mathcal{X}_B\big(\lambda^*\xi\big)$ is not a singleton. Moreover, if $x_0=g(\lambda^*-)$, then $X^*=\overline{X}_B(\lambda^*\xi)$ is an optimal solution; if $x_0<g(\lambda^*-)$, there are infinitely many optimal solutions;
			\item[(c1)] (for \textit{Case 1}, where $\lambda_0=0$) if $x_0\ge g(\lambda_0)=g(0)$, then we have a bliss solution;
			\item[(c2)] (for \textit{Case 2}, where $\lambda_0\!\!>\!0$) let $\theta=\mathbb{E}[\xi \overline{X}_B(\lambda_0\xi)]$. If $x_0=\theta$, then $\overline{X}_B(\lambda_0\xi)$ is an optimal solution; if $g(\lambda_0)<x_0 < \theta$, there are infinitely many optimal solutions. if $x_0> \theta$,  Problem \eqref{problem1} has no finite optimal solutions. We have a sequence $\{\hat{X}_k\}_{k\ge1}$ converging almost surely to $\underline{X}_B(\lambda_0\xi)$ with $\E\[\xi\hat{X}_k\]=x_0$ and
			\begin{equation}\label{eq:main}
				\sup_{X: \E\[\xi X\]\le x_0\atop U(X,B)>-\infty}\E\[U(X,B)\]=\lim_{k\to+\infty}\E\[U\(\hat{X}_k,B\)\] {\color{black} = } \E\[U\(\overline{X}_B\(\lambda_0\xi\),B\)\]+\lambda_0(x_0-\theta).
			\end{equation}
		\end{enumerate}
	\end{enumerate}
\end{theorem}

\begin{remark}
	We point out that in Theorem \ref{main}(2)(a), (b) when $x_0=g(\lambda^*)$ and (c2) when $x_0=\mathbb{E}[\xi \overline{X}_B(\lambda_0\xi)]$, the optimal solution is unique if we further require Assumption \ref{ass_finite}. This is also proved in Appendix \ref{proof3}.
\end{remark}

Theorem \ref{main} gives tractable methods to find out optimal solutions of Problem \eqref{problem1} using $g$ and answers the question that Theorem \ref{general thm} leaves. Specifically, Theorem \ref{general thm} shows that a necessary and sufficient condition for $X$ to be an optimal solution is that $X\in\mathcal{X}_B(\lambda\xi)$ for some $\lambda$ satisfying $\mathbb{E}[\xi X]=x_0$. 
Here we need an extra work of measurable selection on $\mathcal{X}_B(\lambda\xi)$, and Theorem \ref{main} proves the existence of such a selection by construction. (a) and (b) deal with the case $x_0 \in \left(g(+\infty), g(\lambda_0)\right]$, while (c1) and (c2) aim at the situation $x_0> g(\lambda_0)$. For the rest case $x_0\le g(+\infty)=\E[\xi \underline{x}_B]$, we have discussed the case in Proposition \ref{thm_feasible}.

The parts (c1) and (c2) of Theorem \ref{main} (when $x_0>g(\lambda_0)$) are first investigated in this paper. Most literature only concentrates on Case 1. In our models, when a benchmark $B$ is involved, the issue of measurable selection arises. Here, our $g$ considers $\underline{X}_B$ as a candidate. We have shown the existence of a measurable selection satisfying the budget constraint. Based on Theorem \ref{main} and Proposition \ref{thm_feasible}, we figure out the (non-)existence and (non-)uniqueness of an optimal solution for all scenarios. 

The result of non-unique solutions is because of the generality of both $U$ and $B$. On the one hand, if $U$ is strictly concave, then $\mathcal{X}_b(y)$ contains always one element, and non-unique optimal solutions described in (b) will not occur. On the other hand, if $B$ is deterministic, then $\mathcal{X}_B(\lambda\xi)$ is not a singleton if and only if $\lambda\xi$ lies in the discontinuous point set of $\overline{X}_B$. Noting that  $\overline{X}_B$ is decreasing, the set is countable. As such, the probability that the random set $\mathcal{X}_B\big(\lambda^*\xi\big)$ is not a singleton equals zero when the pricing kernel $\xi$ is non-atomic, which means that the non-unique optimal solutions will not happen either. Therefore, it is only a special case of Problem \eqref{problem1}.

\begin{table}
\caption{Solutions of Problem \eqref{problem1} when Assumptions \ref{ass_standing}-\ref{ass_finite} hold and $\inf I=\E[\xi \underline{x}_B]$. Note: Denote by $\lambda^* \in (\lambda_0, +\infty)$ a discontinuous point of $g$.}
	\centering
	\begin{tabular}{|c|c|l|c|}
		\hline
		& Case 1  & \multicolumn{1}{c|}{Case 2}                                                   & Case 3                                                                         \\
		\hline
		$x_0 < \E[\xi \underline{x}_B] = g(+\infty)$  & \multicolumn{3}{c|}{infeasible (Proposition \ref{thm_feasible})}                                                                                                                                                                                                       \\
		\hline
		$x_0 = \E[\xi \underline{x}_B] = g(+\infty)$  & \multicolumn{3}{c|}{at most one feasible solution $\underline{x}_B$ (Proposition \ref{thm_feasible}) }                                                                                                                                                                 \\
		\hline
		$x_0 \in g([\lambda_0, +\infty))$             & \multicolumn{2}{c|}{unique (Thm \ref{main}(2)(a))}                                                                                                                       & \multirow{6}{*}{\begin{tabular}[c]{@{}c@{}}infinite\\(Thm \ref{main}(1)) \end{tabular}}  \\
		\cline{1-3}
		$x_0 \in (g(\lambda^*), g(\lambda^*-))$       & \multicolumn{2}{c|}{non-unique (Thm \ref{main}(2)(b))}                                                                                                              &                                                                                \\
		\cline{1-3}
		$x_0 = g(\lambda^*-)$                         & \multicolumn{2}{c|}{unique   (Thm \ref{main}(2)(b))}                                                                                                                     &                                                                                \\
		\cline{1-3}
		\multirow{3}{*}{$x_0 >g(\lambda_0)$ }     & \multirow{3}{*}{\begin{tabular}[c]{@{}c@{}}bliss\\(Thm \ref{main}(2)(c1)) \end{tabular}} & $x_0>\mathbb{E}[\xi\overline{X}_B(\lambda_0\xi)]$: unattainable (Thm \ref{main}(2)(c2))  &                                                                                \\
		\cline{3-3}
		&                                                                             & $x_0=\mathbb{E}[\xi\overline{X}_B(\lambda_0\xi)]$: unique (Thm \ref{main}(2)(c2))           &                                                                                \\
		\cline{3-3}
		&                                                                             & $x_0<\mathbb{E}[\xi\overline{X}_B(\lambda_0\xi)]$: non-unique (Thm \ref{main}(2)(c2))     &                                                                                \\
		\hline
	\end{tabular}
	
	\label{table1}
\end{table}

\begin{remark}\label{rmk_no monotone}
	As $B$ can be any random variable, the optimal solution $\underline{X}_B(\lambda^*\xi)$ may not be a function of $\xi$. Moreover, as the function $\underline{X}_b(y)$ performs no obvious monotonicity in $b$, $\underline{X}_B(\lambda^*\xi)$ may not be a decreasing function of $\xi$ even if $B$ is a function of $\xi$, {\color{black} and it may not be comonotonic with $\xi$}; see \cite{BMRV2015} for a concrete model. {\color{black} This phenomenon also appears in the literature using quantile formulation. In \cite{JZ2008}, \cite{HZ2011}, and \cite{X2016}, the optimal solution is a decreasing function of the pricing kernel, while in \cite{PWX2023} where a CPT reference and a stochastic benchmark are adopted, the optimal solution may not be a decreasing function of the pricing kernel (though it is still comonotonic with the pricing kernel).}
\end{remark}

For the finiteness and attainability, as our results in Theorem \ref{main} (2)(a)(b)(c1) (and (c2) when $x_0\le\theta$) are tractable, Problem \eqref{problem1} is attainable in all these situations, and one can verify the finiteness directly. For the remaining case (c2) with $x_0>\theta$, Theorem \ref{main} indicates that $\theta$ is the largest initial value for which we can find an optimal solution through $g$ (noting that $\theta$ can be $+\infty$). For this initial value $x_0$, we have an optimal solution $\overline{X}_B(\lambda_0\xi)$, and we can verify the finiteness of Problem \eqref{problem1}. We have two cases:
\begin{itemize}
	\item If Problem \eqref{problem1} is finite for $\theta$, then \eqref{eq:main} shows that for the initial value $x_0$, Problem \eqref{problem1} is also finite. Moreover, Theorem \ref{main} asserts that there is no optimal solution, i.e., the problem is unattainable for $x_0$.
	\item If Problem \eqref{problem1} is infinite for $\theta$, then as $x_0>\theta$ and $U(x,b)$ is nondecreasing in $x$, we can simply turn up the infinite optimal solution $\overline{X}_B(\lambda_0\xi)$ for Problem \eqref{problem1} with initial value $\theta$ to obtain an infinite optimal solution with initial value $x_0$.
\end{itemize}
In a word, we can use the function $g$ and Theorem \ref{main} to completely figure out the existence and expressions of optimal solutions, finiteness and attainability of Problem \eqref{problem1}, while Theorem \ref{general thm} gives support on the uniqueness of the optimal solution.

{\color{black}
To end this section, we propose an example where there may be infinitely many optimal solutions. We provide more detailed characterizations of these solutions and select a specific optimal solution by investigating the liquidation probability.
\begin{example}[Infinitely many solutions and further selection]
Let $\{W_t\}_{0\le t\le T}$ be a standard Brownian motion. Let $S_t=\exp\left\{\(\mu -  \frac{1}{2}\sigma^2\)t + \sigma W_t\right\}$ be the risky asset and $F_t=e^{rt}$ be the risk-free asset with $\mu>r>0$ and $\sigma>0$. Then
$\xi = \exp\left[-\(r+  \frac{1}{2}\theta^2\)T - \theta W_T\right]$ follows a log-normal distribution with $\theta=\frac{\mu-r}{\sigma^2}$. For the benchmark $B$, we follow \cite{BPS2007} to set a general value-weighted benchmark portfolio process:
\begin{equation}\label{benchmark}
	\left\{
	\begin{aligned}
		& \d B_t = \alpha r B_t \d t + \beta B_t \(\frac{\d S_t}{S_t}\), \\
		& B_0 = bx_0,
	\end{aligned}
	\right.
\end{equation}
and $B=B_T$, where $\alpha+\beta=1,~b\ge0$ are fixed constant parameters. Define the utility as  
	\begin{equation}\label{utility}
		U(x,b) =\left\{
		\begin{aligned}
			&(x-b)^{p}, &&  x > b,\\
			&-k(b-x)^{p}, && 0 \le x \leq b.\\
		\end{aligned}
		\right.
	\end{equation}
	where $k>1, p\in (0,1)$ are constants.
Under this setting, we have 
\begin{equation*}
	B=bx_0e^{\(\alpha r+\beta \mu -\frac{\beta^2\sigma^2}{2}\)T+\beta\sigma W_T}=L\xi^{-\frac{\beta \sigma}{\theta}},
\end{equation*}
where $L$ is a constant depending on the market parameters. For the S-shaped utility, we have
\begin{equation*}\label{conjugate point 2}
	\begin{aligned}
		\mathcal{X}_b(y)
		&=\left  \{
		\begin{aligned}
			&\{0\}, &&  b-dI(y) > 0,\\
			&\{0,b+I(y)\}, && b-dI(y)=0,\\
			&\{b+I(y)\}, && b-dI(y) < 0,\\
		\end{aligned}
		\right.
	\end{aligned}
\end{equation*}
where $I(y)=\(\frac{y}{p}\)^{\frac{1}{p-1}}$ and $d>0$ is the unique solution of the tangent equation $1+k d^p=p(d+1)$. Therefore, $\mathcal{X}_B(\lambda\xi)$ has more than one element if and only if $B=dI(\lambda\xi)$, that is,
$$
L\xi^{-\frac{\beta \sigma}{\theta}}=d\(\frac{\lambda\xi}{p}\)^{\frac{1}{p-1}}.
$$
When $\frac{\beta \sigma}{\theta}=\frac{1}{1-p}$ and $\lambda_0=p\(\frac{L}{d}\)^{p-1}$, we have $B=dI(\lambda_0\xi)$ holds almost everywhere, and $\mathcal{X}_B(\lambda_0\xi)=\{0,\frac{d+1}{d}B\}$ is a binary set. Using Theorem \ref{main}, this $\lambda_0$ is a discontinuous point of $g$, and we have $g(\lambda_0+)=0$, $g(\lambda_0-)=\frac{d+1}{d}\E\left[\xi B\right]=\frac{d+1}{d}L\E\left[\xi^{\frac{p}{p-1}}\right]\triangleq C$. For $0<x_0<C$, we have infinitely many optimal solutions. Indeed, for every $A\in\mathcal{F}$, define $X^A=\frac{d+1}{d}B\id_A\in\mathcal{X}_B(\lambda_0\xi)$, if $\E\left[\xi X^A\right]=\frac{d+1}{d}L\E\left[\xi^{\frac{p}{p-1}}\id_A\right]=x_0$, then Theorem \ref{general thm} indicates that $X^A$ is an optimal solution.

Among all these solutions, we can find one with the smallest liquidation probability $\mathbb{P}\left[X^A=0\right]$ (=$P(A)$). If we fix $P(A)=\epsilon>0$, then $\E\left[\xi^{\frac{p}{p-1}}\id_A\right]$ attains its minimum when $A$ has the form 
$A=\{\omega:W_T<\sqrt{T}\Phi^{-1}(\epsilon)\}$, where $\Phi$ is the cumulative distribution function of the standard normal distribution. This is because $\xi$ is decreasing with $W_T$. The minimum value is
$m(\epsilon)=\frac{d+1}{d}Le^{\frac{p}{1-p}(r+\frac{\theta^2}{2(1-p)})T}\Phi\left(\Phi^{-1}(\epsilon)-\frac{p\theta}{1-p}\sqrt{T}\right).$
Similarly, we can prove that the maximum value of $\E\left[\xi^{\frac{p}{p-1}}\id_A\right]$ is
$M(\epsilon)=\frac{d+1}{d}Le^{\frac{p}{1-p}(r+\frac{\theta^2}{2(1-p)})T}\Phi\left(\Phi^{-1}(\epsilon)+\frac{p\theta}{1-p}\sqrt{T}\right).$
Letting $m(\epsilon)\le x_0\le M(\epsilon)$, we can solve the range of the liquidation probability $\epsilon$ from $x_0$. Moreover, we conclude that when $X^A$ attains the minimal liquidation probability $\epsilon_m$, its liquidation happens only for large $W_T>\sqrt{T}\Phi^{-1}(\epsilon_m)$, which means that the market is bad. While $X^A$ attains the maximal liquidation probability $\epsilon_M$, its liquidation happens only for small $W_T<\sqrt{T}\Phi^{-1}(\epsilon_M)$, which means that the market is good. 
\end{example}
}


\vskip 10pt
\section{
	Well-definedness and finiteness}\label{section_condition}

	In this section, we give some further results on the well-definedness and finiteness (we have already given a sufficient condition in Proposition \ref{suff cond}). Based on \eqref{eq:utineq}-\eqref{eq:ut_estimate2} in the proof of Proposition \ref{suff cond}, we can also verify the finiteness of Problem \eqref{problem1} using $g$ and $\underline{X}_B$. We have:

%

\begin{proposition}\label{prop:finite}
	Suppose  that Assumption \ref{ass_standing} holds. If there exists $\lambda\in[0,+\infty)$ such that $g(\lambda)$ is finite, and $\mathbb{E}\left[U\left(\underline{X}_B(\lambda\xi),B\right)\right]$ is well-defined and finite,   then Problem \eqref{problem1} is well-defined and finite for all $x_0\in\R$.
	
\end{proposition}
In light of Proposition \ref{prop:finite}, we can also study the finiteness of Problem \eqref{problem1} using $g$ and
\begin{equation}\label{eq:J}
	J(\lambda)\triangleq\mathbb{E}\left[U\left(\underline{X}_B(\lambda\xi),B\right)\right],~\lambda\in(0,+\infty).
\end{equation}
\begin{assumption}[Well-definedness of $J$]\label{asmp:h}
	The expectation in \eqref{eq:J} is well-defined for any $\lambda \in (0,+\infty)$ and $J(\lambda) > -\infty$ for all $\lambda\in(0,+\infty)$.
\end{assumption}
Under Assumption \ref{asmp:h}, similarly as $\lambda_0$ in \eqref{eq:case}, we define $\lambda_1\in[0,+\infty]$ from
\begin{equation*}
	J(\lambda)<+\infty \text{ for } \lambda\in(\lambda_1,+\infty),~J(\lambda)=+\infty \text{ for }\lambda\in(0,\lambda_1).
\end{equation*}

We have the following results on the value function of Problem \eqref{problem1}:
\begin{theorem}\label{prop:suff2}
	Suppose that Assumptions \ref{ass_standing}-\ref{asmp:well-defined} and \ref{asmp:h} hold and $x_0\in I\backslash\{\mathbb{E}[\xi\underline{x}_B]\}$, then
	\begin{enumerate}[(i)]
		\item If $\lambda_0=+\infty$, then Problem \eqref{problem1} is infinite;
		\item If $\lambda_0<+\infty$, but $\lambda_1=+\infty$, then Problem \eqref{problem1} is infinite;
		\item If $\lambda_0<+\infty$, and $\lambda_1<+\infty$, then Problem \eqref{problem1} is finite, and we have $\lambda_1=\lambda_0$.
	\end{enumerate}
\end{theorem}
\begin{remark}
	Theorem \ref{prop:suff2} indicates that under Assumption \ref{asmp:h}, if Problem \eqref{problem1} is finite (or infinite) for one $x_0$ in the feasible set, then it is finite (or infinite) for all $x_0$ in the feasible set. In this light, we can first compute $g$ and $J$ to verify the finiteness of Problem \eqref{problem1}, and then we use Theorem \ref{main} to find a finite solution or an infinite solution.
\end{remark}

\section{Connection to the univariate framework}\label{section_univ}

In this section, we connect our results to the univariate framework of general utilities, while we assume that $u$ satisfies good conditions (differentiable and strictly concave) in Problem \eqref{problem0}. The univariate utility is state-independent and the univariate framework depends only on the distribution of the wealth $X$, and hence the objective is distribution-invariant. Fix $u$ satisfying Assumption \ref{ass_standing} (we can regard $u$ as a state-dependent utility $U(x,b)$ with constant state). For $y \ge 0$, define the conjugate set function $\mathcal{X}$  as
$$
\begin{aligned}
	&\X(y)=\mathop{\arg \sup}\limits_{x\in\mathbb{R}\cup\{\pm\infty\}}{\left[u(x)-yx\right]}.
\end{aligned}
$$
Define
$
\underline{X}(y)=\inf\mathcal{X}(y),~\overline{X}(y)=\sup\mathcal{X}(y)
$
as well as $g(\lambda)=\E\[\xi\underline{X}(\lambda\xi)\]$, and $\lambda_0$ as in \eqref{eq:case}.

Here the slope of $u$ near $+\infty$ can be positive.  For simplicity, we assume that $\xi$ is integrable, then one can verify that the feasible set $I=\[\underline{x}\E\[\xi\],+\infty\)$ or $\(\underline{x}\E\[\xi\],+\infty\)$, which depends on the type of $\underline{x}$. Moreover, Assumption \ref{asmp:well-defined} is also easy to verify.

Using the results in Sections \ref{section_general}-\ref{section_condition}, we propose a complete result for the univariate utility $u$ in Theorem \ref{thm_known} and the following discussion.


\begin{theorem}\label{thm_known}
	Suppose that $(\Omega, \F, \p)$ is non-atomic and that the univariate version of Assumptions \ref{ass_standing} and \ref{ass_finite} hold.
	
	(1) $X^*$ is an optimal solution of Problem \eqref{problem0} if and only if there exists $\lambda^* \ge 0$ such that $X^* \in \mathcal{X}(\lambda^* \xi)$ and $\E\[\xi X^*\]=x_0$.
	
	(2) 
	The optimal solution is unique if the set $\X(y)$ is singleton for any $y \geq 0$ or $\xi$ is non-atomic. The set of optimal solutions $\X(\lambda^* \xi)$ may be non-unique if $u$ is non-concave and $\xi$ is atomic.
\end{theorem}

Theorem \ref{thm_known} is a complete result for general (discontinuous and non-concave) utilities, which extends the univariate results in literature; see \citet{JXZ2008} and \citet{R2013}. Under our framework, Theorem \ref{thm_known} becomes a direct corollary of the multivariate version Theorems \ref{thm_concave}-\ref{general thm}: to prove Theorem \ref{thm_known}, one can simply regard $B$ as a constant in Theorems \ref{thm_concave}-\ref{general thm} (removing all the $``B"$ appeared in the proof). Thus, our framework also serves to offer proof for the univariate result. 
Apart from the uniqueness, for $x_0\in\R$ we also have:
\begin{enumerate}[(i)]
	\item When $\lambda_0=+\infty$, the problem is infinite for every $x_0\in\(\underline{x}\E\[\xi\],+\infty\)$. For $x_0<\underline{x}\E\[\xi\]$, there is no feasible solutions. While for $x_0=\underline{x}\E\[\xi\]$, there is at most one feasible solution, which depends on the type of $\underline{x}$.
	\item When  $\lambda_0\in(0,+\infty)$, the	optimal solution(s) exists for all $x_0\in\(\underline{x}\E\[\xi\],\E\[\xi\overline{X}(\lambda_0\xi)\]\]$, and one can directly verify the finiteness or infiniteness of the solution. For $x_0>\E\[\xi\overline{X}(\lambda_0\xi)\]\triangleq\theta$, we have two cases:
	\begin{itemize}
		\item If the problem with initial value $\theta$ is finite, the problem for $x_0$ is finite and unattainable, and we can find a sequence $\{\hat{X}_k\}$ converges to $\overline{X}(\lambda_0\xi)$ almost surely satisfying $\E\[\xi\hat{X}_k\]=x_0$, $u(\hat{X}_k)>-\infty$ and
		\begin{equation}\label{eq:univarite_ineq}
			\sup\limits_{X: \mathbb{E}[\xi X]\le x_0\atop u(X)>-\infty}\mathbb{E}[u(X)]=\lim_{k\to+\infty}\E\[u(\hat{X}_k)\]=\E\[u\(\overline{X}(\lambda_0\xi)\)\]+\lambda_0(x_0-\theta).
		\end{equation}
		\item If the problem with initial value $\theta$ is infinite, then it is also infinite with initial value $x_0$, and an optimal solution $X^*$ can be obtained by letting the optimal solution $\overline{X}(\lambda_0\xi)$ for initial value $\theta$ satisfy $\E\[\xi X^*\]=x_0$.
	\end{itemize}
	\item When $\lambda_0=0$, then for each $x_0>\underline{x}\E\[\xi\]$ there exists an optimal solution, and the finiteness or infiniteness can be verified directly.
\end{enumerate}

In \citet{R2013}, it is claimed that there exists $\tilde{x}\in(0,+\infty]$ such that for $x_0\in(0,\tilde{x})$, Problem \eqref{problem0} admits an optimal solution. While in Theorem \ref{thm_known}, $\tilde{x}$ can be indeed determined as $\tilde{x}=\mathbb{E}\left[\xi\overline{X}(\lambda_0\xi)\right]$, and the case when $x_0\ge\tilde{x}$ is also supplemented.
The result of a differentiable and strictly concave function $u$ exists in \citet{JXZ2008} and \citet{KS1998}; in this case, we have $\X = (u')^{-1}$ and $X^* = (u')^{-1}(\lambda^* \xi)$.

A key insight of Theorem \ref{thm_known} is that there may also exist non-unique solutions in the univariate framework. It may only happen if $u$ is non-concave and $\xi$ is atomic. In this case, the random set $\mathcal{X}(\lambda^* \xi)$ is not a singleton and $g(\lambda) := \E[\xi \X(\lambda \xi)]$ is discontinuous at the Lagrange multiplier $\lambda^*$. The optimal solution is unique if $\xi$ is non-atomic. However, in the multivariate framework, even if $\xi$ is non-atomic, there may exist non-unique solutions. It is because of the state-dependent utility $U$ and the stochastic benchmark $B$. They lead to a non-singleton random set $\mathcal{X}(\lambda^* \xi)$.

Traditionally, the optimal solution for the affine utility is known to be $X^* = +\infty$ if the pricing kernel $\xi$ is non-atomic. We finally propose Example \ref{ex:affine}, showing that for the affine utility, {\color{black} $\xi$ must be atomic for the optimal solution to exist.} We claim that the affine utility, although not satisfying Assumption \ref{ass_standing}, satisfies Assumption \ref{ass_G} in the Appendix \ref{app_general}, where we show that Theorems \ref{general thm}-\ref{main} also hold under the weaker but more complex Assumption \ref{ass_G}.
\begin{example}[Affine utility]\label{ex:affine}
	Fix $k \in (0, +\infty)$ and $L \in \R$. Let
	$
	u(x)= k x, ~ x \geq L.
	$
	For any $x_0 > L \E[\xi]$, we find that the optimal solution exists if and only if $\xi_{min} \triangleq \essinf \xi \in (0, +\infty)$ and $\xi$ has an atom at $\xi_{min}$. Indeed,
	if the optimal solution exists, it is given by $X^* \in \X(\lambda^* \xi)$, where $\lambda^* \in (0, +\infty)$ is a constant (to be determined) and
	$$
	\X(y) = \left\{
	\begin{aligned}
		& L, && y > k;\\
		& [L, +\infty], && y = k;\\
		& +\infty, && 0 \leq y < k.
	\end{aligned}
	\right.
	$$
	To satisfy $\E[\xi X^*] = x_0$, we obtain that $\lambda^* = k/\xi_{min}$ and the optimal solution is
	$$
	X^* = \left\{
	\begin{aligned}
		& L, && \xi > \xi_{min};\\
		& \frac{x_0 - L \E\[\xi \id_{\{\xi > \xi_{min}\}}\]}{\xi_{min} \p[\xi = \xi_{min}]}, && \xi = \xi_{min}.
	\end{aligned}
	\right.
	$$
\end{example}
{\color{black} This result indicates that the optimal terminal wealth $X^*$ will gamble on the atomic $\xi=\xi_{min}$ and keep the lowest level $L$ for other cases.}
\section{Application}\label{section_ex2}
In this section, we formulate a constrained utility optimization problem with state-dependent benchmarks:
\begin{equation}\label{CVaR prob}
	\begin{split}
		&\sup_{X\ge0}\mathbb{E} \left[U\left(X,B_1\right)\right]\\
		&\text{subject to}~\mathbb{E}[\xi X]\le x_0,~\mathbb{P}[X\ge B_2]\ge1-\alpha,\\
	\end{split}
\end{equation}
where $U$ is the S-shaped utility function defined in Example \ref{eg2},
$B_1$ is a performance benchmark (reference level) and $B_2$ is a risk benchmark. $\xi$ is a log-normal pricing kernel with the form
$
\xi = \exp(-(r+\frac{1}{2}\theta^2)T-\theta \sqrt{T}W),
$
where $r$ is the risk-free rate of interest, $\theta$ is the risk premium, and $W$ follows a standard normal distribution.

In the literature of risk management with VaR constraints (cf. \citet{NS2020}), $B_2$ is usually taken as a constant $L$, which means the worst level of return under the given confidence level. However, when the market state is not bad, it is reasonable to require a higher level of return. If we raise $L$ only in the situation when the market does not perform poorly, the risk of the return will not become larger. That is, we allow $B_2$ (and also $B_1$) to be stochastic.

In the next subsection, we propose a general solution to Problem \eqref{CVaR prob}. In consideration of the numerical results, we will only consider three simple plans as follows (where $L_i$ and $w$ are constants):

\textbf{Plan I}: $B_1=L_1$, $B_2=L_3$.

\textbf{Plan II}: $B_1=L_1$, $B_2=L_3+(L_4-L_3)\id_{\{W>w\}}$. Note that $B_2$ is state-dependent.

\textbf{Plan III}: $B_1=L_1+(L_2-L_1)\id_{\{W>w\}}$, $B_2=L_3+(L_4-L_3)\id_{\{W>w\}}$. Note that both $B_1$ and $B_2$ are state-dependent.

Plan I is the most classical case. In Plan II, we adjust the level of return in the VaR constraint from $L_3$ to $L_4$ when the market is not bad, while in Plan III, the performance benchmark $B_1$ will also be raised from $L_1$ to $L_2$. For simplicity, we also assume $B_1>B_2$.
To solve Problem \eqref{CVaR prob}, we apply the Lagrange method to convert it to Problem \eqref{problem1}. In literature, the quantile formulation method is adopted to solve an optimization problem with risk management, but the validity of this method requires the objective function to be distribution-invariant, i.e., the values of the objective function are the same for identically distributed random variables $X$. As the objective function involves stochastic benchmarks $B_1$ and $B_2$ in Problem \eqref{CVaR prob}, the quantile formulation method does not work here.
\subsection{Theoretical results: state-dependent utility optimization}
For $\mu\ge0$, define the modified utility function
\begin{equation*}
	U^{\mu}(x,(b_1,b_2))=
	\left\{\begin{array}{ll}
		U(x,b_1)+\mu\id_{\{x\ge b_2\}},&x\ge0,\\
		-\infty,&x<0.
	\end{array}\right.
\end{equation*}

Denote $b=(b_1,b_2)$ and $B=(B_1,B_2)$. Consider the converted problem without the risk constraint:
\begin{equation}\label{remod prob}
	\begin{split}
		&\sup_{X\ge0}\mathbb{E} \left[U^{\mu}\left(X,B\right)\right]\\
		&\text{subject to}~\mathbb{E}[\xi X]\le x_0.\\
	\end{split}
\end{equation}
Under Plans I-III, we have the following result.
\begin{theorem}\label{thm ex}
	For every $x_0>0$, Problem \eqref{remod prob} admits a unique finite optimal solution $X^{\mu}=\underline{X}^{\mu}_B(\lambda(\mu)\xi)$ with $\lambda(\mu)$ as a function of $\mu$. When $b_1>b_2$, the function $\underline{X}_b^{\mu}(y)$ is given as follows:
	\begin{equation*}
		\begin{aligned}
			&\textcircled{1}.~~~~\text{if}~y_1(b)\le y_2^\mu(b),~\text{then}&&\textcircled{2}.~~~~\text{if}~y_1(b)\ge y_2^\mu(b),~\text{then}\\
			&\underline{X}_b^{\mu}(y)=\left\{
			\begin{array}{ll}
				b_1+\(\frac{p}{y}\)^{\frac{1}{1-p}},&y<y_1(b),\\
				b_2,&y_1\le y<y_2^\mu(b),\\
				0,&y\ge y_2^\mu(b),
			\end{array}
			\right.
			&&\underline{X}_b^{\mu}(y)=\left\{
			\begin{array}{ll}
				b_1+\(\frac{p}{y}\)^{\frac{1}{1-p}},&y<y_3^\mu(b),\\
				0,&y\ge y_3^\mu(b),
			\end{array}
			\right.
		\end{aligned}
	\end{equation*}
	where $y_1(b)$, $y_2^\mu(b)$ and $~y_3^\mu(b)$ are defined by
	\begin{equation*}\footnotesize
		y_1(b)=p\(\frac{d(k)}{b_1-b_2}\)^{1-p},~
		y_2^\mu(b)=\frac{1}{b_2}\(\mu+k b_1^p - k (b_1-b_2)^p\),~
		y_3^\mu(b)=p\(\frac{d(k+\mu b_1^{-p})}{b_1}\)^{1-p},
	\end{equation*}
	and $d(s)$ denotes the solution of the equation $1+sx^p = p(x+1)$.
\end{theorem}

Based on Theorem \ref{thm ex}, we  solve the equation $\mathbb{P}[X^{\mu}\ge B_2]=1-\alpha$ to determine the Lagrange multiplier $\mu^*$, and then $X^*:=X^{\mu^*}$ gives a finite optimal solution of Problem \eqref{CVaR prob}. 

\begin{remark}
	The existence of $\mu^*$ is not a trivial issue, and for some initial value $x_0$ there may be no such $\mu^*$. But this is not the key point of this paper; see e.g. \citet{W2018} for details.
\end{remark}

\subsection{Numerical results}
The following figure shows the numerical result of the relation between $X^{*}$ and $\xi$ in Plans I, II and III respectively. The parameter is taken as $r=0.03$, $\theta=0.3$, $T=10$, $p=0.5$, $k=2.25$, $x_0=30$, $L_1=60$, $L_2=70$, $L_3=40$, $L_4=50$, $w=-1$, $\alpha=0.05$. The top axis shows the cumulative probability of $\xi$ from left to right.
\begin{figure}[htbp] 
	\centering
	\includegraphics[width = 0.5\textwidth]{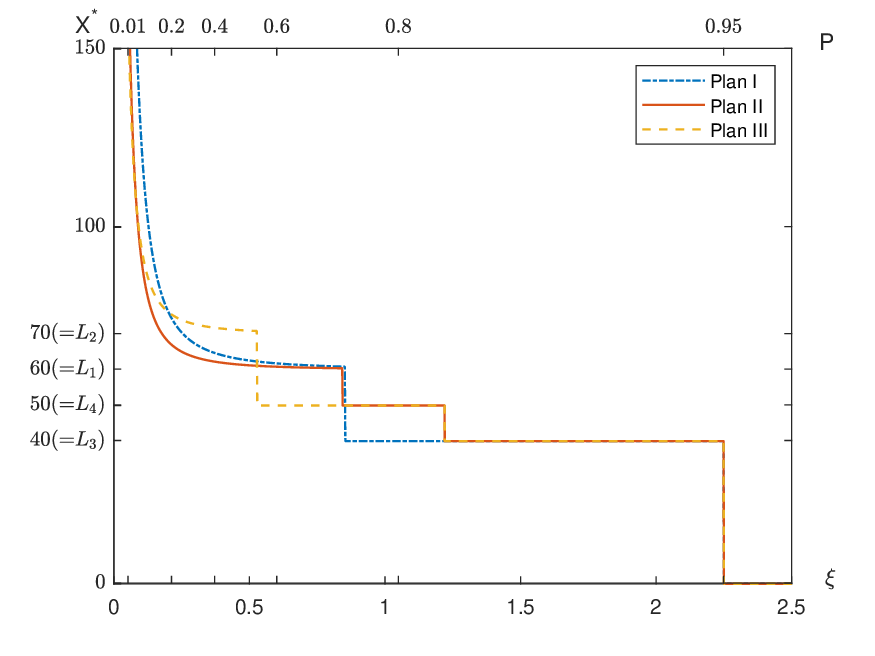}
	\caption{relationship between $X^*$ and $\xi$.}
	\label{fig:numerical}
\end{figure}

We know from Figure \ref{fig:numerical} that, if we adjust the return level higher in the VaR when the market is not bad, that is, we change from Plan I to Plan II, then the return performs better when $\xi\in[0.85,1.22]$ with a stable increment from $L_3 (=40)$ to $L_4 (=50)$, and the probability is about 10.8\%. 
While for $\xi<0.85$, the return of Plan I is higher than that of Plan II. That is, when the market performs well, Plan I gives a higher return than Plan II, and the gap will increase rapidly from $0$ as $\xi$ decreases. Moreover, for a small probability ($<1$\%, when the market performs quite well), the relative gap between Plan I and Plan II reaches more than $50\%$. In conclusion, compared to Plan I, Plan II (different $B_2$) increases the return when the market is not too bad by reducing the return when the market is quite good, which suits bearish managers or people with a high risk aversion.

In Plan III, the manager has higher anticipation than in Plan II when the market does not perform poorly. The result shows that his return becomes higher than Plan II when $\xi<0.53$ with a probability of about $54.8\%$, and the gap decays quickly to $0$ when $\xi$ decreases (which means that the market state is better). While for $\xi\in[0.53,0.84]$, Plan II gives a higher return than Plan III, and the probability is about 18.3\%. Therefore, compared to Plan II, Plan III (different $B_1$) increases the return when the market is relatively good by reducing the return when the market is relatively bad, which suits bullish managers or people with a low risk aversion.
 As Plan I cares least about the risk among the three plans, its return for $\xi<0.2$ (when the market performs quite well) is the highest, and the probability is about 18.42\%.

\begin{remark}
	In Plans II and III, we have $\mathbb{P}[X^*\ge L_3]=0.95$, which is the first stage of the risk value $B_2$. For the second stage, we have $\mathbb{P}[X^*\ge L_4]=0.84$, and events of reaching the second stage are equivalent to the event that $\xi<1.22$. As a possible extension, we consider a risk variable with multiple stages such as:
	\begin{equation*}
		B_2=f(\xi)=L_1\id_{(\xi_1,+\infty)}(\xi)+\sum_{k=2}^{n-1}L_{k}\id_{(\xi_{k},\xi_{k-1}]}(\xi)+L_n\id_{(-\infty,\xi_n]}(\xi),
	\end{equation*}
	which may become an alternative for the model with multiple risk constraints as
	$$\mathbb{P}[X\ge L_i]\ge 1-\alpha_i, ~i=1,2,...,n.$$
\end{remark}
\begin{remark}
	Apart from Plans I-III, if we take $B_1-B_2=d(k)\big(\frac{\xi}{p}\big)^{\frac{1}{p-1}}$, then $y_1(B)=\xi$, and $g(\lambda)$ may be discontinuous at $\lambda=1$. In this case, Theorem \ref{main} shows that there may exist infinitely many optimal solutions.
\end{remark}

\section{Concluding remarks}\label{section_conclusion}
We propose a framework of state-dependent utility optimization with general benchmarks. We give a detailed and complete discussion on the feasibility, finiteness, and attainability. 
We find that: (i) the optimal solutions may be a random set,  which possibly consists of infinitely many optimal solutions; (ii) the Lagrange multiplier may not exist because the function $g$ in \eqref{eq:g} is discontinuous or infinite at some $\lambda^*$; (iii) the measurability issue may arise when applying the concavification to a multivariate utility function 
and when selecting a candidate from the non-unique optimal solutions. We address these technical issues, especially for measurability, and we do not assume a priori that the optimal Lagrange multiplier exists. 
 In light of (i) and (iii), it is of interest to study how to  
 further select the best one from non-unique optimal solutions in future research.

Finally, we stress that the framework does not include probability distortion. When the reference level is deterministic in cumulative prospect theory, Problem \eqref{problem0} can be further modeled with a probability distortion on $X$, which can be solved by the quantile formulation approach. As the benchmark $B$ may be stochastic, the objective function is no longer distribution-invariant and the quantile formulation approach does not work for this framework in general. 

\section*{Acknowledgments}
Y. Liu acknowledges financial support from the research startup fund at The Chinese University of Hong Kong, Shenzhen. 
 The authors acknowledge support from the National Natural Science Foundation of China (Grant Nos. 12271290, 11871036). The authors are grateful to the group members of Mathematical Finance and Actuarial Science at the Department of Mathematical Sciences, Tsinghua University for useful feedback and useful conversations. 




\appendix

\section{Proofs in Section \ref{section_model}}

\begin{proof}[Proof of Proposition \ref{suff cond}.]
	Denote $\gamma(b)=U(\underline{x}_b+\theta(b),b)>-\infty$. We prove
	\begin{equation}\label{ineqA0}
		\overline{X}_b(y)\le\big(1+\frac{K_1(b)}{y}\big)^{\frac{1}{1-\delta}}+K_2(b),
	\end{equation}
	where $K_1(b)=2(u_1(b)+u_2(b)+|\gamma(b)|),K_2(b)=K(b)+2\theta(b)+2|\underline{x}_b|$.
	\vskip 5pt
	If $x\triangleq \overline{X}_b(y)>\big(1+\frac{K_1(b)}{y}\big)^{\frac{1}{1-\delta}}+K_2(b)$, then, based on the definition of $\overline{X}_b(y)$, we know that
	$
	U(x',b)-yx'\le U(x,b)-yx
	$
	holds for any $x'>\underline{x}_b$.
	Letting $x'=\frac{x}{2}$, we have
	\begin{equation}\label{ineqA1}
		x\le \frac{2}{y}\left[U\left(x,b\right)-U\left(\frac{x}{2},b\right)\right].
	\end{equation}
	As $x>K_2(b)$, we know $\frac{x}{2}>\underline{x}_b+\theta(b)$ and $U(\frac{x}{2},b)\ge\gamma(b)$. In addition, as $x>K_2(b)>K(b)$, we also have
	$U(x,b)\le u_1(b)+u_2(b)x^\delta$, and then (\ref{ineqA1}) leads to
	$
	x\le \frac{2}{y}\big[u_1(b)+u_2(b)x^\delta-\gamma(b)\big],
	$
	or
	\begin{equation}\label{ineqA2}
		x^\delta\big(x^{1-\delta}-\frac{2u_2(b)}{y}\big)\le \frac{2}{y}(u_1(b)-\gamma(b)).
	\end{equation}
	However, as $x>\big(1+\frac{K_1(b)}{y}\big)^{\frac{1}{1-\delta}}$, we have $x^\delta>1$, and then
	$
	x^{1-\delta}-\frac{2u_2(b)}{y}>\frac{2u_1(b)+2|\gamma(b)|}{y}.
	$
	As such, (\ref{ineqA2}) leads to $\frac{2}{y}(u_1(b)+|\gamma(b)|)<\frac{2}{y}(u_1(b)-\gamma(b))$, which is a contradiction. Thus (\ref{ineqA0}) holds.
	\vskip 5pt
	For $\lambda>0$, using (\ref{ineqA0}), we have $\overline{X}_B(\lambda\xi)\le\big(1+\frac{K_1(B)}{\lambda\xi}\big)^{\frac{1}{1-\delta}}+K_2(B)$, and there exists $K_3>0$, s.t.
	\begin{equation}\label{ineqA3}
		\overline{X}_B(\lambda\xi)\le K_3\big(1+\lambda^{-\frac{1}{1-\delta}}\xi^{-\frac{1}{1-\delta}}K_1(B)^{\frac{1}{1-\delta}}+\theta(B)+|\underline{x}_B|+K(B)\big).
	\end{equation}
	It follows that for some $K_4>0$:
	\begin{equation*}
		\begin{aligned}
			\xi \overline{X}_B(\lambda\xi)^+\le K_4\bigg[\xi\big(1+|\underline{x}_B|+&\theta(B)+K(B)\big)+\lambda^{-\frac{1}{1-\delta}}\xi^{-\frac{\delta}{1-\delta}}\bigg(u_1(B)^{\frac{1}{1-\delta}}+u_2(B)^{\frac{1}{1-\delta}}+|\gamma(B)|^{\frac{1}{1-\delta}}\bigg)\bigg].
		\end{aligned}
	\end{equation*}
	Based on our conditions in Proposition \ref{suff cond}, we know 
	\begin{equation}\label{eq:g_expect+}
		\mathbb{E}[\xi \underline{X}_B(\lambda\xi)^+]<+\infty,~\forall~\lambda>0.
	\end{equation}
	Combining \eqref{eq:g_expect+} with \eqref{eq:g_expect-} yields $g(\lambda)=\mathbb{E}[\xi \underline{X}_B(\lambda\xi)]\in\mathbb{R}$. That is, Case 1 holds.
	
	For the well-definedness and finiteness of Problem \eqref{problem1}, suppose that we have some $X$ satisfying $\mathbb{E}[\xi X]\le x_0$, and $U(X,B)>-\infty$. We first derive from the definition of $\mathcal{X}_b(y)$ that
	\begin{equation}\label{eq:utineq}
		U(X,B)-\lambda\xi X\le U\left(\underline{X}_B(\lambda\xi),B\right)-\lambda\xi \underline{X}_B(\lambda\xi).
	\end{equation}
	Hence,
	\begin{equation}\label{eq:ut_estimate1}
		U(X,B)^+\le U\left(\underline{X}_B(\lambda\xi),B\right)^++\left(\lambda\xi X-\lambda\xi \underline{X}_B(\lambda\xi)\right)^+\le U\left(\underline{X}_B(\lambda\xi),B\right)^++\left|\lambda\xi X\right|+\left|\lambda\xi \underline{X}_B(\lambda\xi)\right|.
	\end{equation}
	Note that
	$
	\mathbb{E}\left[\left|\xi X\right|\right]=\mathbb{E}\left[\xi X^+\right]+\mathbb{E}\left[\xi X^-\right]=\mathbb{E}\left[\xi X\right]+2\mathbb{E}\left[\xi X^-\right]\le x_0+2\mathbb{E}\left[\xi \underline{x}_B^-\right]<+\infty.
	$
	Similarly, we have
	$
	\mathbb{E}\left[\left|\xi \underline{X}_B(\lambda\xi)\right|\right]\le g(\lambda)+2\mathbb{E}\left[\xi \underline{x}_B^-\right]<+\infty.
	$
	Taking the expectation in \eqref{eq:ut_estimate1}, we obtain
	\begin{equation}\label{eq:ut_estimate2}
		\mathbb{E}[U(X,B)^+]\le \mathbb{E}\left[U\left(\underline{X}_B(\lambda\xi),B\right)^+\right]+C(\lambda),
	\end{equation}
	where $C(\lambda)$ is a finite number depending on $\lambda$. We proceed to prove $\mathbb{E}[U(\underline{X}_B(\lambda\xi),B)^+]<+\infty$. Using (\ref{ineqA3}), 
	\begin{equation*}
		\begin{aligned}
			U(\underline{X}_B(\lambda\xi),B)^+\le u_1(B)K_3^{\delta}\big(1\!+\!\lambda^{-\frac{1}{1-\delta}}\xi^{-\frac{1}{1-\delta}}K_1(B)^{\frac{1}{1-\delta}}\!+\!\theta(B)\!+\!|\underline{x}_B|\!+\!K(B)\big)^{\delta}\!+\!u_2(B).
		\end{aligned}
	\end{equation*}
	For simplicity, we take $\lambda=1$ and derive for some $K_5>0$:
	\begin{equation}\label{ineqA4}
		\begin{aligned}
			U(\underline{X}_B(\xi),B)^+&\le\!\!\! &&K_5\bigg[u_1(B)+u_1(B)\theta(B)^{\delta}+u_1(B)K(B)^{\delta}+u_1(B)|\underline{x}_B|^{\delta}\\
			& &&\!+\xi^{-\frac{\delta}{1-\delta}}\bigg(u_1(B)^{\frac{1}{1-\delta}}+u_1(B)u_2(B)^{\frac{\delta}{1-\delta}}+u_1(B)|\gamma(B)|^{\frac{\delta}{1-\delta}}\bigg)\bigg]+u_2(B).
		\end{aligned}
	\end{equation}
	As $\mathbb{E}[\xi^{-\frac{\delta}{1-\delta}}u_1(B)^{\frac{1}{1-\delta}}]<+\infty$ and $\mathbb{E}[\xi]<+\infty$, using H{\"o}lder's inequality, we obtain
	$$
	\mathbb{E}[u_1(B)]=\mathbb{E}\big[\big(\xi^{-\frac{\delta}{1-\delta}}u_1(B)^{\frac{1}{1-\delta}}\big)^{1-\delta}\xi^\delta\big]<+\infty.
	$$
	Similarly, we can write the terms $u_2(B)$, $u_1(B)\theta(B)^{\delta}$, $u_1(B)K(B)^{\delta}$, $u_1(B)|\underline{x}_B|^{\delta}$, $\xi^{-\frac{\delta}{1-\delta}}u_1(B)u_2(B)^{\frac{\delta}{1-\delta}}$ and $\xi^{-\frac{\delta}{1-\delta}}u_1(B)|\gamma(B)|^{\frac{\delta}{1-\delta}}$ in interpolate forms and then obtain their integrability.
	Using \eqref{eq:ut_estimate2} and (\ref{ineqA4}), we have $\mathbb{E}\left[U(X,B)^+\right]<+\infty$, and  also
	\begin{equation*}
		\sup_{\mathbb{E}[\xi X]\le x_0\atop U(X,B)>-\infty}\mathbb{E}\left[U(X,B)\right]\le\mathbb{E}\left[U\left(\underline{X}_B(\xi),B\right)^+\right]+C(1)<+\infty.
	\end{equation*}
	That is, Problem \eqref{problem1} is well-defined and finite.
\end{proof}

\section{Proofs in Section \ref{section_general}}\label{app_general}
\subsection{Proof of Theorem \ref{thm_concave}}
\label{section_proof_concave}
We first propose a lemma.
\begin{lemma}\label{lemma2}
	Suppose that $Z,Z_1$ and $Z_2$ are finite random variables with $Z>0$, $Z_1\le Z_2$ a.s., and $A\in\mathcal{F}$ is a set with positive measure. If $\mathbb{E}[Z_1 Y^+ - Z_2 Y^-] \le 0$ holds for any bounded random variable $Y$ satisfying that $Y$ is supported on $A$, random variables $ZY, Z_1Y^+$ and $Z_2Y^-$ are integrable and $\mathbb{E}[ZY]=0$, then there exists a real number $\lambda$ such that $Z_1\le\lambda Z\le Z_2$ a.s. on $A$.
\end{lemma}
\begin{proof}[Proof of Lemma \ref{lemma2}.]
	
	\textbf{Step 1:} We consider the case when $Z,Z_1$ and $Z_2$ are bounded. First, we show by contradiction that $Z_2\ge\lambda Z$ holds almost surely on $A$ for some $\lambda\in\mathbb{R}$. Suppose that for any $\lambda\in\mathbb{R}$, the set $A_1=\{Z_2<\lambda Z\}\cap A$ has positive measure. Then we can take $\lambda$ such that $A_2=\{Z_2\ge\lambda Z\}\cap A$ also has a positive measure. As $\mathbb{P}[\{Z_2<\lambda'Z\}\cap A_1]>0$ for any $\lambda'$, we can define
	$$
	Y =
	\left\{
	\begin{aligned}
		& Y^+ = 1, \text{ on } A_2,\\
		& -Y^- = -\frac{\mathbb{E}[Z \id _{A_2}]}{\mathbb{E}[Z \id_{\{Z_2<\lambda'Z\}\cap A_1}]}\id_{\{Z_2<\lambda'Z\}\cap A_1}, \text{ on } A_1.
	\end{aligned}
	\right.
	$$
	It is verified that $Y$ is bounded and $Z Y,Z_1 Y^+$ and $Z_2 Y^-$ are integrable with $\mathbb{E}[ZY]=0$. We have
	\begin{equation*}
		\mathbb{E}[Z_1 \id_{A_2}]=\mathbb{E}[Z_1 Y^+] \le \mathbb{E}[Z_2 Y^-]=\mathbb{E}\left[Z_2 \frac{\mathbb{E}[Z \id_{A_2}]}{\mathbb{E}[Z \id_{\{Z_2<\lambda'Z\}\cap A_1}]} \id_{\{Z_2<\lambda'Z\}\cap A_1} \right] \le \lambda'\mathbb{E}[Z \id_{A_2}],
	\end{equation*}
	which contradicts to the finiteness of $\mathbb{E}[Z_1 \id_{A_2}]$ because $\lambda'$ can be any real number and $\mathbb{E}[Z \id_{A_2}]>0$. As a result, there exists $\lambda$ such that $Z_2\ge\lambda Z$ holds almost surely on $A$.
	
	\textbf{Step 2:} Now we show that we can choose $\lambda_0\in\mathbb{R}$ such that $Z_2\ge\lambda_0 Z\ge Z_1$ holds almost surely on $A$. Let $\lambda_0=\sup\{\lambda:Z_2\ge\lambda Z ~\text{a.s. on } A\}\in\mathbb{R}$, then $Z_2 \ge \lambda_0 Z$ a.s. on $A$. We prove by contradiction that for every $\lambda>\lambda_0$, $Z_1\le\lambda Z$ a.s. on $A$.
	
	If $\mathbb{P}[\{Z_1>\lambda Z\}\cap A]>0$ for some $\lambda>\lambda_0$, denote $A_3=\{Z_1>\lambda Z\}\cap A$. By definition of $\lambda_0$, we have $A_4=\{Z_2<\lambda Z\}\cap A$ with $\mathbb{P}[A_4]>0$. As $Z_2\ge Z_1$, we also have $A_3\cap A_4=\emptyset$. We can take $Y=Y^+ \id_{A_3} - Y^- \id_{A_4} \neq 0$ a.s. on $A_3\cup A_4$ with $\mathbb{E}[Z Y]=0$. It holds that $\mathbb{E}[Z_1Y^+]>\mathbb{E}[\lambda ZY^+]=\mathbb{E}[\lambda ZY^-]>\mathbb{E}[Z_2Y^-]$, which contradicts to the condition that $\mathbb{E}[Z_1 Y^+ - Z_2 Y^-] \le 0$. Therefore, for every $\lambda>\lambda_0$, $\mathbb{P}[\{Z_1>\lambda Z\}\cap A]=0$, which means that $Z_1\le\lambda_0 Z\le Z_2$ a.s. on $A$.
	
	\textbf{Step 3:} For the unbounded case,  define $A_n=A\cap\{Z,|Z_1|,|Z_2|<n\}$. Applying the results above, we have $\lambda_n$ such that $Z_1\le\lambda_n Z\le Z_2$ a.s. on $A_n$. By finiteness of $Z,Z_1,Z_2$ we know $-\infty<\liminf_{n\to+\infty}\lambda_n\le\limsup_{n\to+\infty}\lambda_n<+\infty$, and we find a convergent subsequence $\lambda_{n_i}\to\lambda$. Then we have $Z_1\le\lambda Z\le Z_2$ a.s. on $A$.
\end{proof}

\begin{proof}[Proof of Theorem \ref{thm_concave}.]
	For the ``if" statement, we will prove it in the second part of Theorem \ref{general thm} under a more general setting.
	
	For the ``only if" statement, a basic idea is to apply the variational method to derive an inequality \eqref{ineq} in Lemma \ref{lemma2} and we use the lemma to find a constant $\lambda$. Above all, for any optimal solution $X$, we know  $\mathbb{E}[\xi \underline{x}_B] \le \mathbb{E}[\xi X] \le x_0$. Next, we prove the ``only if'' part in two cases:
	\begin{itemize}
		\item If $\mathbb{E}[\xi \underline{x}_B]= x_0$, then $X=\underline{x}_B~\text{a.s.}$, that is, $X\in\mathcal{X}_B(\lambda\xi)$ with $\lambda=+\infty$.
		
		\item If $\E[\xi\underline{x}_B]<x_0$, then the set $A=\{X>\underline{x}_B\}$ satisfies $\p[A] > 0$. We are going to prove  $U'_+(X,B)\le\lambda\xi\le U'_-(X,B)$ on $A$ for some $\lambda\ge0$.
		Define
		$
		A_n=\left\{\xi\le n,~X-\underline{x}_B>\frac{1}{n},~U'_+\left(X-\frac{1}{n},B\right)<n\right\}.
		$
		We have $\mathbb{P}[A_n]>0$ for large $n$.
		Suppose that $Y$ is a bounded random variable supported on $A_n$ satisfying $\mathbb{E}[\xi Y]=0$. Noting that $U$ is concave, we have that, for $t$ sufficiently small,
		$
		\big|\frac{1}{t}\big(U(X+tY,B)-U(X,B)\big)\big|\le \big|U'_+ \big(X-\frac{1}{n}, B \big) Y\big| \text{ is bounded},
		$
		which implies that 
		\begin{equation}\label{ineq}
			0 \ge \lim\limits_{t \to 0+}\frac{1}{t}\big(\mathbb{E}[U(X+tY,B)]-\mathbb{E}[U(X,B)]\big)=\mathbb{E}[U'_+(X,B)Y^+-U'_-(X,B)Y^-].
		\end{equation}
		Using Lemma \ref{lemma2},
		we obtain $U'_+(X,B)\le\lambda_n\xi\le U'_-(X,B)$ on $A_n$ with some $\lambda_n\ge0$. As $A_n$ is increasing and converges to $A$, we can always find  a $\lambda<+\infty$ as a limit of subsequence of $\{
		\lambda_n\}$ such that $U'_+(X,B)\le\lambda\xi\le U'_-(X,B)$ holds on $A$ (if $\lambda=+\infty$, then $U'_-(X,B)=+\infty$ on $A$, which is a contradiction).
		
		If we have further  $\mathbb{P}[X=\underline{x}_B]>0$ (i.e., $\mathbb{P}[A^c]>0$), 
		we are going to prove  $U'_+(X, B) \le \lambda\xi \text{ a.s. on } A^c$. For a bounded random variable $Y$ satisfying $\mathbb{E}[\xi Y]=0$, $Y>0$ on $A^c$, $Y\le 0$ on $A_n$, and $Y=0$ on $A\cap A_n^c$, we have for $t$ sufficiently small:
		\begin{equation*}
			0\ge\lim\limits_{t\to0+}\frac{1}{t}\big(\mathbb{E}[U(X+tY,B)]-\mathbb{E}[U(X,B)]\big)=\mathbb{E}[U'_+(\underline{x}_B,B)Y1_{A^c}+U'_-(X,B)Y1_{A_n}].
		\end{equation*}
		Noting that we already have $U'_-(X,B)\ge \lambda_n\xi$ on $A_n$, where
		$$
		\lambda_n  := \sup\{\lambda:\lambda\xi\le U'_-(X,B) ~\text{ a.s. on } A_n\}\in\mathbb{R},
		$$
		we attempt to show $U'_+(\underline{x}_B,B)\le\lambda\xi$ a.s. on $A^c$ for any $\lambda>\lambda_n$. Otherwise for some $\lambda>\lambda_n$, we have $\mathbb{P}[U'_+(\underline{x}_B,B)\id_{A^c}>\lambda\xi]>0$. Then we can take $Y=Y^+>0$ on the set $C=\{U'_+(\underline{x}_B,B)>\lambda\xi\}\cap A^c\subset A^c$ and $Y=-Y^-<0$ on the set $D=\{U'_-(X,B)<\lambda\xi\}\cap A_n\subset A_n$ with $\mathbb{E}[\xi Y]=0$ to obtain
		$$
		0 > \mathbb{E}\left[\lambda\xi Y^+ \id_C-\lambda\xi Y^- \id_D\right] = \lambda\mathbb{E}[\xi Y] = 0,
		$$
		leading to a contradiction. Hence,  $U'_+(\underline{x}_B,B)\le\lambda_n\xi$ holds almost surely on $A^c$, while  $U'_+(X,B)\le\lambda_n\xi\le U'_-(X,B)~\text{a.s.}$ on $A_n$. Letting $n\to+\infty$, as $A_n$ is increasing and converges to $A$, we can always find  a $\lambda<+\infty$ as a limit of the increasing sequence $\{
		\lambda_n\}$ such that $U'_+(X,B)\le\lambda\xi\le U'_-(X,B)$ holds on $A$. We have
		$$
		\left\{
		\begin{aligned}
			& U'_+(\underline{x}_B, B) \le \lambda\xi \text{ a.s. on } A^c,\\
			& U'_+(X, B) \le \lambda\xi \le U'_-(X,B) \text{ a.s. on } A.
		\end{aligned}
		\right.
		$$
		In conclusion of (1)-(3), we know that, for a finite optimal solution $X$, there exists $\lambda\in[0,+\infty]$ such that
		\begin{equation}\label{concave arg}
			\left\{
			\begin{aligned}
				& U'_+(\underline{x}_B, B) \le \lambda\xi \text{ a.s. on the set } \{X=\underline{x}_B\},\\
				& U'_+(X, B) \le \lambda\xi \le U'_-(X,B) \text{ a.s. on the set } \{X>\underline{x}_B\}.
			\end{aligned}
			\right.
		\end{equation}
		
		This means  $X\in \mathcal{X}_B(\lambda\xi)$ almost surely.
	\end{itemize}
\end{proof}

\subsection{Technical discussion}
	 To prove Theorems \ref{general thm}-\ref{main} under weaker conditions, we provide the following Assumptions \ref{ass_G} and \ref{asmp:Hn}, which cover Assumption \ref{ass_standing} and generalize our results. The weaker settings include more types of utilities.
	
	\begin{assumption}\label{ass_G}
		$U(x,\cdot)$ is measurable on $(E,\mathcal{E})$ for any $x\in\mathbb{R}$ and $U(\cdot, b)\in \mathcal{H}$ for any $b \in E$, where $\mathcal{H}$ is the set of all of the function $h:\mathbb{R}\to\mathbb{R}\cup\left\{-\infty\right\}$ satisfying
		\begin{enumerate}[(A)]
			\item $h$ is nondecreasing and upper semicontinuous on $\mathbb{R}$;
			\item There exists $x\in\mathbb{R}$ such that $h(x)>-\infty$;
			\item $h$ admits a concavification $\tilde{h} \triangleq \inf\left\{g: \R \rightarrow \R \cup \{-\infty\} ~|~ g \text{ is concave and } g \geq h \right\}$.
			\item
			For any $n \in \N_+$, define
			\begin{equation}\label{def:HG}
				\begin{aligned}
					H_n(t)=\sup\left\{x\le t:\tilde{h}(x)\le h(x)+\frac{1}{n}\right\},~G_n(t)=\inf\left\{x\ge t:\tilde{h}(x)\le h(x)+\frac{1}{n}\right\}, ~t\in\mathbb{R}.
				\end{aligned}
			\end{equation}
			For any $n \in \N_+$ and $t\in\mathbb{R}$, it holds that $H_n(t) > -\infty$ and $G_n(t) < +\infty$.
		\end{enumerate}
	\end{assumption}
	
	If (C) and (D) hold, it means that the concavification $\tilde{h}$ exists and has a relatively close distance with $h$. 
	For $U$ satisfying Assumption \ref{ass_G},  we get a ``family" of state-dependent concavifications: $\tilde{U}(\cdot,b)$ for all $b \in E$.  
	We further define $H_n(t,b)$ and $G_n(t,b)$ as the state-dependent version of $H_n(t)$ and $G_n(t)$:
	\begin{equation}\label{eq:HnGn}
		\begin{aligned}
			H_n(t,b)=\sup\left\{x\le t:\tilde{U}(x,b)\le U(x,b)+\frac{1}{n}\right\},G_n(t,b)=\inf\left\{x\ge t:\tilde{U}(x,b)\le U(x,b)+\frac{1}{n}\right\}.
		\end{aligned}
	\end{equation}
	{\color{black}
	Moreover, we define
	$
	\hat{H}_b^{(n)}\triangleq\inf\{t\in\R: H_n(t,b)<t\}=\inf\{t\in\R:\tilde{U}(t,b)> U(t,b)+\frac{1}{n}\}.
	$
	As we no longer require $\underline{x}_b>-\infty$, $\hat{H}_b^{(n)}$ becomes a substitute of $\underline{x}_b$. The definition of $\hat{H}_b^{(n)}$ also indicates that $\hat{H}_b^{(n)}\ge\underline{x}_b$. As a substitute of the assumption $\E\[\xi\underline{x}_b\]>-\infty$ we propose the following Assumption \ref{asmp:Hn}:
	}
	\begin{assumption}
		\label{asmp:Hn}
		$\forall~n\in\N_+,~\E\[\xi\hat{H}_B^{(n)-}\]<+\infty$. 	
	\end{assumption}
	{\color{black}
	To conclude, as a weaker version of Assumption \ref{ass_standing}, the new Assumptions \ref{ass_G} and \ref{asmp:Hn} require that the concavification $\tilde{h}$ can be as close to $h$ as we want for every $h=U(\cdot,b)$. An ideal case is that for every $h=U(\cdot,b)$, $\tilde{h}$ coincides with $h$ near $\pm\infty$. Under new assumptions, $h$ can have a finite value near $-\infty$ and can have a positive slope near $+\infty$, which covers a lot of functions that do not satisfy Assumption \ref{ass_standing}.
	}
	We proceed to prove in Lemma \ref{prop:suff} that the range of $\mathcal{H}$ includes that of Assumption \ref{ass_standing}.
	
	\begin{lemma}\label{prop:suff}
		If $U$ satisfies Assumption \ref{ass_standing}, then Assumptions \ref{ass_G}-\ref{asmp:Hn} hold.
	\end{lemma}

	To prove Lemma \ref{prop:suff}, we first propose Lemma \ref{concavification1}, which demonstrates properties of the concavification. 
	\begin{lemma}\label{concavification1}
		For a nondecreasing and upper semicontinuous function 
		$h:\mathbb{R}\to\mathbb{R}\cup\{-\infty\}$, suppose  $h(x)>-\infty$ for some $x\in\mathbb{R}$, and that $h$ admits a concavification $\tilde{h}:\mathbb{R}\to\mathbb{R}\cup\{-\infty\}$. Then:
		
		(1) $\tilde{h}=h$ on $(-\infty,\underline{x})$, and $\tilde{h}$ is continuous on $(\underline{x},+\infty)$. Moreover, $\tilde{h}$ is nondecreasing on $\mathbb{R}$.
		
		(2) For $a, b \in \R$ and $a < b$, if $\tilde{h} > h$ on $(a,b)$, then $\tilde{h}$ is affine on $(a,b)$.

		
		(3) If $\underline{x}>-\infty$, then $\tilde{h}(\underline{x})=h(\underline{x})$, and $\tilde{h}$ is right-continuous at $\underline{x}$. Moreover, if $h(\underline{x})=-\infty$, then there exists $\{x_n\}_{n \geq 1} \subset (\underline{x},+\infty)$ with $x_n\downarrow\underline{x}$ and $\tilde{h}(x_n)=h(x_n)$ (This means that the function $H_n(t)$ defined in \eqref{def:HG} is always finite, and for $\tilde{h}(t)>h(t)$ we have $h(H_n(t))>-\infty$).
		
		

		(4) If $\limsup\limits_{x\to+\infty}\frac{h(x)}{x}=0$, then the function $G_n(t)$ defined in \eqref{def:HG} is always finite.
		
		(5) If $h\in\mathcal{H}$ and $\tilde{h}(t)>h(t)+\frac{1}{n}$ for some $n\in\mathbb{N}_+,~t\in\mathbb{R}$, then $H_n(t)<t<G_n(t)$, and $-\infty<\tilde{h}(H_n(t))\le h(H_n(t))+\frac{1}{n}$,~$\tilde{h}(G_n(t))\le h(G_n(t))+\frac{1}{n}$.
		
		(6) If $h\in\mathcal{H}$, then $\tilde{h}(x)=\sup\limits_{(a,b) \in \R^2: a \le x, b \geq x, a\not= b}\frac{(x-a)h(b)+(b-x)h(a)}{b-a}$.
	\end{lemma}
	\begin{proof}[Proof of Lemma \ref{concavification1}.]
		We point out that a nondecreasing upper semi-continuous function is right-continuous.
		(1) is trivial by the definition of the concavification and properties of concave functions. We prove (2)-(6).
		
		(2) As $\tilde{h}(x)=h(x)=-\infty$ on $(-\infty,\underline{x})$, we know $a\ge\underline{x}$. Then for any $[a',b']\subset(a,b)$, as $h(x)$ is upper semicontinuous and $\tilde{h}(x)$ is continuous, $\tilde{h}(x)-h(x)$ admits a minimum value $\epsilon>0$ on $[a',b']$.
		If $\tilde{h}$ is not affine on $[a',b']$, then  on $[a',b']$ we can find one of its linear interpolation $\hat{h}\not=\tilde{h}$ such that $0\le \tilde{h}(x)-\hat{h}(x)\le\frac{\epsilon}{2}$. Define $\hat{h}=\tilde{h}$ on $\mathbb{R}\backslash[a',b']$.
		As such, $\hat{h}$ is concave and satisfies $\hat{h}(x)\ge h(x)$, while $\hat{h}(x)<\tilde{h}(x)$ at some points of $[a',b']$, which contradicts to the fact that $\tilde{h}$ is the concavification of $h$.
		Thus $\tilde{h}$ is affine on any $[a',b']\subset(a,b)$, and then affine on $(a,b)$.
		
		(3) \textbf{Step 1:} We prove the second assertion by contradiction. If there exists $\delta>\underline{x}$ such that $\tilde{h}(x)>h(x)$ on $(\underline{x},\delta)$, then $\tilde{h}$ is affine on $(\underline{x},\delta)$, and we can assume $\tilde{h}(x)=\alpha x+\beta$. As $\underline{x}>-\infty$, we have $\lim\limits_{x\to \underline{x}+}\tilde{h}(x)-h(x)=+\infty$, that is, one can find $\epsilon\in(\underline{x},\delta)$ such that $\tilde{h}(x)-h(x)>1$ for all $x\in(\underline{x},\epsilon)$. Define
		\begin{equation*}
			\hat{h}(x)=\left\{
			\begin{aligned}
				&(\alpha+\frac{1}{\epsilon-\underline{x}})(x-\epsilon)+\alpha\epsilon+\beta,&&x\in[\underline{x},\epsilon],\\
				&\tilde{h}(x),&&x\in(\epsilon,+\infty),\\
				&-\infty,&&x\in(-\infty,\underline{x}).
			\end{aligned}\right.
		\end{equation*}
		Then
		$\hat{h}$ is a concave function, as we have just turned up the slope of $\tilde{h}$ on $[\underline{x},\epsilon]$. On $(\underline{x},+\infty)$ we have
		$
		\hat{h}(x)-\tilde{h}(x)=\frac{x-\epsilon}{\epsilon-\underline{x}}\id_{(\underline{x},\epsilon)}(x)\in(-1,0],
		$
		as such, $h(x)<\tilde{h}(x)-1<\hat{h}(x)\le\tilde{h}(x)$ on $(\underline{x},\epsilon)$ and $h(x)\le\tilde{h}(x)=\hat{h}(x)$ on $[\epsilon,+\infty)$.  At the point $\underline{x}$, as $h$ is right-continuous, we know $h(\underline{x})\le\hat{h}(\underline{x})$. Thus, the function $\min\{\hat{h},\tilde{h}\}$ is a convex function that is not smaller than $h$, and is smaller than $\tilde{h}$ at some points, which contradicts to the fact that $\tilde{h}$ is the concavification of $h$.
		
		\textbf{Step 2:} We prove the first assertion.
		
		When $h(\underline{x})=-\infty$, as $\tilde{h}$ is nondecreasing, using the second assertion proved above, we know that $\tilde{h}$ is right-continuous at $\underline{x}$ and $\tilde{h}(\underline{x})=-\infty$. As such, it remains to study the case when $h(\underline{x})>-\infty$. If $\tilde{h}(\underline{x})>h(\underline{x})$, as $h$ is right-continuous and that $\tilde{h}$ is nondecreasing, there exist $\delta>\underline{x}$ and $\epsilon>0$ such that $\tilde{h}(x)>h(x)+\epsilon$ holds on $[\underline{x},\delta]$. Using the same methods as in (2), one can get a contradiction. Hence, $\tilde{h}(\underline{x})=h(\underline{x})$.
		
		If $\tilde{h}$ is not right-continuous at $\underline{x}$, then $h(\underline{x})=\tilde{h}(\underline{x})<\lim\limits_{x\to\underline{x}+}\tilde{h}(x)$, noting that $h$ is right-continuous, again we have a $\delta>\underline{x}$ and $\epsilon>0$ such that $\tilde{h}(x)>h(x)+\epsilon$ holds on $[\underline{x},\delta]$. The proof then follows.

		(4) \textbf{Case 1:} If there exists $x_1\ge\underline{x}$ such that $\tilde{h}(x_1)=h(x_1)>-\infty$. Suppose that $\tilde{h}>h$ on $(x_1,+\infty)$, then $\tilde{h}$ is affine on $(x_1,+\infty)$ with $\tilde{h}(x)=\alpha x+\beta$, $\alpha\ge0$.
		
		If $\alpha=0$, then for $x>x_1$, we have $h(x)< \tilde{h}(x)=\tilde{h}(x_1)=h(x_1)$, which contradicts to the fact that $h$ is nondecreasing.
		
		If $\alpha>0$, noting that $\liminf\limits_{x\to+\infty}\frac{h(x)}{x}\ge0$, we have $\lim\limits_{x\to+\infty}\frac{h(x)}{x}=0$. For $x$ sufficiently large, we have $h(x)<\frac{\alpha}{2}x+\beta$. In this case, we can replace the right tail of $\tilde{h}$ by a linear function with a lower slope, which leads to a contradiction.
		
		Therefore, we have $x_2\in(x_1,+\infty)$ satisfying $\tilde{h}(x_2)=h(x_2)$. In this light, if $\sup\{x:\tilde{h}(x)=h(x)\}=x^*<+\infty$, as that $\tilde{h}$ is continuous on $(\underline{x},+\infty)$ and that $h$ is nondecreasing, we have $\tilde{h}(x^*)\le h(x^*)\le\tilde{h}(x^*)$. Hence $\tilde{h}(x^*)=h(x^*)$. Then we have $x'\in(x^*,+\infty)$ such that $\tilde{h}(x')=h(x')$, which is a contradiction, and we know $\sup\{x:\tilde{h}(x)=h(x)\}=+\infty$. Therefore, $G_n(t)$ is finite.
		
		\textbf{Case 2:} If for every $x_1\ge x$ we have $\tilde{h}(x_1)>h(x_1)$ or $\tilde{h}(x_1)=h(x_1)=-\infty$, then, based on (3), we know $\underline{x}=-\infty$. Hence $\tilde{h}>h$ on $\mathbb{R}$. As such, $\tilde{h}(x)$ is affine on $\mathbb{R}$. Assume $\tilde{h}(x)=\alpha x+\beta$. Similar as in \textbf{Case 1}, $\alpha=0$, and $\tilde{h}$ is a constant $\beta$ on $\mathbb{R}$. As $\tilde{h}$ is the concavification of $h$, for every $n\in\mathbb{N_+}$ and $x\in\mathbb{R}$, there exists $x_n>x$ such that $h(x_n)>\beta-\frac{1}{n}=\tilde{h}(x_n)-\frac{1}{n}$. Therefore, $G_n(t)$ is finite.
		
		(5) For $\tilde{h}(t)>h(t)+\frac{1}{n}$, if $\underline{x}=-\infty$, then as $H_n(t)>-\infty$ we know $h(H_n(t))>-\infty$. If $\underline{x}>-\infty$, then using (3) we know that there exists $t'<t$ such that $\tilde{h}(t')=h(t)>-\infty$. Then $H_n(t)\ge t'$ and $h(H_n(t))>-\infty$. The remaining assertions can be derived directly from the (right-)continuity of $\tilde{h}$ and $h$.
		
		(6) It only needs to prove for $x\ge\underline{x}$. As $\tilde{h}$ is concave, we have
		\begin{equation*}
			\tilde{h}(x)\ge\frac{(b-x)\tilde{h}(a)+(x-a)\tilde{h}(b)}{b-a}\ge\frac{(b-x)h(a)+(x-a)h(b)}{b-a}.
		\end{equation*}
		For $\forall~n\in\mathbb{N}_+$, if $\tilde{h}(x)>h(x)+\frac{1}{n}$, based on the fact that $h\in\mathcal{H}$, we find $a<x<b$ with $\tilde{h}$ linear on $[a,b]$, $\tilde{h}(x)>h(x)$ on $(a,b)$ and $h(a)\ge\tilde{h}(a)-\frac{1}{n}$, $h(b)\ge\tilde{h}(b)-\frac{1}{n}$. As such,
		\begin{equation*}
			\tilde{h}(x)=\frac{(b-x)\tilde{h}(a)+(x-a)\tilde{h}(b)}{b-a}\le\frac{(b-x)h(a)+(x-a)h(b)}{b-a}+\frac{1}{n}.
		\end{equation*}
		If  $\tilde{h}(x)\le h(x)+\frac{1}{n}$, we take $a=x<b$ and obtain the same inequality. The statement follows.
		
	\end{proof}
	


	\begin{proof}[Proof of Lemma \ref{prop:suff}.]
		We first prove that Assumption \ref{ass_G} holds, that is, we prove that $U(\cdot,b)$ satisfies (A)-(D) for any $b\in E$. (A) and (B) are obvious based on (a) and (b). For (C), noting that $\limsup\limits_{x\to+\infty}\frac{U(x,b)}{x}=0$ for any given $b$, there exist $\alpha\in\R$ and $\beta\in\mathbb{R}$ (depending on $b$) such that $U(x,b)<\alpha x+\beta$ for $x\ge\underline{x}_b>-\infty$. As such, $U(\cdot,b)$ is dominated by a concave function. Hence $U(\cdot,b)$ admits a concavification $\tilde{U}(\cdot,b):\mathbb{R}\to\mathbb{R}\cup\{-\infty\}$. Using Lemma \ref{concavification1}(3)-(4) we know that (D) holds. Then we prove that Assumption \ref{asmp:Hn} holds. In fact, based on the definition of $\hat{H}_b^{(n)}$ we know $\hat{H}_b^{(n)}\ge\underline{x}_b$. Hence
		$
		\E\[\xi\hat{H}_B^{(n)-}\]\le\E\[\xi\underline{x}_B^-\]<+\infty.
		$
	\end{proof}

	Indeed, our setting involves a rather abstract set $\mathcal{H}$, which takes the case that
	$\limsup\limits_{x\to +\infty}\frac{U(x,b)}{x}>0$ or $\underline{x}_b=-\infty$ under consideration. That is, the utility function is allowed to have a positive slope at $+\infty$, and is allowed to be defined near $-\infty$.
	To help understand the most essential conditions (C) and (D), we illustrate $H_n(t,b)$ and $G_n(t,b)$ in Figure \ref{fig1}. 

\begin{figure}[htbp] 
	\centering
	\includegraphics[width = 0.57\textwidth]{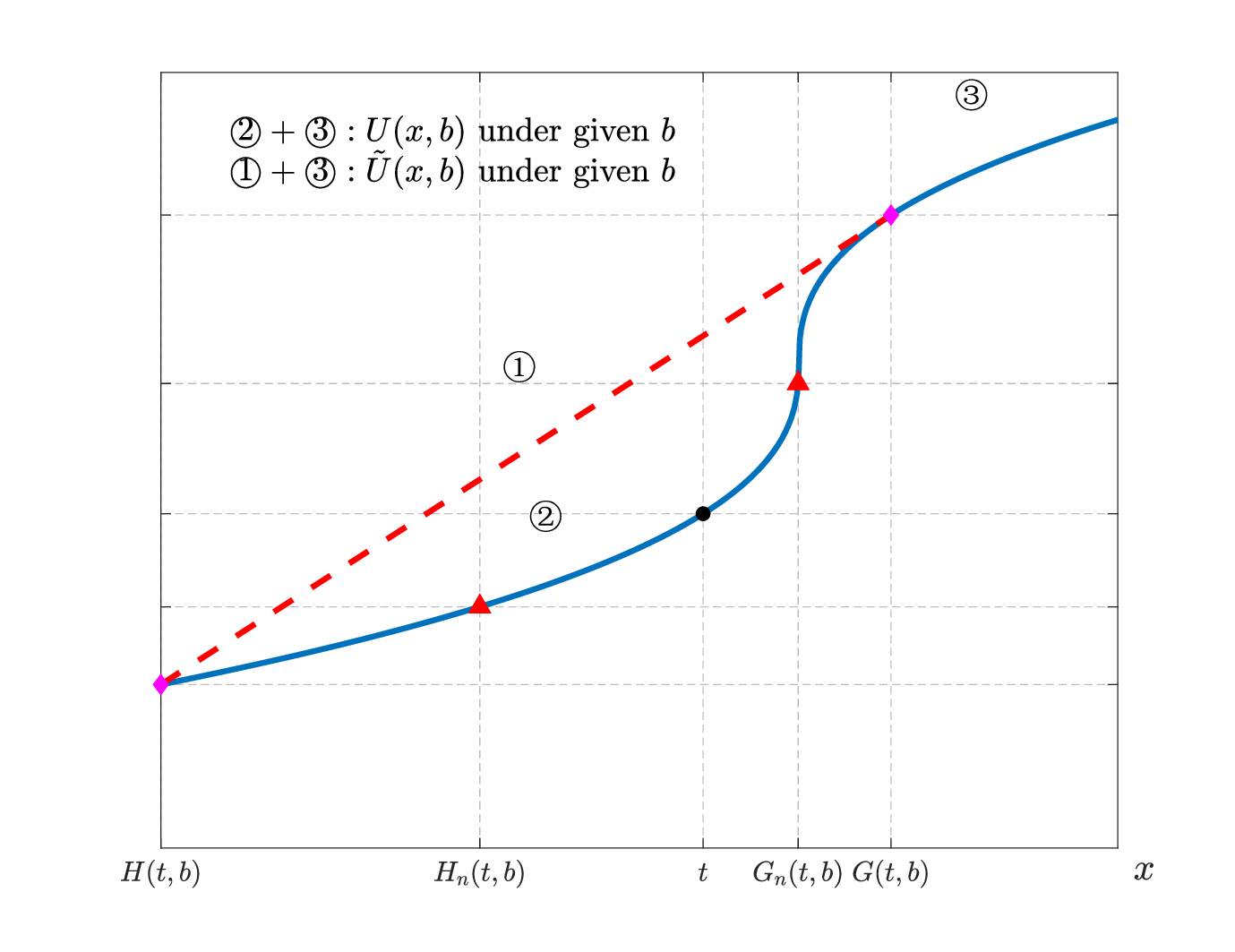}
	\caption{Illustration for functions $G$, $H$, $G_n$ and $H_n$}
	\label{fig1}
\end{figure}

\subsection{Proof of Theorem \ref{general thm}}\label{proof3}
We prove Theorem \ref{general thm} with Assumptions \ref{ass_standing} replaced by weaker Assumptions \ref{ass_G} and \ref{asmp:Hn}.
To prove Theorem \ref{general thm}, we need some further discussion on concavification (Lemma \ref{concavification1}) and non-atomic measures. $H_n$ and $G_n$ are important tools in the proof of Theorem \ref{general thm}, and it is necessary to confirm their measurability before we apply mathematical operations on them, which is stated in the following Lemma \ref{concavification2}. Its proof is rather technical.

\begin{lemma}\label{concavification2}
	Under Assumption \ref{ass_G}, $\tilde{U}$, $H_n$ and $G_n$ are measurable functions on $(\R \times E, \B(\R) \times \e)$.
\end{lemma}

\begin{proof}{Proof of Lemma \ref{concavification2}.}
	(1) For $\tilde{U}$, thanks to Lemma \ref{concavification1}, we have
	\begin{equation*}
		\begin{aligned}
			\tilde{U}(x,b)&=\mathop{\sup\limits_{a\le x\le c}}\limits_{a\not= c}\frac{(x-a)U(c,b)+(c-x)U(a,b)}{c-a}
			&=\sup\limits_{a',c'\in[0,+\infty]\cap\mathbb{Q}\atop (a',c')\not= (0,0)}\frac{a'U(x+c',b)+c'U(x-a',b)}{a'+c'},
		\end{aligned}
	\end{equation*}	
	which indicates that $\tilde{U}$ is measurable.
	
	(2) Define $A_1 = \{(t,b):H_n(t,b)<s\}$ and
	$$
	A_2 = \big\{(t,b) \in \R \times E:t\ge s \text{~and~}\tilde{U}(x,b)>U(x,b)+\frac{1}{n}~\text{for every}~x\in[s,t]\big\}\cup((-\infty,s)\times E).
	$$
	
	We first show that $A_1 = A_2$. It is obvious that $A_1\subset A_2$, and we prove $A_2\subset A_1$. As $t<s$ leads to $H_n(t,b)<s$, it suffices to consider the case that $t\ge s$. If $\tilde{U}(x,b)>U(x,b)+\frac{1}{n}~\text{holds for every}~x\in[s,t]$, then $s>\underline{x}_b$, and by definition of $H_n$ we know $H_n(t,b)\le s$, and when the equality holds, we have $x_k\uparrow s$ satisfying $\tilde{U}(x_k,b)\le U(x_k,b)+\frac{1}{n}$. Therefore, we derive a contradiction that $\tilde{U}(s,b)>U(s,b)+\frac{1}{n}\ge\varlimsup\limits_{x_k\uparrow s}U(x_k,b)+\frac{1}{n}\ge\varlimsup\limits_{x_k\uparrow s}\tilde{U}(x_k,b)=\tilde{U}(s,b)$. As such, $A_2\subset A_1$, and hence $A_1=A_2$.
	
	Now we investigate the structure of $A_2$. Define
	$$
	A_3 = \big\{(t,b):t\ge s \text{~and~}\tilde{U}(x,b)>U(x,b)+\frac{1}{n}~\text{for every}~x\in[s,t]\big\}.
	$$
	We are going to prove
	\begin{equation}\label{reform1}
		A_3=\mathop{\bigcup}_{j\ge 1}\mathop{\bigcup}_{x\in\mathbb{Q},x\ge s}\bigg([s,x]\times\mathop{\bigcap}_{N\ge 1}F_j\left(z_1^{(N)}(x),z_2^{(N)}(x),...,z_{N+1}^{(N)}(x)\right)\bigg),
	\end{equation}
	where $F_j (z_1,z_2,...,z_m)\triangleq\mathop{\bigcap}\limits_{1\le k\le m}\left\{b\in E:~\tilde{U}(z_k,b)>U(z_k,b)+\frac{1}{n}+\frac{1}{j}\right\}$, and $\left\{z_i^{(N)}(x)\right\}_{1\le i\le N+1}$ denotes the $N$-uniform partition points of the interval $[s,x]$. In fact, if we denote the right side of (\ref{reform1}) by $A_4$, then for $(t,b)\in A_3$, we know that $\tilde{U}(x,b)>U(x,b)+\frac{1}{n}$ holds for every $x\in[s,t]$, and hence $t\ge s>\underline{x}_b$. Define $\theta=\inf\left\{x\in[s,t]:\tilde{U}(x,b)-U(x,b)\right\}$. We have $\{y_n\}\subset[s,t]$ satisfying $y_n\to y$ and $\tilde{U}(y_n,b)-U(y_n,b)\to\theta$. As $\tilde{U}(\cdot,b)$ is continuous on $(\underline{x}_b,+\infty)$ and that $U(\cdot,b)$ is upper semicontinuous,
	\begin{equation*}
		\theta=\lim_{n\to+\infty}\left(\tilde{U}(y_n,b)-U(y_n,b)\right)\ge\tilde{U}(y,b)-U(y,b)\ge\theta,
	\end{equation*}
	thus $\theta=\tilde{U}(y,b)-U(y,b)>\frac{1}{n}$, and we have some $j\in\mathbb{N}_+~s.t.~\theta>\frac{1}{n}+
	\frac{1}{j}$. Again using the continuity of $\tilde{U}(\cdot,b)$ and upper semicontinuity of $U(\cdot,b)$, we know  $\tilde{U}(x,b)>U(x,b)+\frac{1}{n}+\frac{1}{2j}$ also holds on $x\in[s,t+\delta]$ for some $\delta>0$, that is, one can find $t'\in\mathbb{Q}$ such that $t'\ge t$ and $[s,t']\times\{b\}\subset A_3$. Based on definition of $F_j$,  we know that $b\in F_{2j}\left(z_1^{(N)}(t'),z_2^{(N)}(t'),...,z_{N+1}^{(N)}(t')\right)$ for any $N$, thus
	\begin{equation*}
		(t,b)\in[s,t']\times\{b\}\subset[s,t']\times\mathop{\bigcap}_{N\ge 1}F_{2j}\left(z_1^{(N)}(t'),z_2^{(N)}(t'),...,z_{N+1}^{(N)}(t')\right)\subset A_4,
	\end{equation*}
	and we know $A_3\subset A_4$. It suffices to show $A_4\subset A_3$.
	
	For $(t,b)\in A_4$, $(t,b)\in[s,x]\times\mathop{\bigcap}\limits_{N\ge 1}F_j\left(z_1^{(N)}(x),z_2^{(N)}(x),...,z_{N+1}^{(N)}(x)\right)$ with some $x\in\mathbb{Q}$ and $j\in\mathbb{N}_+$, that is, $s\le t\le x$ and, $\tilde{U}\left(z_k^{(N)}(x),b\right)>U\left(z_k^{(N)}(x),b\right)+\frac{1}{n}+\frac{1}{j}$ for any $k,~N$ satisfying $1\le k\le N+1$. Noting that $\tilde{U}(\cdot,b)$ is continuous and that $U(\cdot,b)$ is right-continuous, we have $\tilde{U}(x,b)\ge U(x,b)+\frac{1}{n}+\frac{1}{j}>U(x,b)+\frac{1}{n}$ on $[s,x]\supset [s,t]$. It follows that $A_4\subset A_3$, which leads to $A_3=A_4$.
	
	As for a given $z$, $\tilde{U}(z,b)-U(z,b)$ is measurable in $b$, we know
	$
	F_j\left(z_1^{(N)}(x),z_2^{(N)}(x),...,z_{N+1}^{(N)}(x)\right)\in\mathcal{E}.
	$
	Therefore, $A_3=A_4\in\mathcal{B}(\mathbb{R})\times \e$, and $A_1=A_2=A_3\cup((-\infty,s)\times E)$ is also a measurable set, which means that $H_n$ is a measurable function. Similarly, one can prove the measurability of $G_n$.
\end{proof}

At last, for non-atomic measures, we need the following result in \citet{S1922}:
\begin{lemma}\label{prop1}
	If $\mu$ is a non-atomic measure on $(\Omega,\mathcal{F})$ with $\mu(\Omega)=c$, then there exists a one-parameter family of increasing measurable sets $\{A_t\}_{0\le t\le c}$ such that $\mu (A_t)=t$.
\end{lemma}

In light of Lemmas \ref{concavification1}-\ref{prop1} above, we proceed to prove Theorem \ref{general thm}. The idea is to substitute $X(\omega)$ locally by some $\hat{H}(\omega)$ and $\hat{G}(\omega)$ for some $\omega$ such that $X(\omega)$ does not lie on the concave part of $U$ and $\hat{H}\le X\le \hat{G}$. We proceed to find a proper substitution of $X$ with the form
$
\hat{X}=\hat{H} \id_C + \hat{G} \id_D + X \id_{\left(C\cup D\right)^c},
$
which increases the utility, while the budget value $\mathbb{E}[\xi X]$ keeps unchanged. Figure \ref{fig1} is provided to assist in understanding.
Define
$$
H(t,b)=\sup\left\{x\le t:\tilde{U}(x,b)= U(x,b)\right\},
G(t,b)=\inf\left\{x\ge t:\tilde{U}(x,b)= U(x,b)\right\}.
$$
In the following proof, we assume that $H(t,b)$ and $G(t,b)$ are finite, and $\tilde{U}(H(t,b),b)=U(H(t,b),b)>-\infty$, $\tilde{U}(G(t,b),b)=U(G(t,b),b)$ when $\tilde{U}(t,b)>U(t,b)$. Then, based on Lemma \ref{concavification2}, we know that $H(t,b)$ and $G(t,b)$ are measurable. Moreover, we define
$\hat{H}_b=\inf\{t\in\R:H(t,b)<t\}$ and assume $\E\[\xi\hat{H}_B^-\]<+\infty$. These conditions provide a concise proof. For the weaker case with only $H_n(t,b)$ and $G_n(t,b)$ being finite and  Assumption \ref{asmp:Hn}, the proof is similar;  see also Remark \ref{rmk:HnGn}.
\begin{proof}[Proof of Theorem \ref{general thm}.]
	\begin{enumerate}[(i)]
		\item \textbf{Part I:}
		Let us first assume that the concavification problem is well-defined.
		It is equivalent to study both problems with binding budget constraints, i.e., $\E[\xi X] = x_0$. We are going to prove \eqref{eq:equivalence} by contradiction. Suppose 
		\begin{equation}\label{def:beta}
			\alpha \triangleq \sup_{X: \E[\xi X] = x_0\atop U(X,B)>-\infty}\mathbb{E}[\tilde{U}(X,B)] > \beta \triangleq \sup_{X: \E[\xi X] = x_0\atop \tilde{U}(X,B)>-\infty}\mathbb{E}[U(X,B)].
		\end{equation}
		As such, we have some random variable $X$ satisfying $\E[\xi X] = x_0$ and
		$
		\mathbb{E}[\tilde{U}(X,B)]>\beta.
		$
		Define
		$$
		S=\left\{ (x,b)\in\R\times E: \tilde{U}(x,b) > U(x,b)  \right\},~~ Q=\left\{\omega \in \Omega: (X(\omega),B(\omega)) \in S \right\}.
		$$
		It follows that $\mathbb{P}[Q]>0$. Define two random variables
		$
		\hat{H}=H(X,B),~ \hat{G}=G(X,B).
		$
		On $Q$, we have $\hat{H}<X<\hat{G}$. Hence $\tilde{U}(x,B)$ is affine in $x$ for $\hat{H}<x<\hat{G}$ and $\tilde{U}(x,B)=a_Bx+c_B$, where
		$
		a_B=\frac{\tilde{U}(\hat{G},B)-\tilde{U}(\hat{H},B)}{\hat{G}-\hat{H}}.
		$
		Based on Lemmas \ref{concavification1} and \ref{concavification2}, $a_B$ and $c_B$ are measurable on $Q$, and
		$\hat{H}$, $\hat{G}$, $a_B$ and $c_B$ are finite.
		In the following Steps 1-3, we desire to construct a new random variable $\hat{X}$ by slightly changing $X$ on $Q$ satisfying $\E[\xi \hat{X}] = x_0$ and $\E[U(\hat{X}, B)] > \beta$, which is a contradiction.
		
		\textbf{Step 1:} Take a countable partition of $\mathbb{R}^5$ as
		$
		\mathbb{R}^5=\mathop{\bigcup}_{n\ge1}P_n,~ \{P_n\}_{n\ge1} \text{ are bounded sets}.
		$
		This leads to a partition of $Q$ as
		$
		Q=\mathop{\bigcup}_{n\ge1}Q_n,~ Q_n\triangleq Q\cap \left\{\omega\in\Omega: (\hat{H},\hat{G},a_B,c_B,\xi)\in P_n\right\},
		$
		and we proceed to make our substitution of $X$ on every $Q_n$ that $\mathbb{P}\left[Q_n\right]>0$.
		
		\textbf{Step 2:}  
		For every $n$ with $\mathbb{P}\left[Q_n\right]>0$, we need to determine a partition $Q_n = C_n \cup D_n$ and then we replace $X$ by $\hat{H}$ on the set $C_n$ and by $\hat{G}$ on the set $D_n$. We need to make sure that the new variable $\hat{X}$ satisfies
		\begin{equation}\label{eq:budget}
			\mathbb{E}\[\xi \hat{X}\id_{Q_n}\] = \E\[\xi X\id_{Q_n}\]
		\end{equation}
		and
		\begin{equation}\label{eq:ut_increase}
			\mathbb{E}\left[\tilde{U}(X,B)\id_{Q_n}\right]\le\mathbb{E}\left[\tilde{U}(\hat{X},B)\id_{Q_n}\right]=\mathbb{E}\left[U(\hat{X},B)\id_{Q_n}\right].
		\end{equation}
		To this end, we denote $p=\mathbb{E}[\xi X \id_{Q_n}]$. Noting that $\hat{H}$, $\hat{G}$, $a_B$, and $\xi$ are bounded on $Q_n$, we define 
		$
		F(t)=\mathbb{E}\left[\xi G \id_{\left\{\xi\le t a_B\right\} \cap Q_n}\right]+\mathbb{E}\big[\xi \hat{H} \id_{\left\{\xi> ta_B\right\} \cap Q_n}\big],~t\ge0.
		$
		
		As $\hat{H}<X<G$ holds on $Q$, we have $F(0)=\mathbb{E}\left[\xi \hat{H} \id_{Q_n}\right]<p$, $F(+\infty)=\mathbb{E}[\xi G \id_{Q_n}]>p$, and that $F$ is nondecreasing and right-continuous. In the following, we construct $C_n$ and $D_n$ through two cases such that
		$
		\E[\xi (\hat{H} \id_{C_n} + \hat{G} \id_{D_n})] = p.
		$
		\begin{itemize}
			\item If $F(\sigma)=p$ holds for some $\sigma>0$, we define
			$
			C_n=\left\{ \xi> \sigma a_B \right \} \cap Q_n,~D_n=\{\xi\le \sigma a_B\}\cap Q_n,
			$
			and then we have $\mathbb{E}\left[\xi(\hat{H} \id_{C_n} + \hat{G} \id_{D_n})\right]=F(\sigma)=p$.
			
			\item If $p\notin F(\mathbb{R}_+)$, we have $\sigma>0$ such that $F(\sigma-)\le p<F(\sigma)$, and then
			$
			0<F(\sigma)-F(\sigma-)=\mathbb{E}[\xi (\hat{G}-\hat{H}) \id_{\{\xi=\sigma a_B\}\cap Q_n}].
			$
			Hence, the set $Q'\triangleq\{\xi=\sigma a_B\}\cap Q_n$ has a positive measure $\epsilon$. As $(\Omega,\mathcal{F},\mathbb{P})$ is non-atomic, its restriction on $Q'$ is also non-atomic. Using Lemma \ref{prop1}, we obtain a family of increasing measurable sets $\{Q_t'\}_{0\le t\le\epsilon}$ satisfying $\mathbb{P}[Q_t']=t$ and $Q_t'\subset Q'$. Define $F_1(t)=\mathbb{E}[\xi(\hat{G}-\hat{H}) \id_{Q_t'}]$. It is verified that $F_1$ is nondecreasing and continuous, and
			$
			F_1(0)=0,~F_1(\epsilon)=F(\sigma)-F(\sigma-)\ge F(\sigma)-p.
			$
			As such, one can find $\sigma_1$ such that $F_1(\sigma_1) = F(\sigma) - p$. In this case we define
			$
			C_n=\left(\{\xi>\sigma a_B\}\cap Q_n\right) \cup Q_{\sigma_1}', ~D_n=\left(\{\xi\le \sigma a_B\}\cap Q_n\right)\backslash Q_{\sigma_1}',
			$
			which leads to
			\begin{equation*}
				\begin{aligned}
					\mathbb{E}\left[\xi (\hat{H} \id_{C_n} + \hat{G} \id_{D_n})\right]
					&= \mathbb{E}\left[\xi \hat{H} \id_{\{\xi> \sigma a_B\}\cap Q_n}\right]+\mathbb{E}\left[\xi \hat{H} \id_{Q_{\sigma_1}'}\right]+\mathbb{E}\left[\xi \hat{G} \id_{\{\xi\le \sigma a_B\}\cap Q_n}\right]-\mathbb{E}\left[\xi \hat{G} \id_{Q_{\sigma_1}'}\right]\\
					&=F(\sigma)-F_1(\sigma_1)=p.
				\end{aligned}
			\end{equation*}
		\end{itemize}
		
		Define $\hat{X}=\hat{H}\id_{C_n}+\hat{G}\id_{D_n}$ on $Q_n$, and then \eqref{eq:budget} holds. Moreover, as $\hat{H} \leq \hat{X} \leq \hat{G}$ and $\hat{H} \leq X \leq \hat{G}$, they are bounded on $Q_n$. Hence $\tilde{U}(X,B)=a_BX+c_B$ and $\tilde{U}(\hat{X},B)=a_B\hat{X}+c_B$ are also bounded on $Q_n$. In both two cases we have
		\begin{equation*}
			\begin{aligned}
				&\mathbb{E}\left[\left(\tilde{U}\left(\hat{X},B\right)-\tilde{U}\left(X,B\right)\right)\id_{Q_n}\right]		
				=\mathbb{E}\left[a_B(\hat{G}-X)\id_{D_n}\right]-\mathbb{E}\left[a_B(X-\hat{H})\id_{C_n}\right]\\
				\ge&\mathbb{E}\left[\frac{\xi}{\sigma}(\hat{G}-X)\id_{D_n}\right]-\mathbb{E}\left[\frac{\xi}{\sigma}(X-\hat{H})\id_{C_n}\right]
				=\frac{1}{\sigma}\mathbb{E}\left[\xi\left(\hat{X}-X\right)\id_{Q_n}\right]=0,
			\end{aligned}
		\end{equation*}
		and hence \eqref{eq:ut_increase} holds.
		
		\textbf{Step 3:} For every $n$ with $\mathbb{P}[Q_n]=0$, we set $C_n=Q_n$, $D_n=\emptyset$. Define
		$
		C=\mathop{\bigcup}_{n\ge 1}C_n,
		~D=\mathop{\bigcup}_{n\ge 1}D_n.
		$
		Then $C\cup D=Q$ is a partition of $Q$. Define
		$
		\hat{X}=\hat{H} \id_C + \hat{G} \id_D + X \id_{Q^c},
		$
		which is consistent with our definition in \textbf{Step 2}.	Using   \eqref{eq:budget}, we have
		\begin{equation}\label{eq:substitution budget}
			\begin{aligned}
				\mathbb{E}\left[\xi X\right]&=\mathbb{E}\left[\xi X\id_{Q^c}\right]+\sum\limits_{n\ge1}\mathbb{E}\left[\xi X\id_{Q_n}\right]
				=\mathbb{E}\left[\xi \hat{X}\id_{Q^c}\right]+\sum\limits_{n\ge1}\mathbb{E}\left[\xi \hat{X}\id_{Q_n}\right]\\
				&=\mathbb{E}\left[\xi \hat{X}^+\id_{Q^c}\right]-\mathbb{E}\left[\xi \hat{X}^-\id_{Q^c}\right]+\sum\limits_{n\ge1}\left(\mathbb{E}\left[\xi \hat{X}^+\id_{Q_n}\right]-\mathbb{E}\left[\xi \hat{X}^-\id_{Q_n}\right]\right).
			\end{aligned}
		\end{equation}
		As $\hat{H}<X$  on $Q$, based on the definition of $\hat{H}$, for any $t$ satisfying $\hat{H} < t < X$, we have $H(t,B)=\hat{H}<t$. Based on the definition of $\hat{H}_b$, we know $\hat{H}_B\le t$. Hence $\hat{H}_B\le \hat{H}$ and we have $\hat{X}\ge \hat{H}\ge \hat{H}_B$ on $Q$. Hence
		$
		\mathbb{E}\left[\xi \hat{X}^-\id_{Q}\right]\le\mathbb{E}\left[\xi \hat{H}_B^-\id_{Q}\right]<+\infty,
		$
		and the series $\sum\limits_{n\ge1}\mathbb{E}\left[\xi \hat{X}^-\id_{Q_n}\right]$ converges. As such, \eqref{eq:substitution budget} indicates that the series  $\sum\limits_{n\ge1}\mathbb{E}\left[\xi \hat{X}^+\id_{Q_n}\right]$ also converges, and we have
		\begin{equation}\label{eq:same_budget2}
			\mathbb{E}[\xi \hat{X}]=\mathbb{E}[\xi X]=x_0.
		\end{equation}
		Moreover, based on Lemma \ref{concavification1},  we know $U(\hat{X},B)>-\infty$, and then \eqref{eq:well-defined} indicates that the expectation $\mathbb{E}\left[U(\hat{X},B)\right]$ is well-defined. Using \eqref{eq:ut_increase}, we have (noting that $X=\hat{X}$ and $\tilde{U}(X,B)=U(X,B)$ on $Q^c$)
		\begin{equation*}\label{eq_contradiction}
			\begin{aligned}
				\mathbb{E}\left[\tilde{U}\left(X,B\right)\right]&=\mathbb{E}\left[\tilde{U}\left(X,B\right)\id_{Q^c}\right]+\sum\limits_{n\ge1}\mathbb{E}\left[\tilde{U}\left(X,B\right)\id_{Q_n}\right]\\
				&\le \mathbb{E}\left[U\left(\hat{X},B\right)\id_{Q^c}\right]+\sum\limits_{n\ge1}\mathbb{E}\left[U\left(\hat{X},B\right)\id_{Q_n}\right]
				=\mathbb{E}\left[U\left(\hat{X},B\right)\right].
			\end{aligned}
		\end{equation*}
		Therefore, $\mathbb{E}\left[U\left(\hat{X},B\right)\right]\ge\mathbb{E}\left[\tilde{U}\left(X,B\right)\right]>\beta$, which contradicts to the definition \eqref{def:beta} of $\beta$.
		
		\vskip 5pt
		\textbf{Part II:} We are going to prove that the concavification problem is well-defined. For $X$ satisfying $\mathbb{E}\left[\xi X\right]\le x_0$ and $\tilde{U}\left(X,B\right)>-\infty$ (which is equivalent to $U\left(X,B\right)>-\infty$), we have that  \eqref{eq:well-defined} holds with $U$ replaced by $\tilde{U}$. Noting that $\tilde{U}\ge U$, if $\mathbb{E}\left[U\left(X,B\right)^-\right]<+\infty$, we have immediately $\mathbb{E}\left[\tilde{U}\left(X,B\right)^-\right]<+\infty$. It remains to consider the situation where $\mathbb{E}\left[U\left(X,B\right)^+\right]<+\infty$. We discuss two cases:
		\begin{itemize}
			\item If $\tilde{U}\left(X,B\right)^+=U\left(X,B\right)^+~\text{a.s.}$, then $\mathbb{E}\left[\tilde{U}\left(X,B\right)^+\right]<+\infty$.
			\item If $\mathbb{P}\left[\tilde{U}\left(X,B\right)^+>U\left(X,B\right)^+\right]>0$, let us consider the two functions $\tilde{h}=\tilde{U}\left(\cdot,b\right)^+$ and $h=U\left(\cdot,b\right)^+$ for any given $b\in E$. Based on Lemmas \ref{concavification1} and \ref{concavification2}, we have:
			\begin{itemize}
				\item[(1)] For $a,c\in\mathbb{R}$ and $a<c$, if $\tilde{h}>h$ on $(a,c)$, then $\tilde{h}$ is affine on $(a,c)$.
				\item[(2)] For $t\in\mathbb{R}$, define
				$
				\begin{aligned}
					H^*(t)=\sup\{x\le t:\tilde{h}(x)=h(x)\},~G^*(t)=\inf\{x\ge t:\tilde{h}(x)=h(x)\}.
				\end{aligned}
				$
				If $\tilde{h}(t)>h(t)$, then $G^*(t)$ and $H^*(t)\in\mathbb{R}$, and $\tilde{h}(G^*(t))=h(G^*(t))\in\mathbb{R}$, $\tilde{h}(H^*(t))=h(H^*(t))\in\mathbb{R}$, and then $H(t)\le H^*(t)<t<G^*(t)\le G(t)$, where $H(t)$ and $G(t)$ are defined in Lemma \ref{concavification1}.
			\end{itemize}
			Define
			$
			S^+=\left\{ (x,b)\in\R\times E: \tilde{U}(x,b)^+ > U(x,b)^+  \right\},~~ Q^+=\left\{\omega \in \Omega: (X(\omega),B(\omega)) \in S^+ \right\}.
			$
			Then $\mathbb{P}[Q^+]>0$. Using the above two features, we can replicate our operations in \textbf{Part I} to construct $Q_n^+=C_n^+\cup D_n^+$ and $\hat{X}=\hat{H}^*\id_{C^+}+\hat{G}^*\id_{D^+}+X\id_{Q^{+c}}$ with
			\begin{equation*}
				\hat{H}^*\triangleq H^*(X,B),~\hat{G}^*\triangleq G^*(X,B),~C^+ \triangleq \mathop{\bigcup}_{n\ge1}C_n^+,~D^+\triangleq\mathop{\bigcup}_{n\ge1}D_n^+,~Q^+=\mathop{\bigcup}_{n\ge1}Q_n^+.
			\end{equation*}
			On each $Q_n^+$ with $\mathbb{P}\left[Q_n^+\right]>0$, we have
			$
			\mathbb{E}[\xi \hat{X}\id_{Q_n^+}] = \E[\xi X\id_{Q_n^+}]
			$
			and
			\begin{equation}\label{eq:ut_increase+}
				\mathbb{E}\left[\tilde{U}(X,B)^+\id_{Q_n^+}\right]\le\mathbb{E}\left[\tilde{U}(\hat{X},B)^+\id_{Q_n^+}\right]=\mathbb{E}\left[U(\hat{X},B)^+\id_{Q_n^+}\right].
			\end{equation}
			As discussed in \textbf{Step 3}, we have $\mathbb{E}\left[\xi \hat{X}\right]=\mathbb{E}\left[\xi X\right]\le x_0$.
			Using \eqref{eq:well-defined}, we have
			\begin{equation*}
				\mathbb{E}\left[U(\hat{X},B)^+\right]<+\infty \text{ or }\mathbb{E}\left[U(\hat{X},B)^-\right]<+\infty.
			\end{equation*}
			If $\mathbb{E}\left[U(\hat{X},B)^+\right]<+\infty$, then, using \eqref{eq:ut_increase+}, we have $\mathbb{E}\left[\tilde{U}(X,B)^+\right]<+\infty$.
			
			\noindent
			If $\mathbb{E}\left[U(\hat{X},B)^-\right]<+\infty$, then as $\tilde{U}\ge U$, we have $\mathbb{E}\left[\tilde{U}(\hat{X},B)^-\right]<+\infty$.\\ Noting that on $C^+\cup D^+=Q^+$, we have $\tilde{U}\left(X,B\right)^+>U\left(X,B\right)^+\ge0$. Hence $\tilde{U}\left(X,B\right)^-=0$, and we have
			$
			\mathbb{E}\left[\tilde{U}\left(X,B\right)^-\id_{Q^+}\right]=0,~
			\mathbb{E}\left[\tilde{U}\left(X,B\right)^-\id_{Q^{+c}}\right]=\mathbb{E}\left[\tilde{U}\left(\hat{X},B\right)^-\id_{Q^{+c}}\right]<+\infty,
			$
			which indicates that $\mathbb{E}\left[\tilde{U}(X,B)^-\right]<+\infty$.
			Therefore, the concavification problem is well-defined.
		\end{itemize}		
		\item
		Suppose that $X^*$ is a finite optimal solution of Problem \eqref{problem1}. We have
		\begin{equation*}
			\mathbb{E}\left[\tilde{U}\left(X^*,B\right)\right]\ge\mathbb{E}\left[U\left(X^*,B\right)\right]=\sup_{X: \mathbb{E}[\xi X]=x_0\atop U(X,B)>-\infty}\mathbb{E}\left[U(X,B)\right]=\sup_{X: \mathbb{E}[\xi X]=x_0\atop \tilde{U}(X,B)>-\infty}\mathbb{E}\left[\tilde{U}(X,B)\right].
		\end{equation*}
		As such, $X^*$ is also an optimal solution for $\sup_{X: \mathbb{E}[\xi X]=x_0\atop \tilde{U}(X,B)>-\infty}\mathbb{E}\left[\tilde{U}(X,B)\right]$, and
		$$
		\mathbb{E}\left[\tilde{U}(X^*,B)\right]=\mathbb{E}[U(X^*,B)]\in\mathbb{R},
		$$
		that is, $\tilde{U}(X^*,B)=U(X^*,B)$ a.s.. Applying Theorem \ref{thm_concave}, $X^*\in\mathcal{X}_B^{\tilde{U}}(\lambda\xi)$ for some $\lambda\ge0$.
		\begin{itemize}
			\item
			When $\lambda<+\infty$, we obtain $X^*\in\mathcal{X}_B^U(\lambda\xi)$ from
			\begin{equation*}
				U(X^*,B)-\lambda\xi X^*=\tilde{U}(X^*,B)-\lambda\xi X^*\ge\tilde{U}(x,B)-\lambda\xi x\ge U(x,B)-\lambda\xi x,~\forall x\in\mathbb{R}.
			\end{equation*}
			\item
			When $\lambda=+\infty$, we have $X^*=\underline{x}_B^{\tilde{U}}=\underline{x}_B^{U}$ (based on Lemma \ref{concavification1}). Hence, $X^*\in\mathcal{X}_B^U(\lambda\xi)$.
		\end{itemize}
		Conversely, if there exists $X^*\in\mathcal{X}_B^U(\lambda\xi)$ satisfying $\mathbb{E}[\xi X^*]=x_0$ and $U(X^*,B)>-\infty$, then:
		\begin{itemize}
			\item If $\lambda<+\infty$.
			For any other $X$  satisfying the budget constraint and $U(X,B)>-\infty$, we have
			\begin{equation*}
				U(X^*,B)-\lambda\xi X^*\ge U(X,B)-\lambda\xi X>-\infty.
			\end{equation*}
			Taking the expectation on both sides, we obtain $\mathbb{E}[U(X^*,B)]\ge\mathbb{E}[U(X,B)]$, as such,
			$$
			\mathbb{E}[U(X^*,B)] \ge \sup_{X: \E[\xi X] = x_0\atop U(X,B)>-\infty}\mathbb{E}[U(X,B)] \geq \mathbb{E}[U(X^*,B)].
			$$
			Thus, $X^*$ is an optimal solution of Problem \eqref{problem1}.
			\item If $\lambda=+\infty$, then $X^*=\underline{x}_B$. As $U(X,B)>-\infty$ requires $X\ge\underline{x}_B=X^*$, we have $\mathbb{E}\left[\xi X\right]\ge\mathbb{E}\left[\xi X^*\right]=x_0$. This indicates that $X^*$ is the unique random variable satisfying the constraint in Problem \eqref{problem1}, and hence $X^*$ is the optimal solution.
		\end{itemize}
		Note that in this part we do not need the non-atomic condition. The proof is also valid for the ``if'' part in Theorem \ref{thm_concave}.
	\end{enumerate}
\end{proof}
\begin{remark}\label{rmk:HnGn}
	For the weaker case with only $H_n$ and $G_n$ being finite, we should take some $m\in\mathbb{N}_+$ in \textbf{Part I} with 	$\mathbb{E}[\tilde{U}(X,B)]>\beta+\frac{1}{m}.$ And then we make our operation on the set $S=\big\{ (x,b)\in\R\times E: \tilde{U}(x,b) > U(x,b)+\frac{1}{m}  \big\}$ and replace $X$ by $H_{2m}(X,B)$ and $G_{2m}(X,B)$.
\end{remark}

\section{Proofs in Section \ref{section_finite}}

\subsection{Proof of Lemma \ref{lem}}\label{B}
To prove Lemma \ref{lem}, we need:
\begin{lemma}\label{concavification_appendix}
	Suppose that $u:\mathbb{R}\to\mathbb{R}\cup\{-\infty\}$ proper with a concavification $\tilde{u}:\mathbb{R}\to\mathbb{R}\cup\{-\infty\}$. Then for $y\in(0,+\infty)$ we have $\mathcal{X}^u(y)\subset\mathcal{X}^{\tilde{u}}(y)$, where the term is defined as in \eqref{conjugate point} with $b$ omitted.
\end{lemma}
\begin{proof}[Proof of Lemma \ref{concavification_appendix}.]
	Denote by $f^*(y)$ the convex conjugate of $f$, i.e.
	$
	f^*(y)=\sup_{x\in\mathbb{R}}\{xy-f(x)\}.
	$
	Then we have $\tilde{u}=-(-u)^{**}$ (cf. \citet{R1970}). Using the Fenchel-Moreau Theorem, we know
	\begin{equation*}
		V(y)\triangleq\sup_{x\in\mathbb{R}}\{\tilde{u}(x)-xy\}=\sup_{x\in\mathbb{R}}\{x(-y)-(-u)^{**}(x)\}=(-u)^{***}(-y)=(-u)^*(-y)=\sup_{x\in\mathbb{R}}\{u(x)-xy\}.
	\end{equation*}
	For $x\in\mathcal{X}^u(y)$, we consider three cases:
	\begin{enumerate}[(i)]
		\item $x\in\mathbb{R}$.
		In this case, we have
		$
		\tilde{u}(x)-xy\ge u(x)-xy=\sup_{t\in\mathbb{R}}\{u(t)-ty\}=\sup_{t\in\mathbb{R}}\{\tilde{u}(t)-ty\}.
		$
		Hence $x\in\mathcal{X}^{\tilde{u}}(y)$.
		\item $x=+\infty$. Then there exists $x_n\uparrow+\infty$ with $u(x_n)-x_ny\to \sup_{x\in\mathbb{R}}\{u(x)-xy\}=V(y)$.
		As $V(y)\ge\tilde{u}(x_n)-x_ny\ge u(x_n)-x_ny$, we know $\tilde{u}(x_n)-x_ny\to V(y)$. Hence $+\infty\in\mathcal{X}^{\tilde{u}}(y)$.
		\item $x=-\infty$. The proof is similar as in (ii).
	\end{enumerate}
\end{proof}
\begin{proof}[Proof of Lemma \ref{lem}.]
	As in this lemma, only (iii) is relevant to $b$. For simplicity,  we will first prove (i)(ii)(iv)(v) omitting the notation of $b$, and then prove (iii). Also, as terms in the lemma may equal $\pm\infty$, we will prove by contradiction.
	
	(i) Suppose $x^*\triangleq\underline{X}(y)<\underline{x}$ for some $y\in\[0,+\infty\]$, then $\underline{x}>-\infty$. As $\mathcal{X}(+\infty)=\{\underline{x}\}$, we know $y<+\infty$. Hence there exists a sequence of real numbers $\{x_n\}\downarrow x^*$ with $U(x_n)-x_ny\to V(y)$. However, for $x_n$ closed enough to $x^*$ we have $x_n<\underline{x}$, which means $U(x_n)=-\infty$. Hence $V(y)=-\infty$, which contradicts to Assumption \ref{ass_standing}.
	
	Suppose $x^{**}\triangleq\overline{X}(y)>\overline{x}$ for some $y\in(0,+\infty]$, then $\overline{x}<+\infty$. As $\overline{x}\ge\underline{x}$, we know $y<+\infty$. Hence there exists a sequence of real numbers $\{x_n\}\uparrow x^{**}$ with $U(x_n)-x_ny\to V(y)$. However, for $x_n$ closed enough to $x^{**}$ we have $x_n>\overline{x}$, which means $U(x_n)=U(\overline{x})$. As such, $U(x_n)-x_ny=U(\overline{x})-x_ny\to U(\overline{x})-x^{**}y$. That is, $V(y)=U(\overline{x})-x^{**}y<U(\overline{x})-\overline{x}y$, which contradicts to the definition \eqref{conjugate} of $V$.

	\vskip 5pt	
	(ii) \textbf{Step 1:} We prove $x^*\triangleq\underline{X}(y)\in\mathcal{X}(y)$.
	
	For $y=+\infty$, we have $\mathcal{X}(y)=\{\underline{x}\}$. Hence $\underline{X}(y)=\overline{X}(y)=\underline{x}\in\mathcal{X}(y)$.
	
	For $y\in[0,+\infty)$, we have a sequence of real numbers $\{x_n\}\downarrow$ (or $\uparrow$) $x^*$ with $U(x_n)-x_ny\to V(y)$. If $x^*\in\{\pm\infty\}$, then by the definition \eqref{conjugate point} of $\mathcal{X}(y)$ we know $x^*\in\mathcal{X}(y)$. If $x^*\in\mathbb{R}$, then it follows from the upper semicontinuity of $U$ that
	$
	V(y)=\lim_{n\to+\infty}U(x_n)-x^*y\le U(x^*)-x^*y.
	$
	Based on the definition \eqref{conjugate} of $V$, we know $V(y)=U(x^*)-x^*y$, i.e. $x^*\in\mathcal{X}(y)$. Similarly, we have $\overline{X}(y)\in\mathcal{X}(y)$.
	
	\textbf{Step 2:} We prove $\overline{X}(y_2)\le\underline{X}(y_1)$ for $0\le y_1<y_2\le+\infty$.
	
	If $y_2=+\infty$, then $\overline{X}(y_2)=\underline{x}\le\underline{X}(y_1)$ (using (i) above).
	
	If $y_2<+\infty$, suppose that we have $0\le y_1<y_2$ such that $x^{**}\triangleq\overline{X}(y_2)>x^*\triangleq\underline{X}(y_1)$. We have a sequence of real numbers $\{x_n\}\downarrow$ (or $\uparrow$) $x^*$ with $U(x_n)-x_ny_1\to V(y_1)$, and a sequence of real numbers $\{x_n'\}\downarrow$ (or $\uparrow$) $x^{**}$ with $U(x_n')-x_n'y_2\to V(y_2)$, and we can assume 
	\begin{equation}\label{eq:monotonic}
		\left\{
		\begin{aligned}
			&U(x_n)-x_ny_1> V(y_1)-1/n\ge U(x_n')-x_n'y_1-1/n,\\
			&U(x_n')-x_n'y_2> V(y_2)-1/n\ge U(x_n)-x_ny_2-1/n.
		\end{aligned}
		\right.
	\end{equation}
	As $x^{**}>x^*\ge\underline{x}$, for $n$ sufficiently large we have $x_n'>\underline{x}$ and $U(x_n')\in\mathbb{R}$. Then the first inequality in \eqref{eq:monotonic} indicates $U(x_n)\in\mathbb{R}$, and we can add two inequalities in \eqref{eq:monotonic} to derive
	$
	(x_n-x_n')(y_2-y_1)\ge\frac{2}{n}.
	$
	Letting $n\to+\infty$, as $x_n-x_n'\to x^{*}-x^{**}<0$, we obtain a contradiction.
	
	\vskip 5pt
	(iii) Based on (ii), for $0\le y_1<y_2\le+\infty$, we have $\underline{X}(y_2)\le \overline{X}(y_2)\le \underline{X}(y_1)\le \overline{X}(y_1)$.
	The proof of the measurability will be given at the last of the proof.
	\vskip 5pt	
	(iv)  Because of the nonincreasing property just shown in (iii), both two limits exist (including the limit to infinity).
	
	\textbf{Step 1:} We prove $l\triangleq\lim\limits_{y\to 0+}\overline{X}(y)=\overline{x}$.
	
	If $l\in\mathbb{R}$, then there exist real numbers $y_n\downarrow0$ and $x_n\uparrow l$ such that $\overline{X}(y_n)=x_n\in\mathcal{X}(y_n)$. As such,
	\begin{equation}\label{ineq4}
		\begin{aligned}
			U(l)-x_ny_n\ge U(x_n)-x_ny_n\ge U(x)-xy_n, \text{ for any } x\in\mathbb{R}.
		\end{aligned}
	\end{equation}
	Letting  $n\to +\infty$, (\ref{ineq4}) leads to $U(l)\ge U(x)$, and we have $\overline{x}\le l$. Using (i) we know $l=\overline{x}$.
	\vskip 5pt	
	If $l=+\infty$, it follows from (i) that $+\infty\le \overline{x}$, and hence $\overline{x}_b=+\infty=l$.
	
	If $l=-\infty$, then for any $y\in(0,+\infty)$ we have $\overline{X}(y)=-\infty$. Using Lemma \ref{concavification_appendix}, we know $-\infty\in\mathcal{X}^U(y)\subset\mathcal{X}^{\tilde{U}}(y)$. As $\tilde{U}$ is concave, we have $\tilde{U}'(-\infty)\le y$ for every $y>0$. Noting that $\tilde{U}$ is nondecreasing, $\tilde{U}'(-\infty)\ge0$, we derive $\tilde{U}'(-\infty)=0$. This indicates that $\tilde{U}$ is constant on $\mathbb{R}$. Based on the assumption that $U\in\mathcal{H}$, we know that $U$ is also constant on $\mathbb{R}$. Therefore, $\underline{x}=\overline{x}=-\infty=l$.
	
	Similarly, we can prove $\lim\limits_{y\to 0+}\underline{X}(y)=\overline{x}$.
	\vskip 5pt
	\textbf{Step 2:} We prove $L\triangleq\lim\limits_{y\to +\infty}\overline{X}(y)=\underline{x}$. Suppose that $L>\underline{x}$, then $L>-\infty$ and $U(L-)>-\infty$.
	
	If $L\in\mathbb{R}$, there exist real numbers $y_n\uparrow+\infty$ and $x_n\downarrow L$ such that $\overline{X}(y_n)=x_n$. For any $x\in\mathbb{R}$,
	\begin{equation}\label{ineq6}
		\begin{aligned}
			U(x_n)-x_ny_n\ge U(x)-xy_n .
		\end{aligned}
	\end{equation}
	For $v>0$, we take $x=x_n-\frac{v}{y_n}>\underline{x}$ for some large $n$, and (\ref{ineq6}) yields
	$
	\begin{aligned}
		U(x_n)\ge U(x_n-\frac{v}{y_n})+v.
	\end{aligned}
	$
	Letting  $n\to+\infty$, we obtain
	$
	\begin{aligned}
		U(L)\ge U(L-)+v,
	\end{aligned}
	$
	which is a contradiction as $v$ is arbitrary. Thus $L$ cannot be larger than $\underline{x}$. As $L\ge\underline{x}$ we know $L=\underline{x}$.
	\vskip 5pt	
	If $L=+\infty$, then for any $y\in[0,+\infty)$, we have $\overline{X}(y)=+\infty\in\mathcal{X}^{\tilde{U}}(y)$. As such, $\tilde{U}'(+\infty)\ge y$ for every $y<+\infty$, which contradicts to the fact that $\tilde{U}$ is concave.
	
	In conclusion, $L=\underline{x}$. Similarly, we can prove $\lim\limits_{y\to +\infty}\underline{X}(y)=\underline{x}$.
	\vskip 5pt	
	(v) \textbf{Step 1:} We prove $l\triangleq\lim\limits_{y\to y_0+}\overline{X}(y)=\underline{X}(y_0)$. Suppose $l<\underline{X}(y_0)$, which implies $l<+\infty$.
	
	If $l\in\mathbb{R}$, again we have real numbers $y_n\downarrow y_0$ as well as $x_n=\overline{X}(y_n)\uparrow l$ satisfying (\ref{ineq4}). When $n$ tends to infinity, the inequality shows that for any $x\in\mathbb{R}$,
	$
	\begin{aligned}
		U(l)-ly_0\ge U(x)-xy_0 ,
	\end{aligned}
	$
	which means $l\in \mathcal{X}(y_0)$. As such, $l\ge \underline{X}(y_0)$, which is a contradiction.
	
	If $l=-\infty$, then for any $y\in(y_0,+\infty)$ we have $\overline{X}(y)=-\infty$. As we have discussed in (iv), this indicates $\tilde{U}'(-\infty)\le y$ for any $y>y_0$, i.e., $\tilde{U}'(-\infty)\le y_0$.
	Let us discuss several cases.
	
	(a) If $l'\triangleq\underline{X}(y_0)\in\mathbb{R}$, and for any $x<l'$, we have $U(x)=\tilde{U}(x)$.
	
	In this case, for $x<l'$ we have
	$
	U(x)-xy_0=\tilde{U}(x)-xy_0\ge\tilde{U}(l')-l'y_0\ge U(l')-l'y_0=V(y_0),
	$
	which means $x\in\mathcal{X}(y_0)$. This contradicts to the definition of $\underline{X}(y_0)$ and the fact $x<l'$.
	
	(b) If $l'\in\mathbb{R}$, and there exists $x<l'$ such that $U(x)<\tilde{U}(x)$.
	
	As $U\in\mathcal{H}$, there exists a nonincreasing sequence $\{x_k'\}$ with $x_k'\le x$ and $\tilde{U}(x_k')<U(x_k')+\frac{1}{k}$. As such, for $l'> x\ge x_k'$, using the fact $\tilde{U}'(-\infty)\le y_0$, we have
	\begin{equation*}
		U(x_k')-x_k'y_0+\frac{1}{k}>\tilde{U}(x_k')-x_k'y_0\ge\tilde{U}(l')-l'y_0\ge U(l')-l'y_0=V(y_0).
	\end{equation*}
	Letting $k\to+\infty$, we know $\lim\limits_{k\to+\infty}x_k'\in\mathcal{X}(y_0)$, which is also a contradiction because $\lim\limits_{k\to+\infty}x_k'<l'$.
	
	(c) If $l'=+\infty$.
	In this case we have $x_k''$ increasing to $+\infty$ with $U(x_k'')-x_k''y_0>V(y_0)-\frac{1}{k}$. Fix $x\in\mathbb{R}$, using the nonincreasing sequence $\{x_k'\}$ in (b), we derive for $k$ sufficiently large that
	\begin{equation*}
		U(x_k')-x_k'y_0+\frac{1}{k}>\tilde{U}(x_k')-x_k'y_0\ge\tilde{U}(x_k'')-x_k''y_0\ge U(x_k'')-x_k''y_0>V(y_0)-\frac{1}{k}.
	\end{equation*}
	Letting $k\to+\infty$, we know $\lim\limits_{k\to+\infty}x_k'\in\mathcal{X}(y_0)$, which is also a contradiction because $\lim\limits_{k\to+\infty}x_k'\le x<l'$.
	
	Concluding (a)-(c), we have proved $l'=-\infty=l$.
	\vskip 5pt
	\textbf{Step 2:} We prove $L\triangleq\lim\limits_{y\to y_0-}\underline{X}(y)=\overline{X}(y_0)$. Suppose $L>\overline{X}(y_0)$, then $L>-\infty$.
	
	If $L\in\mathbb{R}$, the proof is all the same as in \textbf{Step 1}.
	
	If $L=+\infty$, then for any $y\in[0,y_0)$ we have $\underline{X}(y)=+\infty$. As we have discussed in (iv), this indicates $\tilde{U}'(+\infty)\ge y$ for any $y\in[0,y_0)$, i.e., $\tilde{U}'(+\infty)\ge y_0$. Using the same method in \textbf{Step 1}, we can prove $L'\triangleq\overline{X}(y_0)=+\infty=L$.
	
	\vskip 5pt	
	(iii) At last, we prove the measurability of $V_b(y),~\underline{X}_b(y)$ and $\overline{X}_b(y)$.
	
	\textbf{Step 1:} We prove that $V_b(y)$ is measurable.
	
	Note that $U(\cdot,b)$ is nondecreasing. For $t\in\mathbb{R}$, $V_b(y)\le t$ is equivalent to $    U(x,b)-xy\le t$ for all $x\in\mathbb{Q}$. That is,
	$
	(y,b)\in\mathop{\bigcap}\limits_{x\in\mathbb{Q}}S_x,
	$
	where $S_x=\left\{(y,b)\in[0,+\infty)\times E:U(x,b)-xy\le t\right\}.$ For given $x$, as $U(x,b)$ is a measurable function of $b$, we know that $U(x,b)-xy$ is a measurable function of $(y,b)$ and $S_x\in\mathcal{B}\left([0,+\infty)\right)\times\mathcal{E}$. Therefore, we obtain $\mathop{\cap}\limits_{x\in\mathbb{Q}}S_x\in\mathcal{B}\left([0,+\infty)\right)\times\mathcal{E}$, which leads to the measurability of $V_b(y)$ in $(y,b)$.
	
	\textbf{Step 2:} We prove that $\underline{X}_b(y)$ is measurable. Based on the definition of $\underline{X}_b(y)$,	 for $(y,b)\in[0,+\infty)\times E$ and $t\in\mathbb{R}$ we have:
	\begin{equation}\label{eq:equivalance}
		\underline{X}_b(y)> t\iff~\left\{
		\begin{aligned}
			&\exists~n\in\mathbb{N}_+~s.t.~U(x,b)-xy\le V_b(y)-\frac{1}{n}\text{ holds for all }x\le t,~&\text{if }V_b(y)<+\infty;\\
			&\exists~n\in\mathbb{N}_+~s.t.~U(x,b)-xy\le n\text{ holds for all }x\le t,~&\text{if }V_b(y)=+\infty;\\
		\end{aligned}\right.
	\end{equation}
	%
	Noting that in the definition of $\underline{X}_b(y)$, $y$ is allowed to be $+\infty$. As such, $\underline{X}_b(y)> t$ is equivalent to
	\begin{equation*}
		(y,b)\in\bigg(\mathop{\bigcup}_{n\in\mathbb{N_+}}\mathop{\bigcap}\limits_{x\in\mathbb{Q}\cup\{t\}\atop x\le t}\widetilde{S}_x^n\bigg)\bigcup\bigg(\mathop{\bigcup}_{n\in\mathbb{N_+}}\mathop{\bigcap}\limits_{x\in\mathbb{Q}\cup\{t\}\atop x\le t}\widetilde{T}_x^n\bigg)\bigcup W,
	\end{equation*}
	where
	\begin{equation*}
		\left\{\begin{aligned}
			&\widetilde{S}_x^n\triangleq\left\{(y,b)\in[0,+\infty)\times E:U(x,b)-xy\le V_b(y)-\frac{1}{n}\right\}\cap\{(y,b)\in[0,+\infty)\times E:V_b(y)<+\infty\},\\
			&\widetilde{T}_x^n\triangleq\left\{(y,b)\in[0,+\infty)\times E:U(x,b)-xy\le n\right\}\cap\{(y,b)\in[0,+\infty)\times E:V_b(y)=+\infty\},\\
			&W=\{+\infty\}\times\{b\in E:\underline{x}_b>t\}.
		\end{aligned}\right.
	\end{equation*}
	Based on the measurability of $V_b(y)$, we know that $U(x,b)-xy-V_b(y)$ is measurable in $(y,b)$ for given $x$. Hence $\widetilde{S}_x^n\in\mathcal{B}\left([0,+\infty)\right)\times\mathcal{E}\subset\mathcal{B}\left([0,+\infty]\right)\times\mathcal{E}$. Similarly, we have $\widetilde{T}_x^n\in\mathcal{B}\left([0,+\infty]\right)\times\mathcal{E}$. We consider $W$. As $\underline{x}_b>t$ is equivalent to
	$
	\exists~ n\in\mathbb{N}_+~s.t.~U(x,b)=-\infty,~\forall x\in\mathbb{Q}\cap(-\infty,t+\frac{1}{n}].
	$
	Then
	\begin{equation*}
		\underline{x}_b>t\iff b\in\mathop{\bigcup}_{n\in\mathbb{N}_+}\mathop{\bigcap}\limits_{x\in\mathbb{Q}\atop x\le t+\frac{1}{n}}\{b\in E:U(x,b)=-\infty\}.
	\end{equation*}
	As $U(x,\cdot)$ is measurable on $(E,\mathcal{E})$, we know $\{b\in E:\underline{x}_b>t\}\in\mathcal{E}$. Hence $W\in\mathcal{B}\left([0,+\infty]\right)\times\mathcal{E}$. Therefore, we know $\left\{(y,b)\in[0,+\infty]\times E:\underline{X}_b(y)>t\right\}\in\mathcal{B}\left([0,+\infty]\right)\times\mathcal{E}$, and $\underline{X}_b(y)$ is measurable.
	
	Using the same method, one can prove the measurability of $\overline{X}_b(y)$.
\end{proof}	
\subsection{Proof of Proposition \ref{thm_feasible}}
\begin{proof}[Proof of Proposition \ref{thm_feasible}.]
	To start with, it is clear that
	\begin{itemize}
		\item If $x_0<\mathbb{E}[\xi \underline{x}_B]$, then $x_0\notin I$;
		\item If $x_0=\mathbb{E}[\xi \underline{x}_B]$, $X=\underline{x}_B$ is the only possible feasible solution.
		\item If $x_0\ge\mathbb{E}[\xi \overline{x}_B]$, Problem \eqref{problem1} admits a bliss solution $X\ge\overline{x}_B$.
	\end{itemize}
	The main difficulty appears in the case $\E[\xi \underline{x}_B] < x_0 < \E[\xi \overline{x}_B]$.  We first prove that the set $I$ is a connected interval, i.e., if $x\in I$, then $[x,+\infty)\subset I$. For $x\in I$, suppose $X$ is a feasible solution with respect to the initial value $x$. Setting $X'=X+c \id_{\{\xi<n\}}$ for some positive constants $c$ and $n$, one can prove that any $x'>x$ is also an element of $I$.
	
	Next, we proceed to prove that if $x \in I$ and $x>\mathbb{E}[\xi \underline{x}_B]$, then there exists $x_1<x$ such that $[x_1,+\infty)\subset I$. Suppose that $\mathbb{E}[\xi\underline{x}_B]<x=\mathbb{E}[\xi X]$ for some feasible solution $X$. Then we have the set $A_1=\{X>\underline{x}_B\}\cap\{\xi<N\}$ has a positive probability for some $N>0$. Noting that $\Delta_x U(X,B)\triangleq U(X+,B)-U(X-,B)$ is finite and nonnegative on $A_1$, we have $A_2=\{\Delta_x U(X,B)<N_1\}\cap\{\xi<N\}$ with $\mathbb{P}[A_2]>0$ for some $N_1>0$. As $U(\cdot,B)$ is nondecreasing and right-continuous, we rewrite $A_2$ as
	\begin{equation*}\small
		A_2=\Bigg(\mathop{\bigcup}\limits_{n\in\mathbb{N}_+}\bigg\{U(X,B)-U(X-\frac{1}{n},B)<N_1\bigg\}\Bigg)\bigcap\{\xi<N\}.
	\end{equation*}
	Therefore, there exists $N_2>0$ such that $A
	\triangleq\big\{U(X,B)-U(X-\frac{1}{N_2},B)<N_1\big\}\cap\{\xi<N\}$ has positive probability. Define $X'=X-\frac{1}{N_2}\id_A$. Based on the definition of $A$, we know $X'\ge \underline{x}_B$, and $U(X',B) > U(X,B) - N_1$. As such, $\mathbb{E}[U(X', B)] > \mathbb{E}[U(X,B)] - N_1 > -\infty$. Further, $\mathbb{E}[\xi X']=\mathbb{E}[\xi X] - \frac{1}{N_2} \mathbb{E}[\xi \id_A]$ is a finite number $x_1<x$ with $x_1\in I$, which also indicates $[x_1, +\infty)\subset I$.
	
	As a result, we have either $I=[\mathbb{E}[\xi \underline{x}_B],+\infty)$ or $I=(\tilde{x},+\infty)$ for some $\tilde{x}\ge\mathbb{E}[\xi \underline{x}_B]$, where $\tilde{x}$ is allowed to take value in $\{\pm\infty\}$.
\end{proof}

\subsection{Proof of Theorem \ref{main}}
We prove Theorem \ref{main} with Assumptions \ref{ass_standing} replaced by Assumptions \ref{ass_G} and \ref{asmp:Hn}.
The proof of Theorem \ref{main} requires the following lemma:
\begin{lemma}\label{lemma1}
	For any finite random variable $Y$ in a non-atomic probability space $(\Omega,\mathcal{F},\mathbb{P})$ with $\mathbb{E}[Y^+]=+\infty$ and any  $a\ge0$, there exists an event $A\in\mathcal{F}$ such that $\mathbb{E}[Y \id_A]=a$.
\end{lemma}
\begin{proof}[Proof of Lemma \ref{lemma1}.]
	It suffices to find $A\in\mathcal{F}$ such that $\mathbb{E}[Y^+ \id_A]=a$. Hence we only need to consider the case that $Y\ge0$. Define $F_1(t)=\mathbb{E}[Y \id_{\{Y<t\}}]$. $F_1$ is nondecreasing and left-continuous with $F_1(+\infty)=+\infty$. Therefore, we find $t_1$ such that $F_1(t_1)>a$. Noting that $(\Omega,\mathcal{F},\mathbb{P})$ is non-atomic, based on Lemma \ref{prop1}, we get a family of increasing measurable sets  $\{A_t\}_{0\le t\le 1}$ with $\mathbb{P}[A_t]=t$. Define $F_2(t)=\mathbb{E}[Y \id_{A_t\cap\{Y<t_1\}}]$. $F_2$ is continuous with $F_2(0)=0$ and $F_2(1)=F_1(t_1)>a$. As such, we find $t_2 \in (0, t_1)$ satisfying $F_2(t_2)=\mathbb{E}[Y \id_{A_{t_2}\cap\{Y<t_1\}}]=a$. Thus, $A = A_{t_2}\cap\{Y<t_1\} \in \F$ is the desired set.
\end{proof}

\begin{proof}[Proof of Theorem \ref{main}.]
	\begin{enumerate}[(1)]
		\item As $x_0 > \mathbb{E}[\xi\underline{x}_B]$ and $x_0 \in I$, based on Proposition \ref{thm_feasible}, there admits a feasible solution $X'$ for Problem (\ref{problem1}) with initial value $x_1 < x_0$.
		
		For any $\lambda>0$, as $\mathbb{E}[\xi\underline{X}_B(\lambda\xi)]=+\infty$, we know that  $\mathbb{E}\left[\xi \left(\underline{X}_B\left(\lambda\xi\right)-X'\right)\right]=+\infty$. Using Lemma \ref{lemma1}, we have $A\in\mathcal{F}$ satisfying $\mathbb{E}[\xi (\underline{X}_B(\lambda\xi)-X')\id_A]=x_0-x_1$. Define $X\triangleq\underline{X}_B(\lambda\xi)\id_A+X'\id_{A^c}$, then
		$
		\mathbb{E}[\xi X]=\mathbb{E}[\xi \underline{X}_B(\lambda\xi)\id_A]+x_1-\mathbb{E}[\xi X'\id_A]=x_0,
		$
		and
		$$
		\begin{aligned}
			\mathbb{E}[U(X,B)]=\mathbb{E}[U(\underline{X}_B(\lambda\xi),B)\id_A]+\mathbb{E}[U(X',B)\id_{A^c}]
			& \ge \mathbb{E}[U(X',B)+\lambda\xi (\underline{X}_B(\lambda\xi)-X')\id_A]\\
			& = \mathbb{E}[U(X',B)]+\lambda(x_0-x_1).
		\end{aligned}
		$$
		Letting $\lambda\to+\infty$, it follows that Problem \eqref{problem1} is infinite.
		\item
		\begin{enumerate}
			\item
			For
			$X^* =\underline{X}_B\big(\lambda^*\xi\big)\in\mathcal{X}_B\big(\lambda^*\xi\big)$ and any solution $X$ satisfying $\E[\xi X] = x_0=\E[\xi X^*]$,
			we have $X^*>-\infty~\text{a.s.}$. Based on the definition of $\mathcal{X}_b(y)$ in \eqref{conjugate point}, we know $U(X^*,B)>-\infty~\text{a.s.}$, and
			\begin{equation}\label{ineq0}
				\begin{aligned}
					U(X^*,B)-\lambda^*\xi X^*\ge U(X,B)-\lambda^*\xi X.
				\end{aligned}
			\end{equation}
			As such,
			\begin{equation}\label{ineq1}
				\begin{aligned}
					\E\big[U(X^*,B)\big]\ge \E\big[U(X,B)\big],
				\end{aligned}
			\end{equation}
			and the optimality of $X^*$ then follows. As $x_0\in I$, we can take $X$ in \eqref{ineq1} as a feasible solution, and then we know that the left side is not $-\infty$. When Assumption \ref{ass_finite} holds, (\ref{ineq1}) takes ``=" if and only if (\ref{ineq0}) takes ``=" almost surely, we have $X\in\mathcal{X}_B(\lambda^*\xi)$ happens almost surely. Based on the definition of $\underline{X}_b$, we know $X\ge X^*$.
			As $\E\left[\xi X\right]=\E\left[\xi X^*\right]$,
			we have $X = X^*$ almost surely, and the uniqueness of $X^*$ follows.
			
			\item
			
			For $x_0=g(\lambda^*-)$, as $\lambda^*>\lambda_0$, based on Lemma \ref{lem}, we have $g(\lambda^*-)=\mathbb{E}[\xi \overline{X}_B(\lambda^*\xi)]$. Therefore, $X^*=\overline{X}_B(\lambda^*\xi)$ is an optimal solution. Similar to (1), we know that $X^*$ is the unique one.

			As $\lambda^*$ is a discontinuous point,  we know $g(\lambda^*-)>g(\lambda^*)$, i.e. the set $\{\underline{X}_B(\lambda^*\xi)\not= \overline{X}_B(\lambda^*\xi)\}$ has a positive measure. We show that there in fact exists infinitely many optimal solutions when $x_0<g(\lambda^*-)$. As the probability space is non-atomic, it admits a standard normal distribution $W$.
			For $A\in\mathcal{B}(\mathbb{R})$, define
			$
			X^A=\underline{X}_B(\lambda^*\xi)+\big[\overline{X}_B(\lambda^*\xi)-\underline{X}_B(\lambda^*\xi)\big] \id_{\{W\in A\}}\in\mathcal{X}_B(\lambda^*\xi).
			$
			For simplicity, denote
			$
			X_1=\xi \underline{X}_B(\lambda^*\xi),\  X_2=\xi\big[\overline{X}_B(\lambda^*\xi)-\underline{X}_B(\lambda^*\xi)\big].
			$
			Then
			$$
			\E X_1=g(\lambda^*)<x_0, ~\E X_2=g(\lambda^*-)-g(\lambda^*)(\triangleq \rho_1),
			$$ and
			$$
			\E\left[\xi X^A(T)\right]=\E X_1+\E [X_2 \id_{\{W\in A\}}]=g(\lambda^*)+\E [X_2 \id_{\{W\in A\}}].
			$$
			Define $s(A)\triangleq\E [X_2 \id_{\{W\in A\}}]$. We need to choose $A$ such that $s(A)=x_0-g(\lambda^*)\triangleq \rho_2\in(0,\rho_1)$.
			Choose $k$ large enough such that $\frac{\rho_1-\rho_2}{2^k}<\rho_2$, and denote further
			$
			a_0\triangleq\frac{\rho_1-\rho_2}{2}$, $
			a_n\triangleq\rho_2-\frac{\rho_1-\rho_2}{2^{n+k}},~n\ge1.
			$
			Then for any $ ~i,~j$,
			\begin{equation}\label{order}
				a_1<a_2<a_3<...<\rho_2,\ a_j<a_i+a_0<\frac{\rho_1+\rho_2}{2}.
			\end{equation}
			
			Let $g_1(t)=s\big((-\infty,t)\big)$. Using the monotone convergence theorem, we know $g(-\infty)=0$ and $g(+\infty)=s(\mathbb{R})=\E X_2=\rho_1$, and that $g_1$ is increasing and continuous. Therefore, for any $n$, there exist $\delta_n$, $\epsilon_n$ and $\zeta_n$ such that
			\begin{equation*}
				\delta_n<\epsilon_n<\zeta_n, \ \ g_1(\delta_n)=a_n,\ \ g_1(\epsilon_n)=a_n+a_0,\ \ g_1(\zeta_n)=\frac{\rho_1+\rho_2}{2}=\rho_2+a_0.
			\end{equation*}
			Using (\ref{order}), we know that for any $i$, $j$,
			$
			\delta_1<\delta_2<...<\delta_n,\ \epsilon_i>\delta_j.
			$
			Define $A_n\triangleq(-\infty,\delta_n)\cup[\epsilon_n,\zeta_n)$. We have
			$
			s(A_n)=g_1(\delta_n)+g_1(\zeta_n)-g_1(\epsilon_n)=\rho_2.
			$
			Consequently, $X^{A_n}\in\mathcal{X}_B(\lambda^*\xi)$, $\E\left[\xi X^{A_n}\right]=x_0$. We obtain that $X^{A_n}$ is an optimal solution.
			
			Finally,  for any $i<j$,		
			let $A=(-\infty,\delta_j)$, then
			\begin{equation*}
				\E\left[X_2 \id_{\{W\in A_i\}} \id_{\{W\in A\}} \right] = g_1(\delta_i) = a_i < a_j = g_1(\delta_j) = \E\left[X_2 \id_{\{W\in A_j\}} \id_{\{W\in A\}}\right].
			\end{equation*}
			Thus $X^{A_i}$ and $X^{A_j}$ are different solutions, and we have constructed infinitely many solutions.

			\item[(c1)] The assertion has been already proved in Proposition \ref{thm_feasible}, as $g(0)=\mathbb{E}[\xi\overline{x}_B]$.
			
			\item[(c2)]
			For $g(\lambda_0)<x_0\le\mathbb{E}[\xi \overline{X}_B(\lambda_0\xi)]$, we know $g(\lambda_0)=\mathbb{E}[\xi \underline{X}_B(\lambda_0\xi)]<+\infty$, and again we obtain $\mathbb{P}[A]=\epsilon>0$ with $A=\{\underline{X}_B(\lambda_0\xi)\not= \overline{X}_B(\lambda_0\xi)\}$. Here we also consider two cases:
			\begin{itemize}
				\item If $\mathbb{E}[\xi \overline{X}_B(\lambda_0\xi)]<\infty$, using the same methods in (2), we construct infinitely many optimal solutions when $x_0<\mathbb{E}[\xi \overline{X}_B(\lambda_0\xi)]$ and one (unique) optimal solution when $x_0=\mathbb{E}[\xi \overline{X}_B(\lambda_0\xi)]$.
				\item If $\mathbb{E}[\xi \overline{X}_B(\lambda_0\xi)]=+\infty$, then we use Lemma \ref{lemma1} to construct an optimal solution. As $\mathbb{E}[\xi\underline{X}_B(\lambda_0\xi)]=g(\lambda_0)<x_0<+\infty$, we know  $\mathbb{E}[\xi(\overline{X}_B(\lambda_0\xi)-\underline{X}_B(\lambda_0\xi))]=+\infty$. Then Lemma \ref{lemma1} shows that we can find $A\in\mathcal{F}$ such that
				$
				\mathbb{E}[\xi(\overline{X}_B(\lambda_0\xi)-\underline{X}_B(\lambda_0\xi))\id_A]>x_0-g(\lambda_0).
				$
				Using the same method as in (2), we have $A_1\subset A$ with $
				\mathbb{E}[\xi(\overline{X}_B(\lambda_0\xi)-\underline{X}_B(\lambda_0\xi))\id_{A_1}]=x_0-g(\lambda_0).
				$
				Let
				$
				X^{A_1}=\overline{X}_B(\lambda_0\xi)\id_{A_1}+\underline{X}_B(\lambda_0\xi)\id_{A_1^c}\in\mathcal{X}_B(\lambda_0\xi),
				$
				we have
				$$
				\mathbb{E}[\xi X^{A_1}]=\mathbb{E}[\xi\underline{X}_B(\lambda_0\xi)]+\mathbb{E}[\xi(\overline{X}_B(\lambda_0\xi)-\underline{X}_B(\lambda_0\xi))\id_{A_1}]=g(\lambda_0)+x_0-g(\lambda_0)=x_0.
				$$
				Based on Theorem \ref{general thm}, we have found an optimal solution $X^{A_1}$. Similarly as in (2), one can construct infinitely many optimal solutions.	
			\end{itemize}
			
			For $x_0>\mathbb{E}[\xi \overline{X}_B(\lambda_0\xi)]\triangleq\theta$. We know $\theta\ge g(\lambda_0)>-\infty$. We first prove that there is no finite optimal solution. Suppose $X^*$ is a finite optimal solution. Using Theorem \ref{general thm}, we know  $X^*\in\mathcal{X}_B(\lambda\xi)$ for some $\lambda\ge0$.
			We consider two cases:
			
			\begin{itemize}
				\item If $\lambda\ge\lambda_0$, then, based on Lemma \ref{lem}, we know $X^*\le \overline{X}_B(\lambda_0\xi)$, and hence $x_0=\mathbb{E}[\xi X^*]\le \mathbb{E}[\xi \overline{X}_B(\lambda_0\xi)]$, which leads to a contradiction.
				\item If $\lambda<\lambda_0$, then as $X^*\ge \underline{X}_B(\lambda'\xi)$ for some $0\le\lambda<\lambda'<\lambda_0$, we know $g(\lambda')<+\infty$, which leads to a contradiction.
			\end{itemize}
			
			For \eqref{eq:main}, take a sequence  $\{\lambda_k\}_{k\ge1}\subset(0,\lambda_0)$ with $\lambda_k\uparrow\lambda_0$. For each $k$,
			$$\mathbb{E}\left[\xi \left(\underline{X}_B(\lambda_k\xi)-\overline{X}_B(\lambda_0\xi)\right)\right]=g(\lambda_k)-\theta=+\infty.$$
			Using Lemma \ref{lemma1}, we can find a measurable set $A_k$ such that
			$
			\E\left[\xi\left(\underline{X}_B(\lambda_k\xi)-\overline{X}_B(\lambda_0\xi)\right)\id_{A_k}\right]=x_0-\theta.
			$
			Define $\hat{X}_k=\underline{X}_B(\lambda_k\xi)\id_{A_k}+\overline{X}_B(\lambda_0\xi)\id_{A_k^c}$. Then $\E[\xi\hat{X}_k]=x_0$. Hence $\hat{X}_k>-\infty~\text{a.s.}$. As $\hat{X}_k$ always locates in some $\mathcal{X}_b(y)$, we know  $U(\hat{X}_k,B)>-\infty~\text{a.s.}$. Suppose that $X$ is a random variable satisfying $\E[\xi X]=x_0$ and $U(X,B)>-\infty$. By the definition of $\hat{X}_k$, we have
			\begin{equation*}
				\left\{\begin{aligned}
					&\left(U(X,B)-\lambda_k\xi X\right)\id_{A_k}\le \left(U\(\hat{X}_k,B\)-\lambda_k\xi\hat{X}_k\right)\id_{A_k},\\
					&\left(U(X,B)-\lambda_0\xi X\right)\id_{A_k^c}\le \left(U\(\hat{X}_k,B\)-\lambda_0\xi\hat{X}_k\right)\id_{A_k^c}.
				\end{aligned}\right.
			\end{equation*}
			Taking expectation on both sides and adding the two inequalities, we get
			\begin{equation}\label{eq:ineq1}
				\E\left[U(X,B)\right]-\lambda_0x_0-(\lambda_k-\lambda_0)\E\left[\xi X\id_{A_k}\right]\le\E\left[U\left(\hat{X}_k,B\right)\right]-\lambda_0x_0-(\lambda_k-\lambda_0)\E\left[\xi \hat{X}_k\id_{A_k}\right].
			\end{equation}
			Using Lemma \ref{lem}, we know that $\underline{X}_B(\lambda_k\xi)$ converges to $\overline{X}_B(\lambda_0\xi)$ almost surely. Hence $\underline{X}_B(\lambda_k\xi)$ also converges to $\overline{X}_B(\lambda_0\xi)$ in probability. For $\epsilon>0$, we have
			\begin{equation*}
				\mathbb{P}\left[\left|\hat{X}_k-\overline{X}_B(\lambda_0\xi)\right|>\epsilon\right]\le \mathbb{P}\left[\left|\underline{X}_B(\lambda_1\xi)-\overline{X}_B(\lambda_0\xi)\right|>\epsilon\right].
			\end{equation*}
			As such, $\hat{X}_k$ also converges to $\overline{X}_B(\lambda_0\xi)$ in probability, and we can find one of its subsequence converging to $\overline{X}_B(\lambda_0\xi)$ almost surely, which is still denoted by $\{\hat{X}_k\}_{k\ge1}$ for simplicity.
			As $\E\left[\xi\overline{X}_B(\lambda_0\xi)\right]$ is finite, we know $\E\left[\xi \hat{X}_k\id_{A_k}\right]\ge\E\left[\xi\overline{X}_B(\lambda_0\xi)\id_{A_k}\right]>-\E\left[\xi\overline{X}_B(\lambda_0\xi)^-\right]>-\infty,$
			and
			$
			\E\left[\xi \hat{X}_k\id_{A_k}\right]=x_0-\E\big[\xi \hat{X}_k\id_{A_k^c}\big]=x_0-\E\left[\xi \overline{X}_B(\lambda_0\xi)\id_{A_k^c}\right]\le x_0+\E\left[\xi\overline{X}_B(\lambda_0\xi)^-\right]<+\infty.
			$
			Letting $k\to+\infty$ on both sides of \eqref{eq:ineq1}, we know
			\begin{equation}\label{eq:main_lim}
				\E\[U(X,B)\]\le\liminf_{k\to+\infty}\E\[U\(\hat{X}_k,B\)\].
			\end{equation}
			Hence
			$$\sup_{\E\[\xi X\]\le x_0\atop U(X,B)>-\infty}\E\[U(X,B)\]\le\liminf_{k\to+\infty}\E\[U\(\hat{X}_k,B\)\].$$
			However, as $\E\[\xi\hat{X}_k\]=x_0$ and $U(\hat{X}_k,B)>-\infty$, we have $\E\[U\(\hat{X}_k,B\)\]\le\sup\limits_{\E\[\xi X\]\le x_0\atop U(X,B)>-\infty}\E\[U(X,B)\].$
			As such, we obtain
			$\sup_{\E\[\xi X\]\le x_0\atop U(X,B)>-\infty}\E\[U(X,B)\]=\lim\limits_{k\to+\infty}\E\[U\(\hat{X}_k,B\)\].$
			Moreover, as $U\(\hat{X}_k,B\)\le U\(\overline{X}_B\(\lambda_0\xi\),B\)+\lambda_0\xi\(\hat{X}_k-\overline{X}_B\(\lambda_0\xi\)\)\triangleq Y_k$, 
			{\color{black} using the properties of concavification function, we have
				\begin{equation*}
					\overline{X}_b(y) = \sup \mathcal{X}^{U}_b(y) = \sup \mathcal{X}^{\tilde{U}}_b(y),~~
					\underline{X}_b(y) = \inf \mathcal{X}^{U}_b(y) = \inf \mathcal{X}^{\tilde{U}}_b(y),
				\end{equation*}
				and
				\begin{equation*}
					\tilde{U}\(\overline{X}_b\(y\),b\)=U\(\overline{X}_b\(y\),b\),~~
					\tilde{U}\(\underline{X}_b\(y\),b\)=U\(\underline{X}_b\(y\),b\).
				\end{equation*}
				Therefore, we can write 
			\begin{equation}\label{diff yk u}
				Y_k-U\(\hat{X}_k,B\)=\left[
				\tilde{U}\(\overline{X}_B\(\lambda_0\xi\),B\)-
				\tilde{U}\(\underline{X}_B\(\lambda_k\xi\),B\)-
				\lambda_0\xi\(\overline{X}_B\(\lambda_0\xi\)-
				\underline{X}_B\(\lambda_k\xi\)\)\right]\id_{A_k}.
			\end{equation}
			Noting that $\overline{X}_B\(\lambda_0\xi\)\le
			\underline{X}_B\(\lambda_k\xi\)$, we have
			$$
			\tilde{U}\(\overline{X}_B\(\lambda_0\xi\),B\)-
			\tilde{U}\(\underline{X}_B\(\lambda_k\xi\),B\)\le
			\tilde{U}'_-\(\underline{X}_B\(\lambda_k\xi\),B\)
			\(\overline{X}_B\(\lambda_0\xi\)-
			\underline{X}_B\(\lambda_k\xi\)\).
			$$
			
			Using \eqref{concave arg}, we know $\tilde{U}'_-\(\underline{X}_B\(\lambda_k\xi\),B\)\ge\lambda_k\xi$. Hence, 
			$$
			\tilde{U}\(\overline{X}_B\(\lambda_0\xi\),B\)-
			\tilde{U}\(\underline{X}_B\(\lambda_k\xi\),B\)\le
			\lambda_k\xi
			\(\overline{X}_B\(\lambda_0\xi\)-
			\underline{X}_B\(\lambda_k\xi\)\).
			$$
			Combining the above inequality and \eqref{diff yk u} yields
			$$
			0\le Y_k-U\(\hat{X}_k,B\)\le (\lambda_0-\lambda_k)\xi
			\(\underline{X}_B\(\lambda_k\xi\)-\overline{X}_B\(\lambda_0\xi\)\)\id_{A_k}.
			$$
			As
			$
			\E\left[\xi\left(\underline{X}_B(\lambda_k\xi)-\overline{X}_B(\lambda_0\xi)\right)\id_{A_k}\right]=x_0-\theta,
			$
			we know
			$$
			\E\left[\left|Y_k-U\(\hat{X}_k,B\)\right|\right]
			\le(\lambda_0-\lambda_k)(x_0-\theta).
			$$
			This means
			$\lim\limits_{k\to+\infty}\E\[U\(\hat{X}_k,B\)\]=\lim\limits_{k\to+\infty}\E\[Y_k\]
			=\E\[U\(\overline{X}_B\(\lambda_0\xi\),B\)\]+\lambda_0(x_0-\theta)$.
			}
		\end{enumerate}
	\end{enumerate}	
\end{proof}	
{\color{black}
\begin{remark}
	Indeed, Theorem \ref{main} also holds with Assumption \ref{ass_standing} replaced by Assumptions \ref{ass_G} and \ref{asmp:Hn}, but it needs a more loaded expression with trivial details because we need to deal with the case where $\underline{X}_B$ or $\overline{X}_B$ being infinite. Hence, we propose some interesting examples instead.
\end{remark}

\begin{example}[Demonstrating \eqref{eq:main}]\label{eg:inequality}%
	Suppose that $\xi$ is uniformly distributed on $\[1,2\]$, and we consider a univariate utility
	$
	u(x)=(x+1)\id_{(1,+\infty)}(x)+2x\id_{[0,1]}(x). 
	$
	Then we can compute
	$$
	\underline{X}(y)=\left\{
	\begin{aligned}
		&0,&y>2,\\
		&1,&1\le y< 2,\\
		&+\infty,&0\le y<1.
	\end{aligned}\right.~\text{ and }~
	g(\lambda)=\left\{
	\begin{aligned}
		&0,&\lambda>2,\\
		&\frac{4-\lambda^2}{2\lambda^2},&1\le \lambda\le 2,\\
		&+\infty,&0\le\lambda<1.
	\end{aligned}\right..
	$$
	Hence $\lambda_0=1$, $\overline{X}(\lambda_0\xi)=1~\text{a.s.}$,  $\theta\triangleq\E\[\xi\overline{X}(\lambda_0\xi)\]=\frac{3}{2}$ and $\E\[u\(\overline{X}\(\lambda_0\xi\)\)\]=2$. For $x_0>\theta$, using \eqref{eq:main} or \eqref{eq:univarite_ineq},
	\begin{equation}\label{eg:equity}
		\sup\limits_{X: \mathbb{E}[\xi X]\le x_0\atop u(X)>-\infty}\mathbb{E}[u(X)]=\E\[u\(\overline{X}(\lambda_0\xi)\)\]+\lambda_0(x_0-\theta)=x_0+\frac{1}{2}.
	\end{equation}
	For $\epsilon>0$, define
	$
	X_{\epsilon}=\id_{\(1+\epsilon,2\]}(\xi)+\(1+\frac{2x_0-3}{\epsilon^2+2\epsilon}\)\id_{\[1,1+\epsilon\]}(\xi)\ge0.
	$
	Then
	$
	\E\[\xi X_{\epsilon}\]=
	x_0,
	$
	and
	$
	\E\[u(X_{\epsilon})\]=2(1-\epsilon)+(2+\frac{2x_0-3}{\epsilon^2+2\epsilon})\epsilon=2+\frac{2x_0-3}{2+\epsilon}.
	$
	Letting $\epsilon\to0$, we know that \eqref{eg:equity} takes ``=''.
\end{example}
In this example, if $\essinf~\xi=0$ , then $g(\lambda)=+\infty$ for all $\lambda\ge0$, and the problem becomes infinite. 
}
\section{Proofs in Sections \ref{section_condition}-\ref{section_ex2}}

\begin{proof}[Proof of Proposition \ref{prop:finite}]
	It is similar to \eqref{eq:utineq}-\eqref{eq:ut_estimate2} in the proof of Proposition \ref{suff cond}.
\end{proof}
\begin{proof}[Proof of Theorem \ref{prop:suff2}.]
	The proof of (i) is the same as that of Theorem \ref{main}(1).
	For (ii), if $x_0\in\(g(+\infty),g(\lambda_0)\)$(may be empty) or $x_0=g(\lambda_0)$, using Theorem \ref{main} we find an optimal solution $X^*\in\mathcal{X}_B(\lambda\xi)$ for some $\lambda\ge\lambda_0$. Hence $X^*\ge\underline{X}_B\(\lambda\xi\)$ and $\E\[U(X^*,B)\]\ge J(\lambda)=+\infty$.
	If $x_0>g(\lambda_0)$, then as $U(x,b)$ is nondecreasing in $x$, the optimal value of Problem \eqref{problem1} is nondecreasing in the initial value.  Problem \eqref{problem1} with an initial value $x_0>g(\lambda_0)$ is also infinite.
	
	For (iii), as $\lambda_0<+\infty$ and $\lambda_1<+\infty$, there exists $\lambda\in(0,+\infty)$ such that $g(\lambda)<+\infty$ and $J(\lambda)<+\infty$. Then $g(\lambda)$ and $J(\lambda)$ are all finite. Using Proposition \ref{prop:finite}, we know that Problem \eqref{problem1} is finite.
	If $\lambda_1>\lambda_0$, then we define $\hat{x}=g\(\frac{\lambda_0+\lambda_1}{2}\)$. Using Theorem \ref{main}, we know that Problem \eqref{problem1} with initial value $\hat{x}$ admits an optimal solution $X^*=\underline{X}_B\(\frac{\lambda_0+\lambda_1}{2}\xi\)$, and the optimal value equals $J\(\frac{\lambda_0+\lambda_1}{2}\)=+\infty$, which contradicts to Proposition \ref{prop:finite}. Thus, $\lambda_1\le\lambda_0$.
	If $\lambda_1<\lambda_0$, then we take $\lambda_1 < \lambda < \lambda_0 < \lambda'$ and write
	$$
	U\left(\underline{X}_B\left(\lambda\xi\right)\right) - \lambda \xi \underline{X}_B\left(\lambda\xi\right)\ge
	U\left(\underline{X}_B\left(\lambda'\xi\right)\right) - \lambda \xi \underline{X}_B\left(\lambda'\xi\right).
	$$
	Taking expectation we obtain
	$J(\lambda) - \lambda g(\lambda)\ge J(\lambda') - \lambda g(\lambda')$, which is a contradiction as $J(\lambda), J(\lambda')$ and $g(\lambda')$ are all finite while $g(\lambda) = +\infty$.
\end{proof}

\begin{proof}[Proof of Theorem \ref{thm_known}]
	It is a direct corollary of Theorems \ref{thm_concave}-\ref{general thm}: one can simply regard $B$ as a constant in Theorems \ref{thm_concave}-\ref{general thm} (removing all the $``B"$ appeared in the proof). 
\end{proof}
\begin{proof}[Proof of Theorem \ref{thm ex}.]
	The expression of $\underline{X}_b^{\mu}(y)$ can be derived after some routine but trivial computation. We proceed to prove that Problem \eqref{remod prob} admits a unique finite optimal solution.
	To apply Theorem \ref{main}, we need the condition that Problem \eqref{remod prob} is finite. Indeed, if we take $u_1(b)=\mu$, $u_2(b)=k$, $K(b)=0$, $\delta=p$, $\theta(b)=b_1$, $\gamma(b)=\mu \id_{\{b_1\ge b_2\}}$, then Proposition \ref{suff cond} indicates that Problem \eqref{remod prob} is finite, and the Case 1 in Section \ref{section_finite} holds. Based on Theorem \ref{main}, finite optimal solution exists for every $x_0>g(+\infty)=0$ (noting that $\underline{x}_b=0$).

	From the expression of $\underline{X}^{\mu}_b(y)$, we know 
	\begin{equation}\label{discont set}
		D^{\mu}_b:=\{y:\underline{X}^{\mu}_b(y)\not=\overline{X}^{\mu}_b(y)\}=\left\{
		\begin{array}{lc}
			\{y_1(b),y_2^\mu(b)\},&  y_1(b)\le y_2^\mu(b)\\
			\{y_3^\mu(b)\},& y_1(b)>  y_2^\mu(b)
		\end{array}\right..
	\end{equation}
	As such, in Plans I-III, as $B$ only takes value in a finite set, and $\xi$ is continuously distributed, we have for every $\lambda\ge0$:
	$	\mathbb{P}[\underline{X}^{\mu}_B(\lambda\xi)\not=\overline{X}^{\mu}_B(\lambda\xi)]\le \mathbb{P}[\lambda\xi=y_1(B)]+
	\mathbb{P}[\lambda\xi=y_2^{\mu}(B)]+
	\mathbb{P}[\lambda\xi=y_3^{\mu}(B)]=0.
	$
	Therefore, using Theorem \ref{main}(2), we know that $g(\lambda)=\mathbb{E}[\xi\underline{X}_B(\lambda\xi)]$ is continuous, and Problem \eqref{remod prob} admits a unique optimal solution $X^{\mu}=\underline{X}^{\mu}_B(\lambda(\mu)\xi)$.
\end{proof}




%


%
%
%


\begin{thebibliography}{99}
	\bibliographystyle{informs2014}
	\small
	
	\bibitem[Basak, Pavlova and Shapiro(2007)]{BPS2007}Basak, S., Pavlova, A., \& Shapiro, A. (2007). Optimal asset allocation and risk shifting in money management. \textit{Review of Financial Studies, 20}, 1583-1621. 
	
	\bibitem[Basak, Shapiro and Tepl\'{a}(2006)]{BST2006} Basak, S., Shapiro, A., \& Tepl\'{a}, L. (2006). Risk management with benchmarking. \textit{Management Science, 52}, 542-557.
	
	
	\bibitem[Berkelaar, Kouwenberg and Post(2004)]{BKP2004}Berkelaar, A. B., Kouwenberg, R., \& Post, G. T. (2004). Optimal portfolio choice under loss aversion. \textit{Review of Economics and Statistics, 86}, 973-987.
	
	\bibitem[Bernard, De Staelen and Vanduffel(2019)]{BSV2019} Bernard, C., De Staelen, R. H., \& Vanduffel, S. (2019). Optimal portfolio choice with benchmarks. \textit{Journal of the Operational Research Society, 70}, 1600-1621.
	
	\bibitem[Bernard, Vanduffel and Ye(2018)]{BVY2018} Bernard, C., Vanduffel, S., \& Ye, J. (2018). Optimal portfolio under state-dependent expected utility. \textit{International Journal of Theoretical and Applied Finance, 21}, 1850013.
	
	\bibitem[Bernard, Moraux, R\"{u}schendorf and Vanduffel(2015)]{BMRV2015} Bernard, C., Moraux, F., R\"{u}schendorf, L., \& Vanduffel, S. (2015).  Optimal payoffs under
	state-dependent preferences. \textit{ Quantitative Finance, 15}, 1157-1173.
	
	
	
	\bibitem[Bichuch and Sturm(2014)]{BS2014}Bichuch, M., \& Sturm, S. (2014). Portfolio optimization under convex incentive schemes. \textit{Finance and Stochastics, 18}, 873-915.
	
	\bibitem[Binger and Hoffman(1998)]{BH1998}	Binger, B. R., \& Hoffman, E. (1998). \textit{Microeconomics with Calculus}. Addison-Wesley.
	
	
	\bibitem[Boulier, Huang and Taillard(2001)]{BHT2001} Boulier, J. F., Huang, S., \& Taillard G. (2001). Optimal management under stochastic interest rates: The case of a protected defined contribution pension fund. \textit{Insurance: Mathematics and Economics, 28}, 173-189.
	
	
	\bibitem[Cairns, Blake and Dowd(2006)]{CBD2006} Cairns, A. J. G., Blake, D., \& Dowd, K. (2006). Stochastic lifestyling: Optimal dynamic asset allocation for defined contribution pension plans. \textit{Journal of Economic Dynamics and Control,  30}, 843-877.
	
	\bibitem[Carpenter(2000)]{C2000} Carpenter, J. N. (2000). Does option compensation increase managerial risk appetite? \textit{Journal of Finance, 55}, 2311-2331.
	
	\bibitem[Chen and Hieber(2016)]{CH2016} Chen, A., \& Hieber. P. (2016). Optimal asset allocation in life insurance: The impact of regulation. \textit{Astin Bulletin, 46}, 605-626.
	
	
	\bibitem[Dong and Zheng(2020)]{DZ2020} Dong, Y., \& Zheng, H. (2020). Optimal investment with S-shaped utility and trading and Value at
	Risk constraints: An application to defined contribution pension plan. \textit{European Journal of Operational Research, 281}, 341-356.
	
	\bibitem[F\"ollmer and Schied(2016)]{FS16} F\"ollmer, H., \& Schied, A. (2016). \emph{Stochastic Finance. An Introduction in Discrete Time}.  {Walter de Gruyter, Berlin}, Fourth Edition.
	
	\bibitem[Fryszkowski(2005)]{F2005} Fryszkowski, A. (2005). \textit{Fixed Point Theory for Decomposable Sets (Topological Fixed Point Theory and Its Applications)}. New York: Springer.
	
	
	
	\bibitem[He and Kou(2018)]{HK2018} He, X., \& Kou, S. (2018). Profit sharing in hedge funds. \textit{Mathematical Finance, 28}, 50-81.
	
	\bibitem[He and Zhou(2011)]{HZ2011} He, X., \& Zhou, X. (2011). Portfolio choice under cumulative prospect theory: An analytical treatment. \textit{Management Science, 57}, 315-331.
	
	\bibitem[Hodder and Jackwerth(2007)]{HJ2007} Hodder, J. E., \& Jackwerth, J. C. (2007). Incentive contracts and hedge fund management. \textit{Journal of Financial and Quantitative Analysis, 2}, 811-826.
	
	\bibitem[Jin and Zhou(2008)]{JZ2008} Jin, H., \& Zhou, X. (2008). Behavioral portfolio selection in continuous time. \textit{Mathematical Finance, 18}, 385-426.
	
	\bibitem[Jin, Xu and Zhou(2008)]{JXZ2008} Jin, H., Xu, Z., \& Zhou, X. (2008). A Convex stochastic optimization problem arising from portfolio selection. \textit{Mathematical Finance, 18}, 171-183.
	
	\bibitem[Karatzas, Lehoczky and Shreve(1987)]{KLS1987} Karatzas, I., Lehoczky, J. P., \& Shreve, S. E. (1987). Optimal portfolio and consumption decisions for a ``small investor" on a finite horizon. \textit{SIAM Journal on Control and Optimization, 25}, 1557-1586.
	
	
	
	\bibitem[Karatzas and Shreve(1998)]{KS1998}	Karatzas, I.,\& Shreve, S. E. (1998). \textit{Methods of Mathematical Finance}. Springer, New York.
	
	\bibitem[Kramkov and Schachermayer(1999)]{KS1999} Kramkov, D., \& Schachermayer, W. (1999). The asymptotic elasticity of utility functions and optimal investment in incomplete markets. \textit{Annals of Applied Probability, 9}, 904-950.
	
	\bibitem[Kahneman and Tversky(1979)]{KT1979} Kahneman, D., \& Tversky, A. (1979). Prospect Theory: an analysis of decision under risk. \textit{Econometrica, 47}, 263-291.
	
	\bibitem[Kouwenberg and Ziemba(2007)]{KZ2007} Kouwenberg, R., \& Ziemba, W. T. (2007). Incentives and risk taking in hedge fund. \textit{Journal of Banking and Finance, 31}, 3291-3310.
	
	\bibitem[K\H{o}szegi and Rabin(2007)]{KR2007} K\H{o}szegi, B., \& Rabin, M. (2007). Reference-dependent risk attitudes. \textit{American Economic Review, 97}, 1047-1073.
	
	
	
	\bibitem[Liang and Liu(2020)]{LL2020} Liang, Z., \& Liu, Y. (2020).  A classification approach to general S-shaped utility optimization with principals' constraints.
	\textit{SIAM Journal on Control and Optimization, 58}, 3734-3762.
	
	\bibitem[Merton(1969)]{M1969} Merton, R. C. (1969). Lifetime portfolio selection under uncertainty: The continuous-time case. \textit{Review of Economics and Statistics, 51}, 247-257.
	
	\bibitem[Nguyen and Stadje(2020)]{NS2020} Nguyen, T., \& Stadje, M. (2020). Nonconcave optimal investment with value-at-risk constraint: An application to life insurance contracts. \textit{SIAM Journal on Control and Optimization, 58}, 895-936.


        \bibitem[Peng, Wei and Xu(2023)]{PWX2023} Peng, J., Wei, P., \& Xu, Z. (2023). Relative Growth Rate Optimization Under Behavioral Criterion. \textit{SIAM Journal on Financial Mathematics, 14}, 1140-1174.

 
	\bibitem[Pliska(1986)]{P1986} Pliska, S. (1986). A stochastic calculus model of continuous trading: Optimal portfolios. \textit{Mathematics of Operations Research, 11}, 371-382.
	
	\bibitem[Reichlin(2013)]{R2013} Reichlin, C. (2013). Utility maximization with a given pricing measure when the utility is not necessarily concave. \textit{Mathematics and Financial Economics, 7}, 531-556.
	
	\bibitem[Rieger(2012)]{R2012} Rieger, M. O. (2012). Optimal financial investments for non-concave utility functions. \textit{Economics Letters, 114}, 239-240.
	
	
	\bibitem[Rockafellar(1970)]{R1970} Rockafellar, R. T. (1970). \textit{Convex Analysis}. Princeton University Press, Princeton.
	
	\bibitem[Sierpi\'{n}ski(1922)]{S1922} Sierpi\'{n}ski, W. (1922). Sur les fonctions d'ensemble additives et continues. \textit{Fundamenta Mathematicae, 3}, 240-246.
	
	\bibitem[Sugden(2003)]{S2003} Sugden, R. (2003). Reference-dependent subjective expected utility. \textit{Journal of Economic Theory, 111}, 172-191.
	
	\bibitem[{Tversky and Kahneman(1992)}]{TK1992}
	Tversky, A., \& Kahneman, D. (1992).
	Advances in prospect theory: cumulative representation of uncertainty.
	\textit{Journal of Risk and Uncertainty, 5}, 297-323.
	
	\bibitem[Von Neumann and Morgenstern(1953)]{NM1953} Von Neumann, J., \& Morgenstern, O. (1953). \textit{Theory of Games and Economic Behavior}. Princeton University Press, Princeton.
	
	\bibitem[Wei(2018)]{W2018} Wei, P. (2018). Risk management with weighted VaR. \textit{Mathematical Finance, 28}, 1020-1060.

        
	
	\bibitem[Xu(2016)]{X2016} Xu, Z. (2016). A note on the quantile formulation. \textit{Mathematical Finance, 26}, 589-601.
\end{thebibliography}
\end{document}